\documentclass[12pt]{article}
\usepackage{setspace}
\usepackage[letterpaper, total={6.5in, 9in}]{geometry}

\usepackage[english]{babel}
\usepackage[utf8]{inputenc} 
\usepackage[T1]{fontenc}    
\usepackage[colorlinks=true,linkcolor=blue,citecolor=green]{hyperref}       
\usepackage{url}            
\usepackage{booktabs}       
\usepackage{amsfonts}       
\usepackage{nicefrac}       
\usepackage{microtype}      
\usepackage{geometry}
\usepackage[ruled,linesnumbered]{algorithm2e}
\usepackage{amssymb}
\usepackage{amsfonts}
\usepackage{amsmath}

\numberwithin{equation}{section}
\numberwithin{equation}{section}
\usepackage{amsthm}
\usepackage{latexsym}
\usepackage{epsfig}
\usepackage{subcaption}
\newtheorem{theorem}{Theorem}[section]
\newtheorem{corollary}[theorem]{Corollary}
\newtheorem{lemma}[theorem]{Lemma}
\newtheorem{fact}{Fact}[section]
\newtheorem{proposition}[theorem]{Proposition}

\newtheorem{definition}{Definition}[section]

\newtheorem{remark}[theorem]{Remark}
\newtheorem{assumption}{Assumption}

\usepackage[english]{datetime2}
\DTMlangsetup{en-US}

\usepackage{adjustbox}
\usepackage{array}
\usepackage{wrapfig}
\usepackage{enumitem}
\usepackage{multirow}
\usepackage{diagbox}

\newcommand{\lin}[1]{\vec{#1}}
\newcommand{\linV}{\vec{V}}
\newcommand{\linVp}{\vec{V}_p}
\newcommand{\linVd}{\vec{V}_d}

\newcommand{\dist}{\operatorname{Dist}}

\newcommand{\gap}{\operatorname{Gap}}

\newcommand{\ep}{\operatorname{\textsf{E}}}

\newcommand{\eobj}{\operatorname{E_{obj}}}
\newcommand{\efeas}{\operatorname{E_{feas}}}

\newcommand{\condL}{\mathcal{L}}
\newcommand{\condC}{\mathcal{C}}
\newcommand{\err}{\operatorname{MErr}_\eps}
\newcommand{\width}{\operatorname{Width}}
\newcommand{\bestrescaling}{rPDHG-Ideal}

\newcommand{\adarescaling}{rPDHG-AHR}
\newcommand{\cpipm}{CP-CGM}
\newcommand{\ideal}{rPDHG-Ideal}

\newcommand{\eps}{\varepsilon}

\newcommand{\Diam}{\mathrm{Diam}}

\newcommand{\calC}{{\cal C}}

\newcommand{\calF}{{\cal F}}

\newcommand{\calL}{{\cal L}}

\newcommand{\calS}{{\cal S}}

\newcommand{\calW}{{\cal W}}
\newcommand{\calX}{{\cal X}}
\newcommand{\calY}{{\cal Y}}
\newcommand{\calZ}{{\cal Z}}

\title{The Role of Level-Set Geometry on the Performance of PDHG for Conic Linear Optimization}
\author{Zikai Xiong\thanks{MIT Operations Research Center, 77 Massachusetts Avenue, Cambridge, MA 02139, USA. 
		\href{mailto:zikai@mit.edu}{zikai@mit.edu}.  Research supported by AFOSR Grant No. FA9550-22-1-0356.
	}  
	\and Robert M. Freund\thanks{MIT Sloan School of Management, 77 Massachusetts Avenue, Cambridge, MA 02139, USA. 
		\href{mailto:rfreund@mit.edu}{rfreund@mit.edu}. Research supported by AFOSR Grant No. FA9550-22-1-0356.}}

\date{\today}

\begin{document}

\maketitle

\begin{abstract}
	We consider solving huge-scale instances of (convex) conic linear optimization problems, at the scale where matrix-factorization-free methods are attractive or necessary. The restarted primal-dual hybrid gradient method (rPDHG) -- with heuristic enhancements and GPU implementation -- has been very successful in solving huge-scale linear programming (LP) problems; however its application to more general conic convex optimization problems is not so well-studied. We analyze the theoretical and practical performance of rPDHG for general (convex) conic linear optimization, and LP as a special case thereof. We show a relationship between the geometry of the primal-dual (sub-)level sets $\mathcal{W}_\varepsilon$ and the convergence rate of rPDHG. Specifically, we prove a bound on the convergence rate of rPDHG that improves when there is a primal-dual (sub-)level set $\mathcal{W}_\varepsilon$ for which (i) $\mathcal{W}_\varepsilon$ is close to the optimal solution set (in Hausdorff distance), and (ii) the ratio of the diameter to the ``conic radius'' of $\mathcal{W}_\varepsilon$ is small. And in the special case of LP problems, the performance of rPDHG is bounded only by this ratio applied to the (sub-)level set corresponding to the best non-optimal extreme point. Depending on the problem instance, this ratio can take on extreme values and can result in poor performance of rPDHG both in theory and in practice. To address this issue, we show how central-path-based linear transformations -- including conic rescaling -- can markedly enhance the convergence rate of rPDHG. Furthermore, we present computational results that demonstrate how such rescalings can accelerate convergence to high-accuracy solutions, and lead to more efficient methods for huge-scale linear optimization problems.
\end{abstract}
 


\section{Introduction}\label{sec:intro}
In this paper, we focus on the following general conic linear program (CLP):
\begin{equation}\tag{P}\label{pro: general primal clp}
	\min_{x\in\mathbb{R}^n}  \ c^\top x \quad 	\text{s.t.} \ Ax = b, \  x \in K_p \ ,
\end{equation}where $K_p \subseteq \mathbb{R}^n$ is a closed convex cone, $A\in \mathbb{R}^{m\times n}$ is the constraint matrix,  $b \in \mathbb{R}^m$ is the right-hand side vector, and $c \in \mathbb{R}^n$ is the objective vector.
The family of CLPs has emerged since the 1990s as a fundamental problem class in convex optimization.  CLP includes standard linear optimization problems (LPs) as a subclass in which $K_p$ is the nonnegative orthant $\mathbb{R}^n_+$, and the importance of LP cannot be overstated in application domains as varied as manufacturing \cite{bowman1956production,hanssmann1960linear}, transportation \cite{charnes1954stepping}, economics \cite{greene2003econometric}, computer science \cite{cormen2022introduction}, and medicine \cite{wagner2004large} among many others \cite{dantzig2002linear}. Second-order cone optimization problems (SOCPs) are another subclass of CLP where the fundamental cone is a cross-product of second-order cones $\mathbb{K}^{d+1}_{\textsc{soc}}$, with significant applications in finance \cite{levy1970international,markowitz1950theories}, statistics \cite{tibshirani1996regression}, and others \cite{lobo1998applications}. The very broad class of semidefinite optimization problems (SDPs), where the underlying cone is a cross-product of semidefinite cone $\mathbb{S}^{d \times d}_+$, has also received significant attention, though more for its overarching breadth of potential applications than for practical industrial usage \cite{blekherman2012semidefinite,wolkowicz2012handbook,alizadeh1995interior}.

Algorithms for small and medium-size CLP instances have been extensively researched both theoretically and practically. Virtually all practical LP algorithms since the late 1980s have been based on simplex/pivoting methods and/or interior-point methods (IPMs). Pivoting methods were extended to quadratic optimization problems, and IPMs were extended to SOCPs, SDPs, and others. Today these classic algorithms form the foundation of virtually all modern solvers and have had a profound impact on optimization quite broadly. 

But the success of these two classic algorithms is premised on being able to repeatedly solve linear equation systems at each iteration, whose operations grow superlinearly with respect to the size of the data (measured with the dimensions $m$ and/or $n$ of the operators or the number of nonzero entries $\textsf{nnz}$ in the data $A$, $b$, $c$), which renders the methods impractical when the problem size is very large-scale. Furthermore, the matrix factorizations are not well suited to either parallel or distributed computation. In contrast, first-order methods (FOMs) are emerging as an alternative approach for solving large-scale CLPs, since they require no or only very few matrix factorizations.  Instead, the primary computational cost of FOMs lies in computing matrix-vector products that are needed to compute (or estimate) gradients. As such,  FOMs are inherently more suitable for exploiting data sparsity, and furthermore are well suited for parallel and/or distributed computer architecture and modern graphics processing units (GPUs).

One of the most successful first-order methods for LP is the restarted primal-dual hybrid gradient method (rPDHG) \cite{applegate2023faster}, which directly tackles the saddlepoint formulation of LP, and automatically detects the infeasibility \cite{applegate2024infeasibility}. An advanced implementation of rPDHG is the solver PDLP \cite{applegate2021practical}, which has outperformed the commercial solver Gurobi on large-scale LP problems \cite{applegate2021practical}. Recent GPU implementations of PDLP have further enhanced its performance to the point that its performance has surpassed classic algorithms (simplex methods and IPMs) on a significant number of problem instances \cite{lu2023cupdlp,lu2023cupdlp-c}. And most recently, rPDHG has been embedded in the state-of-art commercial solvers COPT 7.1 and Xpress 9.4 as a base algorithm for LP \cite{coptgithub,xpressnwes} alongside simplex methods and IPMs. Indeed, many problems that used to be considered too large-scale to be solved are now solvable via rPDHG; for instance, a distributed version of PDLP has been used to solve practical LP instances with $\textsf{nnz} = 9.2\times 10^{10}$, which is a scale far beyond the capabilities of traditional methods \cite{pdlpnews}.  Another example is a representative large-scale benchmark instance called \textsc{zib03}. This instance solved in 16.5 hours in 2021 \cite{koch2022progress}, and now solves in 15 minutes using PDLP with GPU architecture \cite{lu2023cupdlp-c}. 

Despite the promising performance of rPDHG on many LP instances, the method can also perform poorly on certain instances -- even very small instances. Figure \ref{fig:number_of_iteration_count} illustrates the \begin{wrapfigure}{hr}{0.35\textwidth}
	\centering
	\includegraphics[width=1\linewidth]{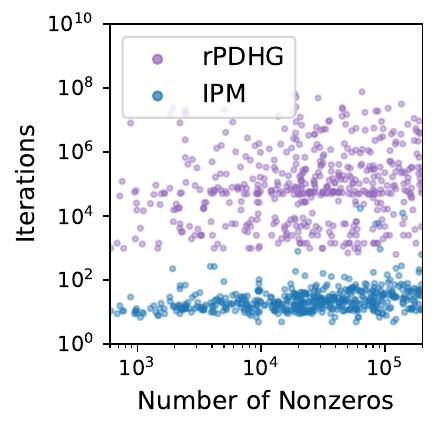}
	\caption{\small The number of iterations of rPDHG and IPM for solving LP instances with different numbers of nonzeros in the constraint matrix.}\label{fig:number_of_iteration_count}
\end{wrapfigure} distribution of the number of iterations required by a standard rPDHG implementation and a standard IPM implementation for a set of LP instances taken from the MIPLIB 2017 dataset \cite{gleixner2021miplib}.  While rPDHG, as a first-order method,  enjoys a much lower per-iteration cost than the IPM, it usually requires orders-of-magnitude more iterations than an IPM, which can offset the advantages of its per-iteration cost.
Moreover, the number of iterations required by rPDHG can vary significantly across different instances, even for very small problems with similar size (as measured by the number of nonzeros $\textsf{nnz}$).  This begs the question of what instance-specific conditions cause some LP instances to be more difficult to solve? Answers to this question lead to the study of traditional as well as novel condition measures for LP and more generally for CLP.

There has been some recent research focused on different condition numbers to analyze the complexity of rPDHG on LP problems. \cite{applegate2023faster} shows that the linear convergence of rPDHG relies on the sharpness of a ``normalized duality gap'' for LP and characterizes the sharpness using a global Hoffman constant of the KKT system.  However, the Hoffman constant is usually overly conservative, and is difficult to analyze, compute, or improve.  \cite{xiong2023computational} connects the sharpness constant to two natural and intuitive condition measures of LP, namely the ``limiting error ratio'' near the optimal solution set, and the LP sharpness (the sharpness of the LP instance). These condition measures are more intrinsically related to stability of LP under perturbation. Both of the above papers rely on the sharpness of LP. For CLP instances with zero sharpness -- or for LP instances with exponentially small positive sharpness -- the practical performance of rPDHG is usually better than the theory indicates. This discrepancy suggests that other condition numbers (in addition to sharpness) play a role in the performance of rPDHG. The identification of these condition numbers -- either theoretically, practically, or both -- can result in schemes to improve these condition numbers so that with these enhancements, rPDHG can achieve more stable and predictable performance and thus enable solutions of ever-more-challenging CLP instances.

Based on the above discussion, this paper seeks to address the following questions. What are the condition numbers (beyond sharpness) that impact the performance of rPDHG on general CLP problems (in theory and in practice)? Is it possible to improve these condition numbers by applying suitable linear transformations of the problem instance? And can the insights gained from answers to these questions be utilized to develop practical improvements in rPDHG?

\subsection{Outline/Overview of Results}

In Section \ref{sec:CLP} we revisit conic linear optimization problems (CLP) and do some elementary transformations to focus on the space of cone variables for both the primal and dual problems.  We revisit the saddlepoint formulation and the primal-dual hybrid gradient method (PDHG) of \cite{chambolle2011first}, and we revisit fundamental convergence results of PDHG.

In Section \ref{sec:complexity_clp} we introduce three condition numbers related to the geometry of the primal-dual sublevel sets of \eqref{pro: general primal clp}, and we use these condition numbers to establish new computational guarantees for rPDHG. (Here the sublevel set (or level set for short) is the set of feasible primal-dual solution cone-variable pairs with a duality gap at most a given threshold.) These condition numbers, which are purely geometric (and are defined using the Euclidean norm), are the diameter $D_\delta$ of the $\delta$-sublevel set, the conic radius $r_\delta$ of the $\delta$-sublevel set, and the Hausdorff distance $d^H_\delta$ between the $\delta$-sublevel set and the optimal solution set. Based on these condition numbers, we derive a new convergence guarantee of rPDHG that applies to general CLP problems and does not rely on the sharpness of the problem instance.

In Section \ref{sec:complexity_lp} we present a new global linear convergence guarantee for rPDHG for LP instances, whose rate bound involves the ratio $D_\delta/r_\delta$ (ratio of the diameter to the conic radius) of the sublevel set corresponding to the best non-optimal extreme point (the ``second-best'' solution point). The slow convergence of rPDHG on certain LP instances can be attributed to the extreme values of this ratio.

In Section \ref{sec:geometric_enhancements} we show how to improve the three condition measures $D_\delta$, $r_\delta$, and $d^H_\delta$ through a rescaling transformation of the primal-dual feasible regions.  We show that such a rescaling, if based on the Hessian of the barrier function of a point on the central path, results in a rescaled problem with guaranteed bounds on the (newly-rescaled) condition numbers $\bar D_\delta$, $\bar r_\delta$, and $\bar d^H_\delta$. If such a rescaling is employed, the overall complexity of rPDHG can be significantly improved. We also present computational experiments that confirm the role of the three condition numbers in the overall performance -- theoretically and practically -- and confirm the effectiveness of our central-path Hessian rescaling methodology.

In Section \ref{sec:exp} we develop and test heuristics to adaptively compute a good central-path rescaling that efficiently trades off the extra computation time to determine the rescaling with the computational savings using rPDHG on the rescaled problem. We compare our methods with a standard implementation of rPDHG and with a standard IPM on LP problems from the MIPLIB 2017 dataset.

\subsection{Other related work for large-scale CLP}
In addition to the research papers discussed earlier, several other works have also analyzed the performance of rPDHG and its variants for solving LP problems. \cite{hinder2023worst} presents a worst-case complexity of rPDHG on totally-unimodular LP instances, that does not rely on any condition numbers.  \cite{lu2023geometry} develops a two-phase theory of the behavior of PDHG without restarts, where the initial sublinear convergence phase is followed by a linear convergence phase that is characterized by the Hoffman constant of a reduced system. and \cite{lu2022infimal} shows that the last iterate of PDHG without restarts also has a linear convergence rate but it is slower than that of rPDHG.

Several other first-order methods have been studied for LP and general CLP problems.  \cite{lin2021admm} proposes an ADMM-based interior-point method that leverages the framework of the homogeneous self-dual interior-point method and employs ADMM to solve the inner log-barrier problems. Further enhancements and extension to CLP problems were subsequently developed by \cite{deng2024enhanced}. \cite{o2016conic,o2021operator} use ADMM to directly solve the homogeneous self-dual formulation for the general CLP problem, and \cite{lu2023practicalqp,huang2024restarted} study accelerated variants of rPDHG for solving a convex quadratic programming problem, which itself is special case of SOCP and hence also of CLP.

\subsection{Notation}
Throughout this paper, we use the following notation for the most common cone examples: $\mathbb{R}^n_+$ denotes the nonnegative orthant,  $\mathbb{S}^{d\times d}_+$ denotes the semidefinite cone, which is the set of positive semidefinite symmetric matrices in $\mathbb{R}^{d\times d}$, and $\mathbb{K}^{d+1}_{\mathsf{soc}}$ denotes the second-order cone $\{(x,t):x\in \mathbb{R}^d, t\in \mathbb{R}, \|x\|_2 \le t\}$.

For a matrix $A\in\mathbb{R}^{m\times n}$, $\operatorname{Null}(A):=\{x\in\mathbb{R}^n:Ax = 0\}$ denotes the null space of $A$ and $\operatorname{Im}(A) :=\{Ax:x\in\mathbb{R}^n\}$ denotes the image of $A$. For any set $\calX\subset \mathbb{R}^n$,  $P_\calX: \mathbb{R}^n \to \mathbb{R}^n$ denotes the Euclidean projection onto $\calX$, namely, $P_\calX(x) := \arg\min_{\hat{x}\in \calX} \|x - \hat{x}\|$.
If not specified via definition, $\|\cdot\|$ in this paper denotes the Euclidean norm. Let $B(x,r)$ denote the ball centered at $x$ with radius $r$. For any $M \in \mathbb{S}_{+}^{n\times n}$, $\|\cdot\|_M$ denotes the inner product ``norm'' induced by $M$, namely, $\|z\|_M :=\sqrt{z^\top Mz}$. (Here we allow $M$ to not be strictly positive definite, in which case $\|z\|_M$ is a semi-norm but not necessarily a norm.) For any $x \in \mathbb{R}^n$ and set $\calX\subset \mathbb{R}^n$, the Euclidean distance between $x$ and $\calX$ is denoted by $\dist(x,\calX):= \min_{\hat{x} \in \calX} \|x-\hat{x}\|$ and the $M$-norm distance between $x$ and $\calX$ is denoted by $\dist_M(x,\calX):= \min_{\hat{x} \in \calX} \|x-\hat{x}\|_M$.
For $A \in \mathbb{R}^{n\times n}$ we use $A^\dag$ to denote the Moore-Penrose inverse of $A$, and we use $\sigma_{\max}^+(A)$ and $\sigma_{\min}^+(A)$ to denote the largest and smallest positive singular values of $A$, respectively.
For $x\in\mathbb{R}^n$, we use $x^+$ to denote the positive part of $x$. For any set $\calX \subset \mathbb{R}^n$, $\textsf{int}\calX$ denotes the interior of $\calX$. For any affine subspace $V$, we use $\lin{V}$ to denote the associated linear subspace corresponding to $V$.
For any linear subspace $\lin{S}$ in $\mathbb{R}^n$, we use $\lin{S}^\bot$ to denote the corresponding complementary linear subspace of $\lin{S}$. For any cone $K$, we use $K^*$ to denote the corresponding dual cone of $K$. The width of a cone $\bar{K}$ is defined as:
\begin{equation}\label{eq:width_K}
	\width_{\bar{K}}:=
	\max \left\{\left.\frac{r}{\|x\|} \right\rvert\, B(x, r) \subset \bar{K}\right\} \ .
\end{equation}
(The width of a cone is an intrinsic property of the cone, though it depends on the choice of norm.  Under the Euclidean norm we have $\width_{\mathbb{R}^n_+} = 1/\sqrt{n}$, $\width_{\mathbb{S}_{+}^{d \times d}} = 1/\sqrt{d}$, and $\width_{\mathbb{K}_{\textsf{soc}}^{d+1}} = 1/\sqrt{2}$, see for example \cite{freund1999condition}.)

\section{Preliminaries: Conic Linear Program, and PDHG}\label{sec:CLP}

\subsection{General conic linear program}\label{sec:get_general_clp_dual}
In this paper we consider the general conic linear optimization problem \eqref{pro: general primal clp}.
Because we are interested in the setting of huge-scale instances, we do \underline{not} assume that $A$ has linearly independent rows, since (i) eliminating linear dependence can be computationally expensive for truly huge-scale problems, and (ii) the performance of PDHG is agnostic to the presence of linear dependence in the rows of $A$. We refer to \eqref{pro: general primal clp} as the primal problem. Let $y$ be the multiplier on the equations $Ax=b$ ; then \eqref{pro: general primal clp} is equivalent to the following saddlepoint problem of the Lagrangian $L(x,y)$:
\begin{equation}\tag{PD}\label{pro: general saddlepoint clp}
	\min_{x \in K_p\subseteq\mathbb{R}^n} \ \max_{y\in\mathbb{R}^m} \ L(x,y):= c^\top x  + b^\top y - x^\top A^\top y \ .
\end{equation}
The corresponding dual problem of \eqref{pro: general primal clp} is the problem that switches the minimum and maximum of \eqref{pro: general saddlepoint clp}. The dual problem is also a CLP instance, and can be equivalently written as a problem whose cone is the dual cone of $K_p$, namely $K_d:= K_p^*:=\left\{s \in \mathbb{R}^n:x^\top s \ge 0 \text { for all } x \in K_p\right\}$:
\begin{equation}\tag{D$_{y,s}$}\label{pro: general dual clp}
	\max_{y\in \mathbb{R}^m, s\in\mathbb{R}^n}  \ b^\top y \quad \	\text {s.t.}  \ c-A^\top y  = s, \ s\in K_d \ 	.
\end{equation}
{For any dual feasible solution $(\bar{y},\hat{s})$ of \eqref{pro: general dual clp}, let  $\hat{y}:=(AA^\top)^{\dag}A(c-\hat{s})$ and then $(\hat{y},\hat{s})$ is also dual feasible because $c - A^\top \hat{y} = c - A^\top (AA^\top)^{\dag}A(c-\hat{s}) =  c - A^\top (AA^\top)^{\dag}AA^\top\bar{y} = c - A^\top \bar{y} = \hat{s}$.  Note that if there exists a solution $x_0$ to the linear system $Ax = b$, the corresponding dual objective function values of $(\bar{y},\hat{s})$ and $(\hat{y},\hat{s})$ are equal, because $b^\top \bar{y} =  x_0^\top A^\top \bar{y}=  x_0^\top (c-\hat{s})  = x_0^\top A^\top \hat{y} =   b^\top \hat{y}$. It therefore follows that dual feasible solutions with the same $s$ component have the same objective function values.  Furthermore, for any feasible solution $({y},{s})$ the objective function value can be expressed as a linear function of ${s}$ since
\begin{equation}\label{eq by=-qs}
	b^\top {y} = b^\top (AA^\top)^{\dag} A(c - {s}) = q_0-b^\top(AA^\top)^{\dag} A {s} \ ,
\end{equation}
where $q_0 = b^\top (AA^\top)^{\dag} Ac$.  Let us define $q := A^\top (AA^\top)^{\dag} b$, whereby \eqref{pro: general dual clp} is equivalent to the following (dual) problem defined only on the variable $s$ :
\begin{equation}\tag{D$_s$}\label{pro: general dual clp on s}
	\max_{s\in\mathbb{R}^n}  \ -q^\top  s + q_0 \quad \  \text{s.t.}  \ s \in c + \operatorname{Im}(A^\top), \ s\in K_d \ .
\end{equation}
For any feasible $\hat{s}$ of \eqref{pro: general dual clp on s}, a corresponding feasible $(\hat{y},\hat{s})$ for \eqref{pro: general dual clp} can be recovered by assigning $\hat{y}:= (AA^\top)^{\dag} A(c - \hat{s})$.

A linear programming problem (``LP'') is an instance of \eqref{pro: general primal clp} where $K_p = \mathbb{R}^n_+$, and has the property of \textit{strong duality}, namely the optimal objective function values of the primal problem \eqref{pro: general primal clp} and dual problem \eqref{pro: general dual clp} (or the equivalent \eqref{pro: general dual clp on s}) are identical if both problems are feasible. Strong duality is not always true for the more general CLP problem, but is guaranteed if both the primal and dual problem have feasible solutions in the interiors of their respective cones, which we formally state as follows:
\begin{assumption}\label{assump:striclyfeasible} There exists a primal feasible solution in the interior of $K_p$ and a dual feasible solution in the interior of $K_d$.
\end{assumption}
Under Assumption \ref{assump:striclyfeasible}, both the primal and the dual problems attain their optima, and there is zero duality gap, see \cite{Duffin1956infinite}.  Furthermore, the Karush-Kuhn-Tucker conditions are both necessary and sufficient for optimality, whereby the solution $(x^\star,s^\star)$ is optimal for \eqref{pro: general primal clp} and \eqref{pro: general dual clp on s} if and only if $(x^\star,s^\star)$ is primal-dual feasible and the duality gap is zero.
Here the primal-dual feasible set is as follows:
\begin{equation}\label{def_p_d_feasibleset}
	\calF:= V \cap K, \text{ where }V:= \{(x,s)\in\mathbb{R}^{2n}:Ax = b \text{ and } \exists \ y \in \mathbb{R}^m, \text{ s.t. } A^\top y + s = c\}, \text{ and }K:=K_p\times K_d \ ,
\end{equation}
which is the intersection of the affine subspace $V$ and the cone $K$, both of which are in $\mathbb{R}^{2n}$.
Notice that $V$ is the cross-product of the affine subspaces for the primal and dual solutions:
$$
	V = V_p \times V_d \ , \quad \text{where } V_p := \{x\in\mathbb{R}^n: Ax = b\} , \ V_d := \{ s\in \mathbb{R}^n : \exists \ y \in \mathbb{R}^m, \text{ s.t. } A^\top y + s = c\} \ . $$
The duality gap is defined as:
\begin{equation}\label{def_duality_gap}
	\gap(x,s):= c^\top x + q^\top s  - q_0\ ,
\end{equation}
which is the difference between the primal and dual objective function values.
Let us introduce $w:=(x,s)$ and let $\calW^\star$ be the set of optimal solutions $w^\star =(x^\star,s^\star)$, namely
\begin{equation}\label{def_optimal_W}
	\calW^\star := \calF \cap \{w\in \mathbb{R}^{2n}: \gap(w)  = 0\} \ .
\end{equation}
We use $\calZ^\star$ to denote the set of optimal solutions $(x^\star,y^\star)$, which is equivalently characterized as the set of saddlepoints of \eqref{pro: general saddlepoint clp}.

In light of the linear constraint $Ax = b$, replacing $c$ with any vector in $c + \operatorname{Im}(A^\top)$ does not change $\calW^\star$. Therefore, in certain places in this paper we will presume that $c \in \operatorname{Null}(A)$. This condition can be satisfied by replacing $c \leftarrow c - P_{\operatorname{Im}(A^\top)}(c)$, which yields $c \in P_{\operatorname{Null}(A)}(c)$, resulting in $q_0 = 0$ and $\gap(x,s) = c^\top x + q^\top s $.  Additionally,  when $c \in \operatorname{Null}(A)$, for any feasible $\hat s$,  a corresponding dual feasible $\hat{y}$ can be obtained as $\hat y :=-(AA^\top)^{\dag} A\hat{s}$.

\subsection{PDHG for Conic LP}
The primal-dual hybrid gradient method (PDHG) was introduced in \cite{esser2010general,pock2009algorithm} in the context of solving general convex-concave saddlepoint problems, of which the saddlepoint problem \eqref{pro: general saddlepoint clp} is a class of instances.
Algorithm \ref{alg: one PDHG} describes a single iteration of PDHG for  \eqref{pro: general saddlepoint clp}, which we denote as \textsc{OnePDHG}$(x,y)$, where $\tau$ and $\sigma$ are the primal and dual step-sizes, respectively.
\begin{algorithm}[htbp]
	\SetAlgoLined
	\SetKwProg{Fn}{Function}{}{}
	\Fn{\textsc{OnePDHG}$(x,y)$}{
		$x^{+} \leftarrow P_{K_p}\left(x-\tau\left(c-A^{\top} y\right)\right) $ \;\label{line:update_x}
		$y^{+} \leftarrow y+\sigma\left(b-A\left(2 x^{+}-x\right)\right)$ \;\label{line:update_y}
		return $(x^+,y^+)$ \;}
	\caption{One iteration of PDHG on $(x,y)$ for problem \eqref{pro: general saddlepoint clp}}\label{alg: one PDHG}
\end{algorithm}
Let $z:=(x,y)\in\mathbb{R}^{m+n}$ denote the combined primal/dual variables, and then PDHG generates iterates as follows:
$$
	z^{k+1} \leftarrow \textsc{OnePDHG}(z^k) \ \text{ for }k=0,1,2,\ldots .
$$
It should be noted that the convergence guarantees for PDHG rely on the step-sizes $\tau$ and $\sigma$ being sufficiently small.  In particular, if the following condition is satisfied:
\begin{equation}\label{robsummer}
	M:=	\begin{pmatrix}
		\frac{1}{\tau}I_n & -A^\top             \\
		-A                & \frac{1}{\sigma}I_m
	\end{pmatrix}  \in  \mathbb{S}^{m+n}_+ \ , 
\end{equation}
then PDHG's average iterates will converge to a saddlepoint of the convex-concave problem \cite{chambolle2011first}. The above requirement is equivalently written as:
\begin{equation}\label{eq:general_stepsize}
	\tau > 0, \ \sigma >0, \ \text{ and } \ \tau\sigma \le  \left( \frac{1}{ \sigma_{\max}^+(A) } \right)^2 \ .
\end{equation}
Furthermore, the matrix $M$ defined in \eqref{robsummer} turns out to be particularly useful in analyzing the convergence of PDHG through its induced inner product norm defined by $\| z \|_M := \sqrt{z^\top M z}$, which will be used extensively in the rest of this paper.

The main computational effort in executing \textsc{OnePDHG} is in computing the two matrix-vector products and computing the projection onto $K_p$. In practice, most CLP problems of interest are instances where $K_p$ is a cross-product of standard cones, namely $\mathbb{R}^n_+$, $\mathbb{K}^{d+1}_{\textsc{soc}}$, and $\mathbb{S}^{d \times d}_+$ \cite{nesterov1994interior}. These cones all have well-known projection operators \cite{parikh2014proximal}, among which only projection onto the semidefinite cone may be computationally challenging because it involves a full matrix eigendecomposition. Projection onto $\mathbb{R}^n_+$ is given by $
	P_{\mathbb{R}^n_+}(v):=v^+ $, and projection onto $\mathbb{K}^{d+1}_{\textsc{soc}}$ is given by:
$$
	P_{\mathbb{K}^{d+1}_{\mathsf{soc}}}(v, t) := \left\{\begin{array}{ll}0 & \text{ if }\|v\| \leq-t \\ (v, t) & \text{ if } \|v\| \leq t \\ \tfrac{1+t /\|v\|}{2}\cdot \left(v,\|v\|\right) & \text{ if } \|v\| \geq|t|\end{array} \right.\ .
$$
Furthermore, if $K_p$ is the cross-product of several cones, then each of these projections can be carried out independently.

\subsection{Normalized duality gap for the saddlepoint problem \eqref{pro: general saddlepoint clp}}

To evaluate the quality of a candidate solution $z=(x,y)$, \cite{applegate2023faster} defined the ``normalized duality gap'' in the context of the saddlepoint formulation of LP problems. Here we simply extend this definition to CLP problems and we show that the normalized duality gap provides upper bounds on the residuals of the optimality conditions of CLP.
\begin{definition}[Normalized duality gap]
	For any $z = (x,y)\in K_p \times \mathbb{R}^m$ and $r > 0$, define
	$$
		B(r;z) := \left\{\hat{z} := (\hat{x},\hat{y}):  \hat{x}\in K_p \text{ and } \|\hat{z} -z\|_M \le r  \right\} \ .
	$$
	The normalized duality gap of the saddlepoint problem \eqref{pro: general saddlepoint clp} is then defined as
	\begin{equation}\label{sunny}
		\rho(r;z) := \frac{1}{r}\sup_{  \hat{z} \in B(r;z) }  \big[ L(x,\hat{y}) - L(\hat{x},y) \big] \ .
	\end{equation}
\end{definition}
\noindent The normalized duality gap defined in \cite{applegate2023faster} is just a special case of the above definition when $K_p = \mathbb{R}^n_+$. Lemma \ref{lm: convergence of PHDG without restart} below shows that the normalized duality gap yields upper bounds on the distances to the affine set $V$ and to the cone $K$ of the primal-dual cone variable pair $w=(x,s)$, and also bounds the duality gap $\gap(x,s)$.  This lemma is a generalization of Lemma 2.2 in \cite{xiong2023computational} in the context of LP instances. Before stating the lemma we need the following definitions:
\begin{equation}\label{eq  def lamdab min max}
	\lambda_{\max}:= \sigma_{\max}^+\left(A \right)\text{, }	\lambda_{\min}:= \sigma_{\min}^+\left(A \right), \text{ and }\kappa := \frac{\lambda_{\max}}{\lambda_{\min}}  \ .
\end{equation}

\begin{lemma}\label{lm: convergence of PHDG without restart}
	For any $r > 0$ and $\bar{z} :=(\bar{x},\bar{y})$ such that $\bar{x} \in K_p$, and $\bar s :=c - A^\top \bar{y}$, the normalized duality gap $\rho(r;\bar{z})$ provides the following bounds for $\bar w :=(\bar x, \bar s)$:
	\begin{enumerate}
		\item Distance to the affine subspace: $ \dist(\bar{w},V)  \le \frac{1}{\sqrt{\sigma} \lambda_{\min}}\cdot \rho(r;\bar{z})$, \label{item_gap1}
		\item Distance to the cone: $ \dist(\bar{w},K)  \le \frac{1}{\sqrt{\tau}} \cdot \rho(r;\bar{z})$, and \label{item_gap2}
		\item Duality gap: $\gap(\bar{w})  \le   \max\{ r, \|\bar{z}\|_M\} \rho(r;\bar{z})$.\label{item_gap3}
	\end{enumerate}
\end{lemma}
\noindent It follows from the definition of the optimal solution set $\calW^\star$ in \eqref{def_optimal_W} that when the normalized duality gap is $0$, then $\bar{w} \in \calW^\star$. Moreover, if $\max\{ r, \|\bar{z}\|_M\}$ is not too large, the magnitude of $\rho(r;\bar{z})$ also measures how close to optimality the primal-dual solution $\bar{w}$ is. We will show later that under some mild initial point conditions, the magnitude of $\max\{ r, \|\bar{z}\|_M\}$ in PDHG is well-controlled by the distance to optimal solutions.
\begin{proof}[Proof of Lemma \ref{lm: convergence of PHDG without restart}]
	The following proof is a variation of the proof of  Lemma 2.1 in \cite{xiong2023computational}. Define $\bar{s}=c - A^\top \bar{y} $ and $\bar{w} = (\bar{x},\bar{s})$. From the definition of $\rho(r;\cdot)$ we have:
	\begin{equation}\label{eq  lm: convergence of PHDG without restart 1}
		L(\bar{x},y) - L(x,\bar{y}) \le r \rho(r;\bar{z}) \ \ \text{for any $z=(x,y) \in B(r;\bar{z})$ . }
	\end{equation}

	We first prove item (\textit{\ref{item_gap1}}.), which is the distance to the affine subspace $V$.
	Let $u = b - A\bar{x}$ and define $y:= \bar{y} + \sqrt{\sigma}r \cdot u/ \|u\|$. Set $z := (\bar x, y)$, whereby $ z \in B(r;\bar{z})
	$ and hence from \eqref{eq  lm: convergence of PHDG without restart 1} we have
	$$ r \rho(r;\bar{z}) \ge L(\bar{x},y) - L(\bar x,\bar{y}) = (b-A \bar x)^\top (y - \bar y)  = \sqrt{\sigma}r \|u\| \ , $$ which means $\|u\|=\|A\bar{x} - b\| \le \frac{\rho(r;\bar{z})}{\sqrt{\sigma}}$. Let $\hat{x} \in \arg\min_{x\in V_p}\|x - \bar{x}\|$ and hence $\dist(\bar{x},V_p) = \|\hat{x} - \bar{x}\|$.  Note from the standard optimality conditions that $\hat{x} - \bar{x} \in \operatorname{Im}(A^\top)$. Since
	$$
		\|A\bar{x} - b\| = \|A\bar{x} - A \hat{x}\| \ge \min_{v\in \operatorname{Im}(A)}\frac{\|Av\|}{\|v\|}\cdot \|\bar{x} - \hat{x}\| = \lambda_{\min} \|\bar{x} - \hat{x}\| \ ,
	$$
	then it follows that $\dist(\bar{x},V_p) = \|\hat{x} - \bar{x}\| \le  \frac{\rho(r;\bar{z})}{\sqrt{\sigma} \lambda_{\min}}$. This proves item (\textit{\ref{item_gap1}}.).

	Let us now prove item (\textit{\ref{item_gap2}}.). It holds trivially from the supposition that $\bar x \in K_p$ that $\dist(\bar{x},K_p)  = 0$, and hence we only need to prove  $ \dist(\bar{s},K_d) \le \frac{1}{\sqrt{\tau}} \cdot \rho(r;\bar{z})$. Let us denote $\hat s := P_{K_d}(\bar{s})$ and $d:= \bar s - \hat s$.  Then it follows from the optimality conditions of the projection problem $\hat s =\arg\min_{s \in K_d} \|s - \bar s\|$ that $\hat s \in K_d$, $-d \in K_d^* = K_p$, and $ d ^\top \bar s = 0$.  If $d = 0$ then $\bar s \in K_d$ and the bound in item (\textit{\ref{item_gap2}}.) holds trivially.  If $d \ne0$ then define $x:= \bar{x} - \sqrt{\tau} r  \cdot d / \|d\|$ and set $z := (x, \bar y)$, whereby $ z \in B(r;\bar{z})
	$ and hence from \eqref{eq  lm: convergence of PHDG without restart 1} we have
	\begin{equation}\label{ineq item 3}
		\begin{aligned}
			r \rho(r;\bar{z}) & \ge L(\bar x, \bar y) - L(x,\bar{y})  =  (c - A^\top \bar{y})^\top (\bar{x} - x) = \bar{s}^\top  d \cdot \sqrt{\tau} r / \|d\|  = ( \hat s  + d)^\top d \cdot \sqrt{\tau} r / \|d\| \\
			                  & \ge d^\top  d\cdot \sqrt{\tau} r / \|d\|  = \sqrt{\tau} r  \|d\| =  \sqrt{\tau} r \cdot \dist(\bar{s},K_d) \ ,
		\end{aligned}
	\end{equation} where the second inequality follows since $\hat s \in K_d$ and $d \in K_p$ whereby $\hat s^\top d \ge 0$.  It therefore follows that $\dist(\bar{s},K_d)  \le \frac{1}{\sqrt{\tau}} \cdot \rho(r;\bar{z}) $, which proves item (\textit{\ref{item_gap2}}.).

	Lastly, we examine the duality gap $\gap(\bar{x},\bar{s}) = c^\top \bar{x} - b^\top \bar{y}$, and we consider two cases, namely $\bar z = 0$ and $\bar{z} \ne 0$.  If $\bar{z} = 0$, then $\gap(\bar{x},\bar{s}) = c^\top \bar{x} - b^\top \bar{y}=0$, which satisfies the duality gap bound trivially.  If $\bar{z} \neq 0$, then define $z := \bar{z} - \min\{\frac{r}{\|\bar{z}\|_M},1\}\bar{z}$, which satisfies $\|z - \bar{z}\|_M \le r$. Substituting this value of $z$ in \eqref{eq  lm: convergence of PHDG without restart 1} yields:
	\begin{equation}\label{eq  lm: convergence of PHDG without restart 7}
		r \rho(r;\bar{z}) \ge L(\bar{x},y) - L(x,\bar{y}) = \min\left\{\frac{r}{\|\bar{z}\|_M},1\right\}(c^\top \bar{x} - b^\top \bar{y}) = \min\left\{\frac{r}{\|\bar{z}\|_M},1\right\}(c^\top \bar{x} +q^\top \bar s -q_0) \ ,
	\end{equation}
	which after rearranging yields
	\begin{equation}\label{eq  lm: convergence of PHDG without restart 8}
		\gap(\bar{w})  = c^\top \bar{x} +q^\top \bar s -q_0 \le   \max\{ r, \|\bar{z}\|_M\} \rho(r;\bar{z}) \ .
	\end{equation}
	This proves the desired bound in item (\textit{\ref{item_gap3}}.).\end{proof}

For LP instances it is shown in \cite{applegate2023faster} that the normalized duality gap $\rho(r;z)$ can be easily computed or approximated.  In Appendix \ref{appendix:compute_rho} we show how to compute $\rho(r;z)$ for more general CLP instances.

\subsection{Sublinear convergence of PDHG for \eqref{pro: general saddlepoint clp}}

Let the $k$-th iterate of PDHG be denoted as $z^k$, and the average of the first $k$ iterates be denoted as $\bar{z}^k := \frac{1}{k}\sum_{i=1}^k z^i$. The iterates generated by PDHG satisfy the following desirable distance properties to the set of saddlepoints $\calZ^\star$, as stated in the following lemma.

\begin{lemma}{\bf (Nonexpansive property, essentially Proposition 2 of \cite{applegate2023faster})}\label{lm: nonexpansive property} Suppose that  $\sigma, \tau$ satisfy \eqref{eq:general_stepsize}. For any saddlepoint $z^\star$ of \eqref{pro: general saddlepoint clp}, and for all $k\ge0$, it holds that
	\begin{equation}\label{eq nonexpansive property}
		\| z^{k+1} -z^\star \|_M \le  \|z^k  - z^\star \|_M \ .
	\end{equation}
	Therefore under the assignment $z:= z^k$ or $z := \bar z^k$ it holds that $
		\left\| z -z^\star\right\|_M \le \left\|z^0  - z^\star\right\|_M$.
\end{lemma}
\noindent Lemma \ref{lm: nonexpansive property} is essentially a restatement of Proposition 2 in \cite{applegate2023faster}. The inequality \eqref{eq nonexpansive property}, also known as the nonexpansive property, appears in many other operator splitting methods \cite{liang2016convergence,ryu2022large,applegate2023faster}. We also will make use of the following lemma, which is a restatement from Lemma 2.7 of \cite{xiong2023computational}.

\begin{lemma}{\bf (from Lemma 2.7 of \cite{xiong2023computational})}\label{lm: R in the opt gap convnergence} Suppose $z^a$, $z^b$, and $z^c$ satisfy the nonexpansive properties:  $\|z^b - z^\star\|_M \le \| z^{a} - z^\star\|_M$ and $\|z^c - z^\star\|_M \le \| z^{a} - z^\star\|_M$ for every $z^\star \in \calZ^\star$. Then
	\begin{equation}\label{eq of lm: R in the opt gap convnergence}
		\begin{aligned}
			\max\{ \|z^b - z^c\|_M, \|z^b\|_M\}  \le & \ 2 \dist_M(z^a,\calZ^\star)  + \|z^{a} \|_M	\ .
		\end{aligned}
	\end{equation}
\end{lemma}

Using Lemma \ref{lm: nonexpansive property}, the sublinear convergence of the normalized duality gap has been demonstrated in \cite{applegate2023faster,xiong2023computational}, among others. Here we directly present a restatement of Corollary 2.4 of \cite{xiong2023computational}, which was initially developed for LP problems but in fact holds more broadly for the more general conic optimization problem \eqref{pro: general saddlepoint clp}.
\begin{lemma}{\bf (Sublinear convergence of PDHG, from Corollary 2.4 of \cite{xiong2023computational})}\label{lm: original sublinear PDHG}
	Suppose that  $\sigma, \tau$ satisfy \eqref{eq:general_stepsize}. Then for any $z^0 := (x^0,y^0)$ with $x^0 \in K_p$, it holds for all $k \ge 1$ that
	\begin{equation}
		\rho(\|\bar{z}^k-z^0\|_M;\bar{z}^k) \le \frac{8\dist_M(z^0,\calZ^\star)}{k} \ .
	\end{equation}
\end{lemma}
\noindent Combining the results of Lemma \ref{lm: convergence of PHDG without restart} and Lemma \ref{lm: original sublinear PDHG}, we obtain the following corollary regarding sublinear convergence of PDHG for \eqref{pro: general saddlepoint clp}.
\begin{corollary}\label{thm covergence result}
	Suppose that  $\sigma, \tau$ satisfy \eqref{eq:general_stepsize}, and PDHG is initiated with $z^0 = (x^0,y^0)$. For all $k \ge 1$, let $\bar{s}^k := c - A^\top \bar{y}^k$. Then the following hold for $\bar{w}^k:= (\bar{x}^k,\bar{s}^k)$ for all $k \ge 1$:
	\begin{enumerate}
		\item Distance to the affine subspace: $ \dist(\bar{w}^k,V)  \le \frac{8}{\sqrt{\sigma} \lambda_{\min}}\cdot  \frac{\dist_M(z^0,\calZ^\star)}{k}$, \label{item_conv1}
		\item Distance to the cone: $ \dist(\bar{w}^k,K)  \le \frac{8}{\sqrt{\tau}} \cdot  \frac{\dist_M(z^0,\calZ^\star)}{k}$, and \label{item_conv2}
		\item Duality gap: $\gap(\bar{w}^k)  \le   \left(16\dist_M(z^0,\calZ^\star) + 8\|z^0\|_M\right)\cdot  \frac{\dist_M(z^0,\calZ^\star)}{k}$.\label{item_conv3}
	\end{enumerate}
\end{corollary}

\begin{proof}
	The upper bounds for the distances to the affine subspace $V$ and the cone $K$ follow directly from Lemma \ref{lm: convergence of PHDG without restart} and Lemma \ref{lm: original sublinear PDHG}.  To prove item (\textit{\ref{item_conv3}}.), we apply Lemma \ref{lm: R in the opt gap convnergence} with $z^a := z^0$, $z^b:=\bar{z}^k$ and $z^c:= z^0$, which then satisfy the nonexpansive properties of Lemma \ref{lm: R in the opt gap convnergence}, whereby it holds that
	$$
		\max\{ \|\bar{z}^k - z^0\|_M, \|\bar{z}^k\|_M\}  \le    2 \dist_M(z^0,\calZ^\star)  + \|z^0 \|_M \ .
	$$
	Then item (\textit{\ref{item_conv3}}.) of the corollary follows by applying the above inequality to item (\textit{\ref{item_gap3}}.) of Lemma \ref{lm: convergence of PHDG without restart} with $r = \|\bar{z}^k - z^0\|_M$. \end{proof}

\section{Complexity of restarted-PDHG for Conic LP}\label{sec:complexity_clp}
In addition to the convergence analysis of PDHG, \cite{applegate2023faster,xiong2023computational} show that fixed-period and/or adaptive restarts lead to faster convergence of PDHG for LP problems, in both theory and practice. Algorithm \ref{alg: PDHG with restarts} describes our general restart scheme for PDHG for conic LP, which is similar to Algorithm 2 in \cite{xiong2023computational} for LP problems. We refer to this algorithm as ``rPDHG'' which is short for ``restarted-PDHG.'' 

\begin{algorithm}[htbp]
	\SetAlgoLined
	{\bf Input:} Initial iterate $z^{0,0}:=(x^{0,0}, y^{0,0})$, $n \gets 0$, and step-size $\tau,\sigma$ satisfying \eqref{eq:general_stepsize} \;
	\Repeat{\text{Either $z^{n,0}$ is a saddlepoint or $z^{n,0}$ satisfies some other convergence condition }}{
		\textbf{initialize the inner loop:} inner loop counter $k\gets 0$ \;
		\Repeat{ $\bar{z}^{n,k}$ satisfies some (verifiable) restart condition \ }{
			\textbf{conduct one step of PDHG: }$z^{n,k+1} \gets \textsc{OnePDHG}(z^{n,k})$ \;
			\textbf{compute the average iterate in the inner loop. }$\bar{z}^{n,k+1}\gets\frac{1}{k+1} \sum_{i=1}^{k+1} z^{n,i}$
			\label{line:average} \;  \label{line:output-is-average-of-iterates}
			$k\gets k+1$ \;
		}\label{line:restart_condition}
		\textbf{restart the outer loop:} $z^{n+1,0}\gets \bar{z}^{n,k}$, $n\gets n+1$ \;
	}
	{\bf Output:} $z^{n,0}$ ($ \ = (x^{n,0},  y^{n,0})$)
	\caption{rPDHG: restarted-PDHG}\label{alg: PDHG with restarts}
\end{algorithm}

Here $z^{k+1} \gets \textsc{OnePDHG}(z^k)$ is an iteration of PDHG as described in Algorithm \ref{alg: one PDHG}. For each iterate $z^{n,k} = (x^{n,k},y^{n,k})$, we define $s^{n,k}:= c - A^\top y^{n,k}$ and $\bar{s}^{n,k}:= c - A^\top \bar{y}^{n,k}$, and $\bar{s}^{n,k}$ denotes the average of dual cone variable iterate values. The double superscript on the variables $z^{n,k}$, $s^{n,k}$, and $\bar{s}^{n,k}$ indexes the outer iteration counter followed by the inner iteration counter, so that $z^{n,k}$ is the $k$-th inner iteration of the $n$-th outer loop.

In order to implement Algorithm \ref{alg: PDHG with restarts} (rPDHG) it is necessary to specify a (verifiable) restart condition on the average iterate $\bar z^{n,k}$ in Line \ref{line:restart_condition} that is used to determine when to re-start PDHG.
We will primarily consider Algorithm \ref{alg: PDHG with restarts} (rPDHG) using the following restart condition in Line \ref{line:restart_condition}:
\begin{equation}\label{catsdogs}\rho(\|\bar{z}^{n,k} - z^{n,0}\|_M; \bar{z}^{n,k}) \le \beta \cdot \rho(\|z^{n,0} - z^{n-1,0}\|_M; z^{n,0}) \ , \end{equation}
for a specific value of $\beta \in (0,1)$ (in fact we will use $\beta = 1/e$ where $e$ is the base of the natural logarithm).  In this way \eqref{catsdogs} is nearly identical to the condition used in \cite{applegate2023faster}.
Note that condition \eqref{catsdogs} essentially states that the normalized duality gap shrinks by the factor $\beta$ between restart values $\bar{z}^{n,k}$ and ${z}^{n,0}$. (In Appendix \ref{appendix:compute_rho} we show that computing the normalized duality gap can be done efficiently.  Also, in practice the restart condition \eqref{catsdogs} does not need to be checked frequently, so the overall cost of evaluating the restart condition is quite minor.)

\subsection{Condition numbers of the primal-dual sublevel set}

In this section we present computational guarantees for rPDHG.  Our analysis involves three condition numbers related to the primal-dual sublevel sets, and all three condition numbers have a geometric flavor. Recall from Section \ref{sec:CLP} the definition of the primal-dual optimal solution set $\calW^\star$ in \eqref{def_optimal_W} in the space of the cone variables $w=(x,s)$.
\begin{definition}{\bf ($\delta$-sublevel set)}\label{def level set}
	For any $\delta \ge  0$, the $\delta$-sublevel set is defined as
	\begin{equation}\label{eq delta level set}
		\calW_\delta :=   \calF \cap \{w: \gap(w) \le \delta\} \ ,
	\end{equation}
	which is the set of feasible primal-dual solution $w:=(x,s)$ whose duality gap $\gap(w)$ is at most $\delta$.
\end{definition}
Observe that when $\delta = 0$, then $\calW_0 = \calW^\star$, and for all $\delta \ge 0$ we have $\calW^\star \subseteq \calW_\delta$. We now present the three condition numbers we will use as the basis of our computational guarantees for rPDHG as follows.
\begin{definition}{\bf (Diameter of $\calW_\delta$)}\label{def diameter}
	For any $\delta$, the diameter of $\calW_\delta$ is:
	\begin{equation}\label{eq diameter}
		D_\delta := \max_{u,v \in \calW_\delta} \|u - v\| \ .
	\end{equation}
\end{definition}
\begin{definition}{\bf (Conic radius and conic center of $\calW_\delta$)}\label{def radius}
	For any $\delta > 0$, the conic center $w_\delta$ and the conic radius $r_\delta$ of $\calW_\delta$ are the optimal solutions of the following problem:
	\begin{equation}\label{eq radius}
		\begin{aligned}
			(w_\delta,r_\delta) := & \arg\max_{w\in\mathbb{R}^{2n},r\ge0}  \quad r                                                   \\
			                       & \ \ \ \ \ \ \quad  \operatorname{s.t.} \quad  \ \  w \in \calW_\delta, \ B(w,r)\subseteq K  \ .
		\end{aligned}
	\end{equation}
\end{definition} \noindent The conic center $w_\delta$ is the point in $\calW_\delta$ of maximum distance to the boundary of $K$, and $r_\delta$ is the distance from $w_\delta$ to the boundary of $K$.

\begin{definition}{\bf (Hausdorff distance between $\calW_\delta$ and $\calW^\star$)}\label{def distance to optima}
	For any $\delta$, let $d^H_\delta$ denote the Hausdorff distance between $\calW_\delta$ and $\calW^\star$, namely
	\begin{equation}\label{eq distance}
		d^H_\delta := D^H(\calW_\delta, \calW^\star) = \max_{w\in\calW_\delta} \dist(w,\calW^\star) \ ,
	\end{equation}
	where $D^H(\cdot,\cdot)$ denotes the Hausdorff distance. (Note that the second equality above holds because $\calW^\star \subseteq \calW_\delta$.)
\end{definition}

While the condition number $r_\delta$ is similar in concept to the measure $r_\delta$ defined in \cite{freund2003primal}, the difference is that \eqref{eq radius} is defined on the primal-dual suboptimal solution set $\calW_\delta$ and thus has quite different properties. The diameter $D_\delta$ should not be confused with $R_\delta$ defined in \cite{freund2003primal}, as $R_\delta$ refers to the maximum norm of a solution in the primal or dual sublevel sets. From the definition of $d^H_\delta$, the smaller $\delta$ is then the smaller $d^H_\delta$ is. Furthermore, as $\delta$ goes to $0$, since the sublevel sets converge to the set of optimal solutions, then $d^H_\delta$ converges to $0$. We also have the following straightforward observation regarding $D_\delta$, $r_\delta$ and $d^H_\delta$.

\begin{lemma}\label{lm D ge r}
	Under Assumption \ref{assump:striclyfeasible} it holds for all $\delta > 0$ that $D_\delta \ge d^H_\delta > r_\delta > 0$ .
\end{lemma}
\begin{proof} We have
	$$d^H_\delta = \max_{w \in \calW_\delta} \min_{w' \in \calW^\star} \|w-w'\| \le \max_{w \in \calW_\delta} \max_{w' \in \calW^\star} \|w-w'\| \le \max_{w \in \calW_\delta} \max_{w' \in \calW_\delta} \|w-w'\| = D_\delta \ , $$
	where the second inequality follows since $\calW^\star \subseteq \calW_\delta$.
	Given $(x^\star, s^\star) \in \calW^\star$, then since $x^\star \in K_p$ and $s^\star \in K_d = K_p^\star$ and $(x^\star)^\top s^\star =0$, it follows that either $x^\star \in \partial K_p$ or $s^\star \in \partial K_d$ (or both), whereby $(x^\star, s^\star) \in \partial K$.  Therefore $\calW^\star \subset \partial K$ and we have:
	$$d^H_\delta = \max_{w \in \calW_\delta} \dist(w, \calW^\star) > \dist(w_\delta, \calW^\star) \ge \dist(w_\delta, \partial K) = r_\delta  \ , $$
	where the second inequality follows from $\calW^\star \subset \partial K$.
	In addition, the first inequality holds strictly because $w_\delta \in \textsf{int} K$ but $\arg\max_{w\in\calW_\delta}\dist(w,\calW^\star)$ must lie in $\partial K$.
	Last of all, under Assumption \ref{assump:striclyfeasible} and using the convexity of the sublevel sets it follows that there exists $(x,s) \in \calW_\delta$ with $(x,s) \in \textsf{int} K_p \times \textsf{int} K_d$, and hence $(x,s) \in \textsf{int} K$ and $r_\delta > 0$.\end{proof}

We now develop and state an ``error bound'' type of result for $\calW_\delta$ involving the quotient $D_\delta / r_\delta$ that will be useful in later proofs, and that perhaps might be of independent interest. Let us use the notation $\calF_{++}$ to denote the set of strictly feasible solutions in $\calF$, namely $\calF_{++} := V \cap \textsf{int} K$. Let $w\in V \setminus \calF$ and $w_{int}\in \calF_{++}$ be given, whereby the line segment from $w$ to $w_{int}$ will contain a unique point that lies in $\partial K$, and let us denote this point by $\calF(w;w_{int}) $.  More formally we have
\begin{equation}\label{eq  def v}
	\calF(w;w_{int})  := \arg\min_{\tilde{w}}\left\{\|w - \tilde{w}\| : \tilde{w}:= \lambda \cdot w_{int} + (1- \lambda)\cdot w, \ \tilde{w}\in \calF
	\right\} \ .
\end{equation}
The following lemma states for the sublevel set $\calW_\delta$ that if the ratio $D_\delta / r_\delta$ is small, then a point in $V \cap \{w: \gap(w)\le \delta\}$ that is close to $K$ must also be close to $\calW_\delta$. In this way we see that $D_\delta / r_\delta$ is in fact an error bound for $\calW_\delta$.

\begin{lemma}\label{lm error bound R r}
	For any $\delta >0$ and $w \in V$ with $\gap(w) \le \delta$,  it holds that either $\dist(w,K) = 0$ and $w \in \calW_\delta$, or
	\begin{equation}\label{eq local error bound}
		\frac{\dist(w, \calW_\delta)}{\dist(w, K)} \le \frac{\| w - \calF(w;w_\delta) \|}{\dist(w, K)} \le \frac{\| w_\delta - \calF(w;w_\delta) \|}{r_\delta}  \le \frac{D_\delta}{r_\delta} \ .
	\end{equation}
\end{lemma}
\begin{proof}[Proof of Lemma \ref{lm error bound R r}]
	If $\dist(w,K) =0$, then $w \in \calW_\delta$ because $w \in V$ and $\gap(w) \le \delta$ by the hypotheses of the lemma.  If $w \notin K$, then $w \notin \calF$, and let $v := \calF(w;w_\delta)$.  Notice that $w_\delta \in \calW_\delta $, $w \in V$, and $\gap(w) \le \delta$ together imply that $v := \calF(w;w_\delta) \in \calW_\delta$. Then the first inequality in \eqref{eq local error bound} holds because $v \in \calW_\delta$ and so $ \dist(w, \calW_\delta) \le \| w - v \| = \| w - \calF(w;w_\delta) \|$.  For the third inequality of \eqref{eq local error bound} notice that $w_\delta \in \calW_\delta $ and $v \in \calW_\delta $ imply that $ \| w_\delta - v \| = \| w_\delta - \calF(w;w_\delta) \|  \le \Diam(\calW_\delta) = D_\delta$, which yields the third inequality of \eqref{eq local error bound}.

	We now prove the second inequality of \eqref{eq local error bound}. From the definition in \eqref{eq  def v}, because $w_\delta \in \calF_{++}$, there exists $\lambda \in (0,1)$ for which
	\begin{equation}\label{eq actual set v}
		v = {\lambda} \cdot w_{\delta} + (1- {\lambda}) \cdot w \ .
	\end{equation}
	Furthermore, since $v \in \partial K$, then there exists a supporting hyperplane $H$ of $K$ that contains $v$.  It then follows that there exists $p \in \mathbb{R}^{2n}$ for which $H := \{\hat{w}\in \mathbb{R}^{2n}: p^\top \hat{w} = 0\}$, $p^\top v = 0$, $p \in K^*$, and $p^\top w<0$ and $p^\top \hat w >0$ for all $\hat w \in \textsf{int}K$ and so in particular $p^\top w_\delta > 0$.  From \eqref{eq actual set v} we have
	$$
		{\lambda} \cdot p^\top w_\delta + (1 - {\lambda}) \cdot p^\top  w = p^\top v = 0 \ , $$ which can be rearranged to yield the following equalities:
	\begin{equation}\label{eq error bound R r}
		\frac{\lambda}{1-\lambda} = \frac{|p^\top w |}{|p^\top w_\delta |}  = \frac{\dist(w, H)}{\dist(w_\delta,H)} \ .
	\end{equation}
	From \eqref{eq actual set v} the left side of \eqref{eq error bound R r} can be further expressed as \begin{equation}\label{eq error bound R r 1}
		\frac{\lambda}{1-\lambda} = \frac{\| w - v\|}{\|w_\delta - v\|} \ .
	\end{equation}
	Also, since $K$ and $w$ are on different sides of the hyperplane $H$, this implies that	\begin{equation}\label{eq error bound R r 2}
		\dist(w, H) \le \dist(w, K) \ .
	\end{equation}
	Additionally, because $B(w_\delta,r_\delta) \subseteq K$ (from Definition \ref{def radius}), we have
	\begin{equation}\label{eq error bound R r 3}
		\dist(w_\delta, H) \ge r_\delta \ .
	\end{equation}
	Substituting \eqref{eq error bound R r 1}, \eqref{eq error bound R r 2}, \eqref{eq error bound R r 3} back into \eqref{eq error bound R r} yields
	\begin{equation}\label{eq error bound R r 4}
		\frac{\| w - v\|}{\|w_\delta - v\|} \le \frac{ \dist(w, K)  }{r_\delta} \ ,
	\end{equation}
	which proves the second inequality in  \eqref{eq local error bound}.
\end{proof}

\subsection{Computational guarantees for CLP problems}

In this subsection we present our computational guarantees for rPDHG, which are based on the three geometry-based condition numbers $D_\delta$, $r_\delta$ and $d^H_\delta$ for sublevel sets $\calW_\delta$. We consider the following restart condition rPDHG which is the same as the adaptive restart scheme introduced in \cite{applegate2023faster} and also used in \cite{xiong2023computational}.
\begin{definition}[$\beta$-restart condition]\label{def beta restart}
	For a given $\beta \in(0,1)$, the iteration $(n,k)$ satisfies the \textit{$\beta$-restart condition} if $n \ge 1$ and condition \eqref{catsdogs} is satisfied, or $n=0$ and $k=1$.
\end{definition}

We also introduce the following definition of objective function error which separately considers primal objective error and dual objective error.
\begin{definition}[Objective function error]\label{def objective error} Suppose that the primal and dual problems \eqref{pro: general primal clp} and \eqref{pro: general dual clp on s} have the common optimal objective function value $f^\star$. For the candidate pair $w := (x,s)$, the objective function error of $w$ is defined as:
	\begin{equation}\label{eq objective error}
		\eobj(w) :=  |c^\top x - f^\star | + |  f^\star -q_0 + q^\top s  | \ .
	\end{equation}
\end{definition}
Note that when both $x$ and $s$ are feasible, then the objective error $\eobj(w)$ is equal to $\gap(w) = c^\top x + q^\top s - q_0$. However, having a small $\gap(w)$ does not imply that $\eobj(w)$ is small. For this reason, $\eobj(w)$ is a particularly appropriate measure of the optimality gap of a proposed solution when $w$ is not primal/dual feasible.

Depending on the specific application, the tolerance requirement of a given candidate solution $w=(x,s)$ in terms of infeasibility, duality gap, and objective error can be different, which motivates the following definition of the $\eps$-tolerance requirement.
\begin{definition}{($\eps$-tolerance requirement)}\label{eq accuracy requirement clp} Let $w=(x,s) \in \mathbb{R}^{2n}$ be given. For the triplet $\eps := (\eps_{\mathrm{cons}}, \eps_{\mathrm{gap}},\eps_{\mathrm{obj}})\in\mathbb{R}_{++}^3$, the solution $w$ satisfies the $\eps$-tolerance requirement of \eqref{pro: general primal clp} and \eqref{pro: general dual clp on s} if
	\begin{enumerate}
		\item Distances to constraints: $\max\left\{\dist(w,V), \ \dist(w,K)\right\} \le \eps_{\mathrm{cons}}$, \label{item_gap11}
		\item Duality gap: $\gap(w) \le \eps_{\mathrm{gap}}$, and \label{item_gap22}
		\item Objective functions tolerance: $\eobj(w)\le \eps_{\mathrm{obj}}$.\label{item_gap33}
	\end{enumerate}
\end{definition}

\noindent In Definition \ref{eq accuracy requirement clp}, $\eps \in \mathbb{R}^3_{++}$ denotes the tolerance of three kinds of target error requirements, namely the distance to constraints $\eps_{\mathrm{cons}}$, duality gap $\eps_{\mathrm{gap}}$, and objectives error $\eps_{\mathrm{obj}}$. Recalling the definitions of $\lambda_{\max}$, $\lambda_{\min}$, and $\kappa$ from \eqref{eq  def lamdab min max}, we now state our main computational guarantee for Algorithm \ref{alg: PDHG with restarts} (rPDHG).

\begin{theorem}\label{thm overall complexity clp}
	Under Assumption \ref{assump:striclyfeasible}, suppose that $c\in\operatorname{Null}(A)$ and Algorithm \ref{alg: PDHG with restarts} (rPDHG) is run starting from $z^{0,0} = (x^{0,0},y^{0,0} ) = (0,0)$ using the $\beta$-restart condition with $\beta := 1/e$, and the step-sizes are chosen as follows:
	\begin{equation}\label{eq:stepsize_in_theorem}
		\tau = \frac{1}{\kappa}  \ \ \ \mathrm{and} \ \ \  \sigma = \frac{1}{\lambda_{\max}\lambda_{\min}} \ .
	\end{equation}
	Given $\eps := (\eps_{\mathrm{cons}}, \eps_{\mathrm{gap}},\eps_{\mathrm{obj}})\in\mathbb{R}_{++}^3$, let $T$ be the total number of \textsc{OnePDHG} iterations that are run in order to obtain $n$ for which $w^{n,0}=(x^{n,0},s^{n,0})$ satisfies the $\eps$-tolerance requirement (Definition \ref{eq accuracy requirement clp}).
	Then
	\begin{equation}\label{eq overall complexity}
		T \le \inf_{\delta > 0}  \ T_\delta := 190\kappa \cdot \frac{D_\delta}{r_\delta}  \cdot \Big[\ln\big(33\kappa \cdot \dist(0,\calW^\star) \big) + \ln\Big(\frac{1}{\err}\Big) \, \Big]
		+ \frac{50\kappa \cdot d^H_\delta}{\err}	 \ ,
	\end{equation}
	in which $\err$ is the following (weighted) minimum of the target tolerance values $\eps$:
	\begin{equation}\label{def eq calN}
		\err := \min\left\{\eps_{\mathrm{cons}}, \
		\tfrac{\sqrt{2}}{4 \cdot \dist(0,\calW^\star)}\cdot \eps_{\mathrm{gap}}, \
		\tfrac{1}{14}  \left( \displaystyle\sup_{\gamma  > 0} \tfrac{ r_\gamma}{\gamma} \right) \cdot \eps_{\mathrm{obj}}
		\right\} \ .
	\end{equation}
\end{theorem}

In the statement of Theorem \ref{thm overall complexity clp} we define and use the term $\err$ purely for expositional convenience, as otherwise the bound in \eqref{eq overall complexity} would be more cumbersome to write down. We note that the three target tolerances have different multipliers inside the minimum in $\err$, namely $1$, $\tfrac{\sqrt{2}}{4 \cdot \dist(0,\calW^\star)} $ , and $\tfrac{1}{14}  \left( \sup_{\gamma  > 0} \tfrac{ r_\gamma}{\gamma} \right) $, which affect the overall complexity in \eqref{eq overall complexity}.  For example, if $\dist(0,\calW^\star)$ is small, then the impact of $\eps_{\mathrm{gap}}$ on the complexity bound is not as pronounced as it would be if $\dist(0,\calW^\star)$ is large.  And while the term $\left( \sup_{\gamma  > 0} \tfrac{ r_\gamma}{\gamma} \right) $ may look rather opaque at first glance, in fact it has the property of being nearly equal to the reciprocal of $\max_{w\in \calW^\star}\|w\|$, namely:
\begin{equation}\label{xcski} \frac{\width_K}{ \max_{w\in \calW^\star}\|w\|} \ \le \ \sup_{\gamma>0} \frac{r_\gamma}{\gamma} \ \le \ \frac{1}{\max_{w\in \calW^\star}\|w\|}  \ ,\end{equation}where $\width_K$ is the width of $K$ defined in \eqref{eq:width_K}. (The proof of \eqref{xcski} is presented in Appendix \ref{app:proof of lm r calW small}.) Generally speaking, the overall complexity bound is monotone increasing in the reciprocals of $\eps_{\mathrm{cons}}$, $\eps_{\mathrm{gap}}$, and $\eps_{\mathrm{obj}}$, and (except for the multiplier terms) the structural effect of each target tolerance requirement is similar.  Hence for the purpose of discussion we will simply refer to them in aggregate as the triplet $\eps$.

Let us now examine the overall dependence on $\err$ in \eqref{eq overall complexity}. For any $\delta >0$, $T_\delta$ depends on $\err^{-1}$ in two places: one is inside the logarithm term and the other is in the right-most term, which we will call the ``linear term'' in $\err^{-1}$. When $\err$ is not very small, then the logarithm term is the dominant term, while when $\err$ is very small the linear term is the dominant term.  Note in the bound that the logarithm term is multiplied by $\kappa \frac{D_\delta}{r_\delta}$ whereas the linear term is multiplied by $\kappa d^H_\delta$, so the notion of ``very small'' will depend on the relative magnitudes of $\frac{D_\delta}{r_\delta}$ and $d^H_\delta$.

Last of all, observe in \eqref{eq overall complexity} that $T$ is bounded above by the \underline{smallest} $T_\delta$ over $\delta >0$. The constant in front of the linear term in $T_\delta$ is $O(\kappa\cdot d_\delta^H)$, which decreases to $0$ as $\delta$ goes to $0$. However, the constant in front of the logarithm term in $T_\delta$ is $O(\kappa\cdot \frac{D_\delta}{r_\delta})$, which might go to a constant as $\delta$ goes to $0$, or might go to $\infty$ as $\delta$ goes to zero. If $\lim\inf_{\delta \searrow 0} \frac{D_\delta}{r_\delta} < \infty$ (which can happen for example if the instance is a linear program and the optimal primal-dual solution is unique), then Theorem \ref{thm overall complexity clp} implies the following linear convergence complexity.
\begin{corollary}\label{cor: linear convergence complexity clp}
	In the setting of Theorem \ref{thm overall complexity clp}, the total number of iterations $T$ is bounded above by
	$190\kappa \cdot \left(\lim\inf_{\delta \searrow 0}\frac{D_\delta}{r_\delta} \right) \cdot \left(\ln(33\kappa\cdot \dist(0,\calW^\star) ) + \ln\big(\frac{1}{\err}\big)  \right)$.
\end{corollary}
\noindent Of course Corollary \ref{cor: linear convergence complexity clp} might itself be too conservative, especially when the target tolerance values are not very small.

To illustrate the above discussion, consider the following linear programming family of instances parameterized by $\nu \ge 0$:
\begin{equation}\tag{$P_\nu$}\label{pro:example_lp_validate}
	\begin{array}{c}
		\min_{x = (x_1,x_2,x_3) \in \mathbb{R}^3_+} \frac{2 + \nu }{10}\cdot x_1 + x_2 + (1 + \nu)x_3 \ \ \text{s.t. } -10x_1 + x_2 + x_3  = 1  \ ,
	\end{array}
\end{equation} and here let us choose $\nu = 10^{-4}$.  Figure \ref{fig:validation} shows the values of the ratio $D_\delta/ r_\delta$ and the Hausdorff distance $d^H_\delta$ as a function of $\delta$ for different sublevel set values. Figure \ref{fig:validation} illustrates that both of these measures can exhibit significant variation.
\begin{figure}[htbp]
	\centering
	\includegraphics[width=0.65\textwidth]{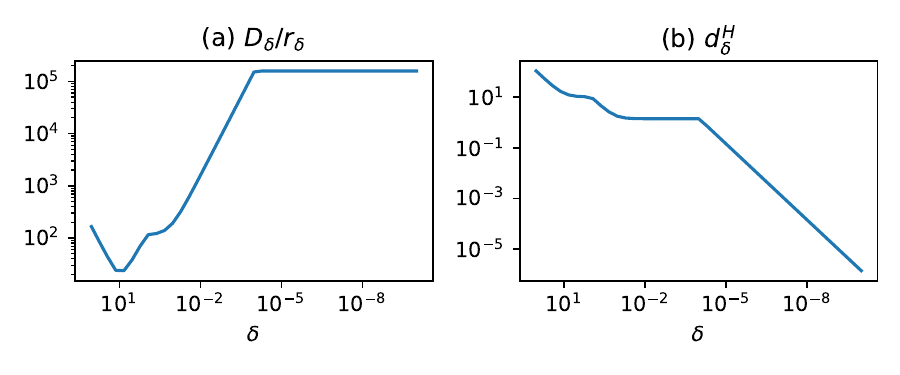}\vspace{-15pt}
	\caption{Values of $D_\delta/r_\delta$ and $d_\delta^H$ as a function of $\delta$ for the linear programming instance \eqref{pro:example_lp_validate} with $\nu = 10^{-4}$. }\label{fig:validation}
\end{figure} 

Let us now examine the values of the complexity bound functions $T_\delta$ applied to the linear programming instance \eqref{pro:example_lp_validate}.  For a given scalar $\bar\eps >0$ we can set the required tolerances uniformly, namely $\eps =(\eps_{\mathrm{cons}},\eps_{\mathrm{gap}},\eps_{\mathrm{obj}}) := (\bar\eps,\bar\eps,\bar\eps ) $, and in this way the required tolerance is just a function of the scalar $\bar\eps$. The value of $\delta$ that yields the smallest value of $T_\delta$ can be different depending on the value of the tolerance $\bar\eps$. This is illustrated in Figure \ref{fig:Tdelta and actual T}. The left figure in Figure \ref{fig:Tdelta and actual T} shows plots of $T_\delta$ as a function of $\bar\eps$ for $\delta=10^{-1}$ and $\delta=10^{-8}$. One can see that $T_{10^{-1}}$ is a better bound than $T_{10^{-8}}$ for $\bar{\eps} > 10^{-5}$, and that $T_{10^{-8}}$ is a better bound than $T_{10^{-1}}$ for $\bar{\eps} < 10^{-5}$. The right figure in Figure \ref{fig:Tdelta and actual T} shows plots of the actual iteration counts of \textsc{OnePDHG} in running rPDHG on \eqref{pro:example_lp_validate} with $\nu = 10^{-4}$, and also shows the lower envelope theoretical bound \eqref{eq overall complexity}, namely $\inf_{\delta >0} T_\delta$ for  \eqref{pro:example_lp_validate} as a function of $\bar\eps$.  Here we observe that the theoretical bound $\inf_{\delta >0} T_\delta$ is consistently off by roughly a constant factor of $10^{4}$. When the target tolerance is not extremely small, such as $\bar{\eps} > 10^{-5}$, then $\inf_{\delta >0} T_\delta$ is more similar to $T_{10^{-1}}$ than $T_{10^{-8}}$, and the actual iteration count is also more closely matched by $T_{10^{-1}}$. As the target tolerance $\bar{\eps}$ becomes smaller, then $\inf_{\delta >0} T_\delta$ is more similar to $T_{10^{-8}}$ than $T_{10^{-1}}$, and points to the observation that rPDHG achieves linear convergence after a certain number of iterations. And although it is tempting to think that the linear convergence result in Corollary \ref{cor: linear convergence complexity clp} is ``better'' than sublinear convergence, this example shows that this is not necessarily the case when the target tolerance is not so small.  Indeed, the more nuanced complexity bound $\inf_{\delta > 0}T_\delta$ corresponds better to computational practice on this example.

\begin{figure}[htbp]
	\centering
	\includegraphics[width=0.65\textwidth]{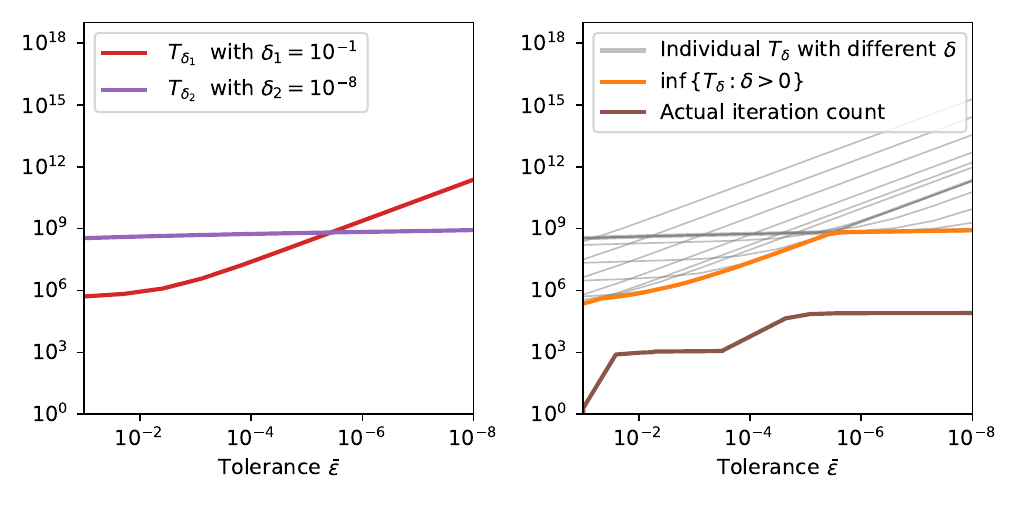}\vspace{-10pt}
	\caption{(Left) Plots of $T_\delta$ for $\delta=10^{-1}$ and $\delta=10^{-8}$, as a function of the required tolerance $\eps := (\eps_{\mathrm{cons}}, \eps_{\mathrm{gap}},\eps_{\mathrm{obj}}) := (\bar\eps,\bar\eps,\bar\eps ) $ for the linear programming instance \eqref{pro:example_lp_validate} for $\nu=10^{-4}$. (Right) Plots of the actual iteration counts of \textsc{OnePDHG}, and the lower envelope $\inf_{\delta >0} T_\delta$.}\label{fig:Tdelta and actual T}
\end{figure}

We also point out that $d^H_\delta \rightarrow 0$ as $\delta \rightarrow 0$, and hence $d^H_\delta$ can be made arbitrarily small by choosing $\delta$ sufficiently small.  However, the ratio $D_\delta/r_\delta$ might be exceedingly large (especially for small values of $\delta$) and so might be $\inf_{\delta>0}\frac{D_\delta}{r_\delta}$ . Therefore, in such cases, the logarithmic term might itself be the dominant term in the iteration bound in $\inf_{\delta > 0}T_\delta$.

In the case of linear programming instances, if there are multiple optimal solutions then $\lim_{\delta \searrow 0}\frac{D_\delta}{r_\delta} = \infty$, and Corollary \ref{cor: linear convergence complexity clp} does not itself yield a linear convergence bound.  (However, linear convergence is guaranteed through previous analyses; see \cite{applegate2023faster,xiong2023computational}.) Moreover, we will show a different linear convergence result for linear programming instances using the geometry of sublevel sets in Section \ref{sec:complexity_lp}.

Note that all terms of the complexity bound \eqref{eq overall complexity} are linear in $\kappa$.  This suggests a strategy where we use a preconditioner $D$ on the linear constraints, replacing $Ax=b$ by $DAx = Db$, so that $\sigma^+_{\max}(DA)/\sigma^+_{\min}(DA)$ becomes smaller. Except for $\kappa$, such a preconditioner does not impact any of the quantities in the complexity bound \eqref{eq overall complexity}. For example, using $D = (AA^\top)^{-1/2}$ yields $\kappa=1$ without changing any other terms in \eqref{eq overall complexity}. Furthermore, getting the optimal diagonal preconditioner $D$ is essentially solving a semidefinite program \cite{qu2024optimal}, and a scalable approximation method is available in \cite{gao2023scalable}.

Finally, we remark that the step-sizes used in \eqref{eq:stepsize_in_theorem} are relatively easy to compute to reasonable accuracy so long as it is inexpensive to estimate the largest and smallest singular values of $A$. Furthermore, so long as the step-sizes satisfy \eqref{eq:general_stepsize}, using step-sizes that only approximately adhere to \eqref{eq:stepsize_in_theorem} only moderately changes the computational bounds in Theorem \ref{thm overall complexity clp}.

The proof of Theorem \ref{thm overall complexity clp} is based on demonstrating an upper bound on the number of \textsc{OnePDHG} iterations needed to ensure a sufficient decrease in the normalized duality gap \eqref{sunny}.  Towards the proof thereof, we first present the following more general theorem which shows that the number of \textsc{OnePDHG} iterations of rPDHG can be appropriately bounded so long as the normalized duality gap provides an upper bound on the distance to optimal solutions.

\begin{theorem}\label{thm: complexity of PDHG with adaptive restart}Under Assumption \ref{assump:striclyfeasible}, suppose that Algorithm \ref{alg: PDHG with restarts} (rPDHG) is run starting from $z^{0,0} = (x^{0,0},y^{0,0} )$ using the $\beta$-restart condition with $\beta := 1/e$. Suppose further that there exists $\condL \ge 1$ and $ \condC \ge 0$ such that
	\begin{equation}\label{eq restart L C condition}
		\dist_M(z^{n,0}, \calZ^\star) \le \rho(\|z^{n,0} - z^{n-1,0} \|_M;z^{n,0}) \cdot \condL + \condC \
	\end{equation}
	holds for all $n \ge 1$. Let $T$ be the total number of \textsc{OnePDHG} iterations that are run in order to obtain the first outer iteration $N$ that satisfies $\rho(\|{z}^{N,0} - z^{N-1,0}\|_M; {z}^{N,0})\le \eps$.  Then

	\begin{equation}\label{eq all complexity 6}
		T \ \le \ 23\calL  \cdot \ln\left(\frac{23 \dist_M(z^{0,0},\calZ^\star)}{\eps}\right)  +  \frac{35\calC}{\eps} \ .
	\end{equation}
\end{theorem}
\begin{proof}
	We first derive an upper bound $k_n$ on the number of iterations $k$ between two consecutive restarts $z^{n, 0}$ and $z^{n+1, 0}$.  For $ n =0$, it follows trivially from Definition \ref{def beta restart} that $k_0 := 1$. For $n \ge 1$ and $k \ge 1$ it holds from Lemma \ref{lm: original sublinear PDHG} that
	\begin{equation}\label{eq thm: complexity of PDHG with adaptive restart 1}
		\rho(\|\bar{z}^{n,k}-z^{n,0}\|_M;\bar{z}^{n,k}) \le \frac{8\dist_M(z^{n,0} ,\calZ^\star)}{k}\ .
	\end{equation}
	If $\rho(\|z^{n,0}-z^{n-1,0}\|_M; z^{n,0}) = 0$ it follows from Lemma \ref{lm: convergence of PHDG without restart} that $z^{n,0} \in \calZ^\star$ which then implies that $z^{n,k} = z^{n,0}$ for all $k\ge 1$, and so in particular $k_n = 1$.  If $\rho(\|z^{n,0}-z^{n-1,0}\|_M; z^{n,0}) \neq 0$, then dividing both sides of \eqref{eq thm: complexity of PDHG with adaptive restart 1} by $\rho(\|z^{n,0}-z^{n-1,0}\|_M; z^{n,0})$ yields:
	\begin{equation}\label{eq thm: complexity of PDHG with adaptive restart 2}
		\frac{	\rho(\|\bar{z}^{n,k}-z^{n,0}\|_M;\bar{z}^{n,k})}{\rho(\|z^{n,0}-z^{n-1,0}\|_M; z^{n,0})} \le \frac{8}{k}  \cdot 	\frac{\dist_M(z^{n,0} ,\calZ^\star)}{\rho(\|z^{n,0}-z^{n-1,0}\|_M; z^{n,0})}  \ .
	\end{equation}
	Furthermore, \eqref{eq restart L C condition} implies:
	\begin{equation}\label{miggie}
		\frac{\dist_M(z^{n,0} ,\calZ^\star)}{\rho(\|z^{n,0}-z^{n-1,0}\|_M; z^{n,0})} \le  \condL + \frac{\condC}{\rho(\|z^{n,0} - z^{n-1,0} \|_M;z^{n,0})} \ .
	\end{equation}
	Let us define: $ \bar k_n :=    \frac{8}{\beta} \cdot \left(
		\condL + \frac{\condC}{\rho(\|z^{n,0} - z^{n-1,0} \|_M;z^{n,0})}
		\right) $.
	It then follows from \eqref{eq thm: complexity of PDHG with adaptive restart 2} and \eqref{miggie} that
	the restart condition \eqref{catsdogs} is satisfied for all $k \ge \bar k_n$, whereby
	\begin{equation}\label{eq restart iteration condition}k_n := \frac{8}{\beta} \cdot \left(
		\condL + \frac{\condC}{\rho(\|z^{n,0} - z^{n-1,0} \|_M;z^{n,0})}
		\right)  + 1  \end{equation} is an upper bound on the number of iterations between the consecutive restarts $z^{n, 0}$ and $z^{n+1, 0}$.

	Next we examine the first outer iteration $N$ that satisfies $\rho(\|{z}^{N,0} - z^{N-1,0}\|_M; {z}^{N,0})\le \eps$.  It follows from the $\beta$-restart condition \eqref{catsdogs} for all $n \ge 1$ that
	\begin{equation}\label{eq all complexity 1}
		\rho(\|{z}^{n,0} - z^{n-1,0}\|_M; {z}^{n,0}) \le \beta^{n-1} \cdot \rho(\|{z}^{1,0} - z^{0,0}\|_M; {z}^{1,0}) \ ,
	\end{equation}
	which in combination with Lemma \ref{lm: original sublinear PDHG} (and using $k=1$) yields:
	\begin{equation}\label{eq all complexity 2}
		\rho(\|{z}^{n,0} - z^{n-1,0}\|_M; {z}^{n,0}) \le 8 \beta^{n-1} \cdot \dist_M(z^{0,0}, \calZ^\star) \ .
	\end{equation}
	Let us define $\bar N := 1+ \frac{1}{\ln( 1/\beta) } \cdot \ln\left( \frac{8 \dist_M(z^{0,0}, \calZ^\star)}{ \eps} \right) $.  It follows from \eqref{eq all complexity 2} that $\rho(\|{z}^{n,0} - z^{n-1,0}\|_M; {z}^{n,0}) \le \eps$ for all $n \ge \bar N$, from which it follows that
	\begin{equation}\label{eq all complexity 3}
		N \ \le \  \frac{1}{\ln( 1/\beta) } \cdot \ln\left( \frac{8  \dist_M(z^{0,0}, \calZ^\star)}{ \eps} \right) + 2 \ . 	\end{equation}The total number of \textsc{OnePDHG} iterations $T$ satisfies
	\begin{equation}\label{eq all complexity 4}
		\begin{aligned}
			T \le \sum_{n=0}^{N-1} k_n  = \  & 1  + \sum_{n = 1}^{N-1}\left(
			\frac{8}{\beta} \cdot \left(
				\condL + \frac{\condC}{\rho(\|z^{n,0} - z^{n-1,0} \|_M;z^{n,0})}
				\right) + 1
			\right)                                                                                                                      \\
			\le \                            & N  + \frac{8\condL (N-1)}{\beta} +  \frac{8\condC}{\beta}  \cdot \sum_{n = 1}^{N-1}\left(
			\frac{1}{\rho(\|z^{n,0} - z^{n-1,0} \|_M;z^{n,0})}
			\right) \ .
		\end{aligned}
	\end{equation}
	We have from the definition of $N$ that $\rho(\|{z}^{N-1,0} - z^{N-2,0}\|_M; {z}^{N-1,0}) > \eps$.  Also, the $\beta$-restart condition \eqref{catsdogs} implies for all $n \le N-2$ that:
	$$
		\frac{	1}{\rho(\|z^{n,0} - z^{n-1,0}\|_M; z^{n,0})} \le 	 \frac{\beta}{\rho(\|z^{n+1,0} - z^{n,0}\|_M; z^{n+1,0})}  \le \cdots \le
		\frac{\beta^{N - 1 -n}}{\rho(\|{z}^{N-1,0} - z^{N-2,0}\|_M; {z}^{N-1,0})} \ ,
	$$
	Therefore
	\begin{equation}\label{eq all inverse gap}
		\sum_{n = 1}^{N-1}\left(
		\frac{1}{\rho(\|z^{n,0} - z^{n-1,0} \|_M;z^{n,0})}
		\right) < \frac{1}{\eps} \cdot \left(
		1 + \beta^{1} + \beta^{2} + \cdots
		\right) = \frac{1}{\eps(1-\beta)} \ .
	\end{equation}
	Substituting \eqref{eq all inverse gap}  into \eqref{eq all complexity 4} yields
	\begin{equation}\label{eq all complexity 5}
		T \le  N  + \frac{8\condL (N-1)}{\beta} +  \frac{8\condC}{\beta}  \cdot\frac{1}{\eps(1-\beta)} =  \left(1 + \frac{8 \calL}{\beta}\right) N - \frac{8 \calL}{\beta} + \frac{8 \calC}{\beta (1- \beta)}\cdot \frac{1}{\eps} \ ,
	\end{equation}
	and then using \eqref{eq all complexity 3} we arrive at
	$$
		T \le 	\left(1 + \frac{8 \calL}{\beta} \right)\cdot \frac{1}{\ln(1/\beta)} \cdot \ln\left(\frac{8 \dist_M(z^{0,0},\calZ^\star)}{\eps}\right) + \left(2 + \frac{16 \calL}{\beta} \right) - \frac{8\calL}{\beta}
		+ \frac{8 \calC}{\beta (1-\beta)}\cdot \frac{1}{\eps} \ .
	$$
	Noting that $\condL \ge 1$ and $\beta = 1/e \approx     0.3679$, the above bound can be relaxed slightly to yield 
	\begin{equation}\label{eq all complexity 7}
		T \ \le \ 23\calL  \cdot \ln\left(\frac{8 \dist_M(z^{0,0},\calZ^\star)}{\eps}\right)  + 24 \calL
		+  \frac{35\calC}{\eps} \ .
	\end{equation}
	Finally, notice that the middle term above satisfies $24 \calL \le 23\calL \cdot \ln\big(e^{24/23}\big)$ and substituting it back into \eqref{eq all complexity 7} yields  $T \le \ 23\calL  \cdot \ln\left(\frac{8 \cdot e^{24/23} \cdot \dist_M(z^{0,0},\calZ^\star)}{\eps}\right) + \frac{35\calC}{\eps}$, which then simplifies to \eqref{eq all complexity 6}.
\end{proof}

Note that if there exist $\calL < \infty$ and $\calC =0$ for which \eqref{eq restart L C condition} holds for all $n \ge 1$, then Theorem \ref{thm: complexity of PDHG with adaptive restart} states that the total number of \textsc{OnePDHG} iterations required to obtain a normalized duality gap smaller than $\eps$ is bounded above by $O(\condL \cdot \ln(\dist_M(z^{0,0}, \calZ^\star)/\eps))$.

\subsubsection{Some lemmas useful for the proof of Theorem \ref{thm overall complexity clp}}

In this subsection we present some lemmas that we will use in the proof of Theorem \ref{thm overall complexity clp} in Section \ref{gillianwelch}.

Notice that the result in Theorem \ref{thm: complexity of PDHG with adaptive restart} involves the normalized duality gap instead of more standard optimization tolerance quantities such as KKT error tolerances or distances to feasibility sets. In order to prove Theorem \ref{thm overall complexity clp} we will first need to translate between norms and tolerance metrics. Let
\begin{equation}\label{icemelt}
	c_0 := \max\left\{
	\frac{1}{\sqrt{\sigma} \lambda_{\min}  }, \ \frac{1}{\sqrt{\tau}}
	\right\} \ .
\end{equation}
We first prove the following lemma on the relation between two different norms.
\begin{lemma}\label{lm change of norm}
	Suppose that $\tau,\sigma$ satisfy \eqref{eq:general_stepsize}. For any convex sets $\calX$ and $\calS$ in $\mathbb{R}^n$, let $\calY:= \{y:c-  A^\top y \in \calS\}$. Given any point $z :=(x,y) \in \mathbb{R}^{n+m}$, let $s:= c - A^\top y$ and $w := (x,s) $.  Then it holds that
	\begin{equation}\label{eqlm change of norm}
		\dist_M(z, \calX \times \calY) \le \sqrt{2}c_0 \cdot \dist(w,\calX \times\calS) \ .
	\end{equation}
\end{lemma}
\begin{proof}
	Define:
	$$
		\left\|(x,y)\right\|_N := \sqrt{\frac{1}{\tau} \|x\|^2 + \frac{1}{\sigma} \|y\|^2} \ \text{ where } \
		N:= 	\begin{pmatrix}
			\frac{1}{\tau}I_n &                     \\
			                  & \frac{1}{\sigma}I_m
		\end{pmatrix} \ .
	$$
	Then $2N - M \in \mathbb{S}^{m+n}_+$ because $1/(\tau \sigma) \ge \lambda_{\max}^2$ due to \eqref{eq:general_stepsize}, and therefore
	$\|z\|_M \le \sqrt{2}\|z\|_N $ for any $z$. This means
	\begin{equation}\label{ineq key norm 1}
		\dist_M(z, \calX \times \calY) \le \sqrt{2}  \cdot \dist_N(z,\calX \times\calY) =
		\sqrt{2}\cdot \sqrt{\frac{1}{\tau}\cdot\dist(x,\calX)^2 + \frac{1}{\sigma}\cdot\dist(y,\calY)^2}
		\ .
	\end{equation}

	Next we claim that
	\begin{equation}\label{ineq key norm 2}
		\dist(y,\calY) \le \dist(s,\calS)\cdot \frac{1}{\lambda_{\min}} \ .
	\end{equation}
	Towards establishing \eqref{ineq key norm 2}, first observe that:
	$$
		\dist(s,\calS) = \dist(c - A^\top y, \calS) = \dist(c - A^\top y, c - A^\top (\calY)) = \dist(A^\top y, A^\top (\calY)) = \dist_{AA^\top} (y,\calY) \ .
	$$
	Let $AA^\top = PD^2P^\top$ denote the thin eigendecomposition of $AA^\top$, so that $P^\top P=I$ and $D$ is the diagonal matrix of positive singular values of $A$, whereby $D_{ii} \ge \min_j D_{jj} = \lambda_{\min}$ for each $i$.  Now let $\hat{y}$ solve the shortest distance problem from $y$ to $\calY$ in the norm $\| \cdot\|_{AA^\top}$, hence $\hat{y} \in \calY$ and  $\dist_{AA^\top} (y,\calY) = \| y - \hat{y}\|_{AA^\top}$, and let us write $y-\hat{y} = u + v $ where $u \in \operatorname{Im}(A)$ and $v \in \operatorname{Null}(A^\top)$. Then setting $\tilde y = \hat{y} + v$ and noting that $\tilde y \in \calY$, we have:
	\begin{equation}\label{skylight} \dist_{AA^\top}(y,\calY) \le \|y - \tilde y\|_{AA^\top} = \|u\|_{AA^\top} \ . \end{equation}
	Next notice that since $u \in \operatorname{Im}(A) = \operatorname{Im}(AA^\top)$, there exists $\pi$ for which $u=AA^\top \pi$, and define $\lambda = D^2P^\top \pi$. It then follows that $u = P \lambda$, $\lambda = P^\top u$, and $\|u\| = \|\lambda\|$.   We therefore have:
	\begin{equation}\label{firepit}\begin{aligned}
			\dist_{AA^\top} (y,\calY)^2 & = (u+v)^\top AA^\top (u+v) \\ & = u^\top AA^\top u = \lambda^\top P^\top P D^2P^\top P\lambda = \lambda^\top D^2 \lambda \ge \lambda_{\min}^2 \|\lambda\|^2 \ ,
		\end{aligned}
	\end{equation}
	and hence
	$$
		\dist(s,\calS) = \dist_{AA^\top} (y,\calY) \ge \lambda_{\min} \|\lambda\| = \lambda_{\min} \|u\| \ge \lambda_{\min}\dist(y,\calY) \ ,
	$$
	where the second inequality uses \eqref{skylight}.  This proves \eqref{ineq key norm 2}.

	Finally, combining \eqref{ineq key norm 1} and \eqref{ineq key norm 2} we obtain
	\begin{equation}\label{ineq key norm 3}
		\dist_M(z, \calX \times \calY) \le
		\sqrt{2}\cdot \sqrt{\frac{1}{\tau}\cdot\dist(x,\calX)^2 + \frac{1}{\sigma \lambda_{\min}^2}\cdot\dist(s,\calS)^2} \le \sqrt{2}c_0\cdot \dist(w,\calX\times \calS) \ .
	\end{equation}
\end{proof}

Under the property that $ c \in \operatorname{Null}(A)$, we have the following property of the initial iterate $z^{0,0}  =(x^{0,0},y^{0,0}): =(0,0)$.
\begin{proposition}\label{lm:distance_to_optimal_initial} Suppose that $ c \in \operatorname{Null}(A)$ and the initial iterate is $(z^{0,0})  =(x^{0,0},y^{0,0}): =(0,0)$, and $w^{0,0} = (x^{0,0},s^{0,0}) = (0,c-A^\top y^{0,0})$. Then
	\begin{equation}\label{eq bound w to wstar 2}
		\dist(w^{0,0},\calW^\star) \le \dist(0,\calW^\star)   \ .
	\end{equation}
\end{proposition}
\begin{proof}
	Let $\calX^\star$ and $\calS^\star$ denote the optimal primal and dual optimal solution sets for \eqref{pro: general primal clp} and \eqref{pro: general dual clp on s}, respectively.  Then $(x^{0,0},y^{0,0}) = (0,0)$ implies that $w^{0,0} = (x^{0,0},s^{0,0}) = (0,c)$ and hence
	\begin{equation}\label{eq bound w to wstar 1}
		\dist(w^{0,0},\calW^\star) = \sqrt{\dist(x^{0,0},\calX^\star)  + \dist(s^{0,0},\calS^\star) } = \sqrt{\dist(0,\calX^\star)  + \dist(c,\calS^\star) }  \ .
	\end{equation}
	Let $\hat{s} \in \arg\min_{s\in\calS^\star} \| {s}\|$.  Then we have
	$$ \| \hat s \|^2 =  \| (\hat s -c) + c \|^2 = \| (\hat s -c)\|^2 + \| c \|^2 \ge \| (\hat s -c)\|^2 \ ,
	$$where the second equality follows since  $\hat{s} - c \in \linVd$ and $c \in \linVd^\bot$. Therefore
	$\dist(c,\calS^\star) \le \|c-\hat s\| \le \| \hat s \| = \dist(0,\calS^\star)$. Substituting this inequality into \eqref{eq bound w to wstar 1} yields
	\begin{equation}\label{eq bound w to wstar 3}
		\dist(w^{0,0},\calW^\star) \le \sqrt{\dist(0,\calX^\star)  + \dist(0,\calS^\star) } = \dist(0,\calW^\star)   \ .
	\end{equation}\end{proof}

Regarding the objective function error $\eobj(w)$, the following lemma presents a useful upper bound.  Let $\calF_p$ and $\calF_d$ denote the feasible regions of \eqref{pro: general primal clp} and \eqref{pro: general dual clp}, respectively.
\begin{lemma}\label{lm upper bounf of eobj}
	For any $w = (x,s) \in \mathbb{R}^{2n}$, it holds that
	\begin{equation}\label{eq lm upper bounf of eobj}
		\eobj(w) \le 2\left(\max\{0,\gap(w)\} + \min_{\hat{x}\in \calF_p} |c^\top x - c^\top \hat{x}| + \min_{\hat{s}\in \calF_d} |q^\top s - q^\top  \hat{s}| \right) \ .
	\end{equation}
\end{lemma}
\begin{proof} For ease of exposition, let $E(w)$ denote the right-hand side of \eqref{eq lm upper bounf of eobj}, and let $t_g$, $t_p$, and $t_d$ denote the three terms in the large parentheses in \eqref{eq lm upper bounf of eobj}, whereby $E(w) = 2(t_g + t_p + t_d)$.  Also, define $f_p:=c^\top x$ and $f_d:= q_0-q^\top  s$.  Then $\gap(w) = f_p - f_d$, and observe:
	\begin{equation}\label{mydean}
		f_p \le f^\star \ \Rightarrow \ t_p = f^\star - f_p  \ , \ \ \ \ \ \text{and} \ \ \ \ \ f_d \ge f^\star \ \Rightarrow \ t_d =  f_d - f^\star \ .
	\end{equation}
	There are six cases to consider that depend on the ordering of $f_p$, $f_d$, and $f^\star$.
	\begin{enumerate}
		\item Suppose $f_d \le f^\star \le f_p$. Then $\eobj(w) = | f_p - f^\star | + | f^\star - f_d  | = f_p - f_d = \gap(w) \le 2\gap(w) \le E(w)$, thus showing \eqref{eq lm upper bounf of eobj} in this case.
		\item Suppose $f_p \le f^\star \le f_d$. It then follows from \eqref{mydean} that
		      $$ \eobj(w) = f^\star - f_p + f_d - f^\star = t_p + t_d \le 2(t_p + t_d) \le E(w) \ , $$ thus showing \eqref{eq lm upper bounf of eobj} in this case.
		\item Suppose $f_p \le f_d \le f^\star$. Then
		      $$\eobj(w) = | f_p - f^\star | + | f^\star - f_d  | \le 2 | f_p - f^\star | = 2 t_p \le E(w) \ , $$
		      where the last equality follows from \eqref{mydean}, thus showing \eqref{eq lm upper bounf of eobj} in this case.
		\item Suppose $f_d \le f_p \le f^\star$.  Then  $\gap(w) = f_p - f_d \ge 0$ and thus $t_g = \max\{0,\gap(w)\} = \gap(w)$ and also $t_p = f^\star - f_p$ (from \eqref{mydean}). Therefore
		      $$\eobj(w) = | f_p - f^\star | + | f^\star - f_d  | \le 2| f^\star - f_d  | = 2 (f^\star - f_d ) = 2(f^\star -f_p + f_p - f_d ) = 2(t_g + t_p) \le E(w) \ , $$thus showing \eqref{eq lm upper bounf of eobj} in this case.
		\item Suppose $f^\star \le f_p \le f_d$. Then
		      $$\eobj(w) = | f_p - f^\star | + | f^\star - f_d  | \le 2 | f_d - f^\star | = 2 t_d \le E(w) \ , $$
		      where the last equality follows from \eqref{mydean}, thus showing \eqref{eq lm upper bounf of eobj} in this case.
		\item Suppose $f^\star \le f_d \le f_p$.   Then  $\gap(w) = f_p - f_d \ge 0$ and thus $t_g = \max\{0,\gap(w)\} = \gap(w)$ and also $t_d = f_d - f^\star $ (from \eqref{mydean}). Therefore
		      $$\eobj(w) = | f_p - f^\star | + | f^\star - f_d  | \le 2| f^\star - f_p  | = 2 (f_p - f^\star ) = 2(f_p - f_d + f_d - f^\star) = 2(t_g + t_d) \le E(w) \ , $$thus showing \eqref{eq lm upper bounf of eobj} in this last and final case.\end{enumerate}\end{proof}

By using the above lemma, we now present an upper bound of $\eobj(w)$ using $\gap(w)$ and $\max\{\dist(w,K),\dist(w,V)\}$.
\begin{lemma}\label{lm:upperbound_eobj_using_regular_errors}Suppose that $ c \in \operatorname{Null}(A)$. For any $w = (x,s)\in \mathbb{R}^{2n}$, it holds that
	\begin{equation}\label{eq:lm:upperbound_eobj_using_regular_errors}
		\eobj(w)\le 2\left(\max\{0,\gap(w)\} + 4 \left(\inf_{\gamma>0}\frac{\gamma}{r_\gamma}\right)\cdot \max\{\dist(w,K),\dist(w,V)\} \right) \ .
	\end{equation}
\end{lemma}
\noindent Before proving Lemma \ref{lm:upperbound_eobj_using_regular_errors}, we state the following simple result that will be used in the proof thereof.
\begin{proposition}\label{lm:convexity_ineq}Let $f:[0,\infty)\mapsto \mathbb{R}$ and $g:[0,\infty)\mapsto \mathbb{R}$ be nonnegative convex functions for which $f(0) = g(0) $, and suppose that $f$ is linear. If there exists $u \ge 0$ such that $g(u)\ge f(u) $, then for any $v \ge u$ it holds that  $g(v) \ge f(v)$.
\end{proposition}
\begin{proof} Define $F(\cdot) := g(\cdot) - f(\cdot)$, and note that $F:[0,\infty)\mapsto \mathbb{R}$ is convex since $f(\cdot)$ is a linear function, also $F(0) = 0$, and $F(u)\ge 0$. From the convexity of $F$ we have for $v > u$ that $(\frac{u}{v})\cdot F(v)  = (\frac{u}{v})\cdot F(v) + (\frac{v - u}{v})\cdot F(0) \ge  F\left(	(\frac{u}{v})\cdot v + (\frac{v - u}{v})\cdot 0 	\right) = F(u)\ge 0$.  Hence $ g(v) -f(v) \ge 0$ for $v >u$. And when $v=u$ the result holds trivially.
\end{proof}
\noindent We now prove Lemma \ref{lm:upperbound_eobj_using_regular_errors}.
\begin{proof}[Proof of Lemma \ref{lm:upperbound_eobj_using_regular_errors}] Let $w = (x,s)\in \mathbb{R}^{2n}$ be given. In light of Lemma \ref{lm upper bounf of eobj}, it suffices to prove for all $\gamma > 0 $ that
	\begin{equation}\label{georgiadean}
		\min_{\hat{x}\in \calF_p} |c^\top x - c^\top \hat{x}| + \min_{\hat{s}\in \calF_d} |q^\top s - q^\top  \hat{s}|  \le \left(\frac{4\gamma}{r_\gamma}\right)\cdot \max\{\dist(w,K),\dist(w,V)\}  \ .
	\end{equation}
	Let $\bar{w}=(\bar{x},\bar{s}) := P_{V}(w)$, and it follows that $\bar{x} := P_{V_p}(x)$ and $\bar{s}:= P_{V_d}(s)$. Then because $c \in \operatorname{Null}(A) = \linVp$ and $q\in\operatorname{Im}(A^\top) = \linVd$, it follows that $c^\top \bar{x} = c^\top x$ and $q^\top \bar{s} = q^\top s$, whereby $\gap(\bar{w}) = \gap(w)$.

	Let $\gamma >0$ be given.  Let us first suppose that $\bar w \in K$, whereby we have $\eobj(w) =\eobj(\bar w) =\gap(\bar w)  \le 2\max\{0, \gap(w)\}  $, from which \eqref{eq:lm:upperbound_eobj_using_regular_errors} follows directly.  We henceforth assume that $\bar w \notin K$.

	We first consider the case where $\gamma \ge \gap(\bar{w})$. Let $w_\gamma$ and $r_\gamma$ be the conic center and conic radius of $\calW_\gamma$ (recall Definition \ref{def radius}).  Then from Lemma \ref{lm error bound R r} we have:
	\begin{equation}\label{eq lm c123 of e obj 1}
		\frac{\| \bar{w} - \calF(\bar{w};w_\gamma) \|}{\dist(\bar{w}, K)} \le \frac{\| w_\gamma - \calF(\bar{w};w_\gamma) \|}{r_\gamma} \ ,
	\end{equation}
	and let us denote $\tilde{w} = (\tilde{x},\tilde{s}) := \calF(\bar{w};w_\gamma)$. Since $(\tilde{x},\tilde{s})$ lies in the relative interior of the line segment between $\bar{w}=(\bar{x},\bar{s})$ and $w_\gamma = (x_\gamma, s_\gamma)$, then if $\bar{x}\neq x_\gamma$ it holds that
	\begin{equation}\label{eq lm c123 of e obj 1-0}
		\frac{\|\bar{w} - \tilde{w}\|}{\|\tilde{w} - w_\gamma\|} = \frac{\|\bar{x} - \tilde{x}\|}{\|\tilde{x}-x_\gamma\|} \ .
	\end{equation}

	We now establish an upper bound on $|c^\top \bar{x} - c^\top \tilde{x}| $. When $\bar{x} = x_\gamma$, then  $\bar{x} = \tilde{x} = x_\gamma$ and $|c^\top \bar{x} - c^\top \tilde{x}| =0$. When $\bar{x} \neq x_\gamma$, then since $\bar{x}, \tilde{x}$ and $x_\gamma$ are all on the same line segment it holds that 	\begin{equation}\label{eq lm c123 of e obj 2}
		|c^\top \bar{x} - c^\top \tilde{x}|   = |c^\top x_\gamma - c^\top \tilde{x}|  \cdot \frac{\| \bar{x} - \tilde{x}\|}{\| \tilde{x} - x_\gamma\|}  =  |c^\top x_\gamma - c^\top \tilde{x}|\cdot\frac{\|\bar{w} -  \tilde{w}\|}{\|\tilde{w} - w_\gamma\|}
	\end{equation}
	where the second equality uses \eqref{eq lm c123 of e obj 1-0}.
	Note that because $\tilde{w}$ and $w_\gamma$ are both feasible and their duality gaps are no larger than $\gamma$, then $c^\top \tilde{x}$ and $c^\top x_\gamma$ lie in the interval $[ f^\star, f^\star + \gamma]$, from which it follows that $ |c^\top x_\gamma - c^\top \tilde{x}| \le \gamma$. Then \eqref{eq lm c123 of e obj 2} implies that $|c^\top \bar{x} - c^\top \tilde{x}|   \le \gamma\cdot \frac{\| \bar{w} - \tilde{w}\|}{\| \tilde{w} - w_\gamma\|}$. And using \eqref{eq lm c123 of e obj 1} we obtain:
	\begin{equation}\label{eq lm c123 of e obj 4}
		\begin{aligned}
			|c^\top \bar{x} - c^\top \tilde{x}| & \le \gamma\cdot \frac{\dist(\bar{w},K)}{r_\gamma} \, .
		\end{aligned}
	\end{equation}
	Using identical logic applied to the dual variable yields:
	\begin{equation}\label{eq lm c123 of e obj 5}
		\begin{aligned}
			|q^\top \tilde{s} - q^\top  \bar{s}| & \le \gamma\cdot \frac{\dist(\bar{w},K)}{r_\gamma} \, .
		\end{aligned}
	\end{equation}

	Combining  \eqref{eq lm c123 of e obj 4} and \eqref{eq lm c123 of e obj 5} we arrive at:
	\begin{equation}\label{eq lm c123 of e obj 8}
		\min_{\hat{x}\in \calF_p} |c^\top \bar x - c^\top \hat{x}| + \min_{\hat{s}\in \calF_d} |q^\top \bar s - q^\top  \hat{s}|
		\le 	|c^\top \bar{x} - c^\top \tilde{x}|  + 	|q^\top  \tilde{s} - q^\top  \bar{s}| \le \frac{2\gamma}{r_\gamma}\cdot \dist(\bar{w},K)  \ .
	\end{equation}
	Next observe that $\dist(\bar{w},K) \le \dist(w,K) + \|w - \bar{w}\| \le 2 \max\{\dist(w,K),\dist(w,V)\}$, and also recall that $c^\top \bar{x} = c^\top x$ and $q^\top \bar{s} = q^\top s$, which then implies with \eqref{eq lm c123 of e obj 8} that
	\begin{equation}\label{eq lm c123 of e obj 9}
		\min_{\hat{x}\in \calF_p} |c^\top  x - c^\top \hat{x}| + \min_{\hat{s}\in \calF_d} |q^\top  s - q^\top  \hat{s}|
		\le 	 \frac{4\gamma}{r_\gamma}\cdot \max\{\dist(w,K),\dist(w,V)\} \ ,
	\end{equation}
	thus demonstrating \eqref{georgiadean}, and hence \eqref{eq:lm:upperbound_eobj_using_regular_errors},  in the case when $\gamma \ge \gap(w)$.

	We now turn our attention to the case when $\gamma < \gap(w)$, and here we will invoke Proposition \ref{lm:convexity_ineq} to complete the proof. Let $w^\star = \arg\min_{\hat{w}\in\calW^\star}\|\hat{w} - w\|$ and define  $w_t := w^\star + t\cdot (w - w^\star)$ for $t \in [0, \infty)$. Then define the following functions of $t$ :
	$$
		f(t):= \eobj(w_t) \ , \text{ and   } \ g(t):= 2\left(\max\{0,\gap(w_t)\} +  \frac{4\gamma}{r_\gamma}\cdot \max\{\dist(w_t,K),\dist(w_t,V)\} \right) \ .
	$$
	Then $f(t)$ is a nonnegative linear function on $[0,\infty)$, and $f(0) = 0$, and $g(t)$ is a nonnegative convex function on $[0,\infty)$, and $g(0) = 0$. In addition, because $\gap(\cdot)$ is a linear function and $\gap(w_t) = t \cdot \gap(w)$, then setting $u := \gamma/\gap(w) < 1$ we obtain $\gap(w_u) = u \cdot \gap(w) = \gamma$. We can then invoke \eqref{eq:lm:upperbound_eobj_using_regular_errors} using $w_u$ in the place of $w$, which shows that $g(u) \ge f(u)$. Now it follows from Proposition \ref{lm:convexity_ineq} with $v=1 > u $ that $g(1) \ge f(1)$ which is
	\begin{equation}\label{eq lm c123 of e obj 10}
		\eobj(w) \le 2 \left(\max\{0,\gap(w)\} +  \frac{4\gamma}{r_\gamma}\cdot \max\{\dist(w,K),\dist(w,V)\}\right) \ .
	\end{equation}
	This inequality proves the result in the case when $\gamma < \gap(w)$, completing the overall proof.
\end{proof}

\begin{lemma}\label{lm use gap to bound error in w}
	Suppose the step-sizes $\tau,\sigma$ satisfy \eqref{eq:general_stepsize} and the starting point is $z^{0,0} =(x^{0,0},y^{0,0}) := (0,0)$ and so $w^{0,0}= (x^{0,0},s^{0,0}) = (0,c)$. Then for any outer iteration value $n \ge 1$ it holds that:
	\begin{align}
		 & \max\left\{\dist(w^{n,0},V), \dist(w^{n,0},K)\right\} \le c_0 \cdot \rho(\| z^{n,0} - z^{n-1}\|_M \ ,z^{n,0} ) \, ; \label{ineq translate feasibility error}           \\
		 & \gap(w^{n,0}) \le
		2\sqrt{2}c_0\cdot  \dist(0,\calW^\star)
		\cdot  \rho(\| z^{n,0} - z^{n-1}\|_M ; z^{n,0} ) \ , \text{ and}
		\label{ineq translate gap error}                                                                                                                                          \\
		 & \eobj(w^{n,0})\le 14 c_0 \cdot \left( \inf_{\gamma >0}\frac{\gamma}{r_\gamma}\right) \cdot  \rho(\| z^{n,0} - z^{n-1}\|_M ; z^{n,0} )  \ , \label{ineq translate eobj}
	\end{align} where $c_0$ is defined in \eqref{icemelt}.
\end{lemma}
\begin{proof}
	The inequality \eqref{ineq translate feasibility error} follows directly from items ({\it \ref{item_gap1}.})  and ({\it \ref{item_gap2}.}) of Lemma \ref{lm: convergence of PHDG without restart}. Towards the proof of \eqref{ineq translate gap error}, note that $z^a := z^{0,0}$, $z^b:=z^{n,0}$, and $z^c:= z^{n-1,0}$ satisfy the nonexpansive conditions of Lemma \ref{lm: R in the opt gap convnergence} whereby it follows from Lemma \ref{lm: R in the opt gap convnergence} that
	$$
		\max\{ \|z^{n,0} - z^{n-1,0}\|_M, \|z^{n,0}\|_M\}  \le    2 \cdot \dist_M(z^{0,0},\calZ^\star)  + \|z^{0,0} \|_M = 2 \cdot \dist_M(z^{0,0},\calZ^\star) \ ,
	$$
	where the equality follows since $z^{0,0} = 0$. And applying Lemma \ref{lm change of norm} with $\calX := \calX^\star$ and $\calY := \calY^\star$, together with Proposition \ref{lm:distance_to_optimal_initial}, we obtain	$$
		\dist_M(z^{0,0},\calZ^\star)  \le \sqrt{2}c_0 \cdot \dist(w^{0,0}, \calW^\star)  \le \sqrt{2}c_0 \cdot \dist(0, \calW^\star)  \ .
	$$
	Inequality \eqref{ineq translate gap error} then follows directly from the above two inequalities and ({\it \ref{item_gap3}.}) of Lemma \ref{lm: convergence of PHDG without restart} using $r:= \|z^{n,0} - z^{n-1,0}\|_M$.

	Substituting \eqref{ineq translate feasibility error} and \eqref{ineq translate gap error} into \eqref{eq:lm:upperbound_eobj_using_regular_errors} of Lemma \ref{lm:upperbound_eobj_using_regular_errors} yields
	$$
		\eobj(w^{n,0})\le \left(4\sqrt{2}c_0 \cdot \dist(0,\calW^\star)+ 8c_0 \cdot \inf_{\gamma >0}\frac{\gamma}{r_\gamma} \right)\cdot  \rho(\| z^{n,0} - z^{n-1}\|_M ; z^{n,0} ) \ .
	$$
	Then applying the second inequality of \eqref{xcski}, we obtain $\dist(0,\calW^\star) \le \max_{w\in\calW^\star}\|w\|\le \gamma/r_\gamma$ for any $\gamma > 0$, from which \eqref{ineq translate eobj} follows since $4\sqrt{2} + 8 \le 14$, which completes the proof.
\end{proof}

The next lemma states that the distance to optimal solutions can be bounded by terms involving the distance to constraints and the Hausdorff distance $d^H_\delta$.

\begin{lemma}\label{lm efeas bound dist}
	Suppose that $c\in\operatorname{Null}(A)$. For any  $w = (x,s) \in \mathbb{R}^{2n}$,  and any $\delta>0$, it holds that
	\begin{equation}\label{eq gap small bound}
		\dist(w,\calW^\star) \le \frac{3 D_\delta}{r_\delta} \cdot  \max\{\dist(w,V),\dist(w,K)\} + \frac{d^H_\delta}{\delta} \cdot \max\{\gap(w),\delta\} \ .
	\end{equation}
\end{lemma}

\begin{proof}
	We first consider the case where $\delta \ge  \gap(w)$. Define $\hat{w} := P_{V}(w) $, whereby $\|w - \hat{w}\| = \dist(w,V)$. Then because $c \in \operatorname{Null}(A) = \linVp$ and $q \in \operatorname{Im}(A^\top) = \linVd$, then $\gap(w) = \gap(\hat{w})$, therefore $\gap(\hat{w}) \le \delta$ as well. We have:
	\begin{equation}\label{eq error bound distance clp 1}
		\dist(w, \calW^\star) \le \dist(\hat{w}, \calW^\star) + \| \hat{w}- w\|  = \dist(\hat{w}, \calW^\star) +\dist(w, V) \ ,
	\end{equation} and from the definition of $d^H_\delta$ we also have:
	\begin{equation}\label{eq error bound distance clp 2}
		\dist(\hat{w},\calW^\star) \le \dist(\hat{w}, \calW_\delta) + d^H_\delta \ .
	\end{equation}
	Since $\hat{w}\in V$ and $\gap(\hat{w})\le \delta$, from Lemma \ref{lm error bound R r} it follows that $ \dist(\hat{w}, \calW_\delta) $ in \eqref{eq error bound distance clp 2} can be bounded as follows:
	\begin{equation}\label{eq error bound distance clp 3}
		\begin{aligned}
			\dist(\hat{w}, \calW_\delta) & \le \frac{D_\delta}{r_\delta} \cdot \dist(\hat{w},K)
			\le \frac{D_\delta}{r_\delta} \cdot\left(\|w - \hat{w}\| + \dist(w,K)    \right)
			\\
			                             & \le  \frac{2 D_\delta}{r_\delta} \cdot \max\{\dist(w,V),\dist(w,K)\}  \ .
		\end{aligned}
	\end{equation}
	Then combining \eqref{eq error bound distance clp 3}, \eqref{eq error bound distance clp 2}, and \eqref{eq error bound distance clp 1} yields
	$$
		\dist(w,\calW^\star) \le \frac{2 D_\delta}{r_\delta} \cdot  \max\{\dist(w,V),\dist(w,K)\} +\dist(w , V) + d^H_\delta \ .
	$$
	Last of all note that $D_\delta \ge r_\delta$ from Lemma \ref{lm D ge r}, from which the above inequality then implies \eqref{eq gap small bound}.

	Let us now consider the case where  $\delta \le \gap(w)$.  Here we will make use of Proposition \ref{lm:convexity_ineq} to complete the proof.  Let $w^\star = P_{\calW^\star}(w) = \arg\min_{\bar{w}\in\calW^\star}\|\bar{w} - w\|$ and define $w_t := w^\star + t\cdot (w - w^\star)$ for $t\in[0,\infty)$. Then define the following functions of $t$ :
	$$
		f(t):= \dist(w_t,\calW^\star) \ ,  \text{ and  }g(t):= \frac{3 D_\delta}{r_\delta} \cdot  \max\{\dist(w_t,V),\dist(w_t,K)\} + \frac{d^H_\delta}{\delta} \cdot \max\{\gap(w_t),\delta\} \ .
	$$
	Then $f(t)$ is a nonnegative  linear  function on $[0,\infty)$, and $f(0) = 0$. And $g(t)$ is convex and nonnegative on $[0,\infty)$, and $g(0) = 0$. In addition, because $\gap(\cdot)$ is a linear function and $\gap(w_t) = t \cdot \gap(w)$, then setting $u := \delta / \gap(w)$ we obtain $ \gap(w_u) = u\cdot \gap(w) = \delta$. We can then invoke \eqref{eq gap small bound} using $w_u$ in the place of $w$, which yields $g(u) \ge f(u)$. Now it follows from Proposition \ref{lm:convexity_ineq} with $v := 1\ge  u$ that  $g(1) \ge f(1)$, which is precisely \eqref{eq gap small bound} in the case $\delta \le \gap(w)$, and completes the proof.
\end{proof}

\begin{lemma}\label{thm L C}
	Suppose that $c\in\operatorname{Null}(A)$. Under Assumption \ref{assump:striclyfeasible}, suppose that Algorithm \ref{alg: PDHG with restarts} (rPDHG) is run starting from $z^{0,0} = (x^{0,0},y^{0,0}) = (0,0)$, and the step-sizes $\sigma$ and $\tau$ satisfy the step-size inequality \eqref{eq:general_stepsize}. Then for all $n\ge 1$ and any $\delta > 0$, it holds that
	\begin{equation}\label{eq thm L C}
		\dist_M(z^{n,0},\calZ^\star) \le 8.25 \cdot c_0^2  \cdot \frac{D_\delta}{r_\delta} \cdot \rho(\| z^{n,0} - z^{n-1,0}\|_M; z^{n,0}) + \sqrt{2}c_0 d^H_\delta \ .
	\end{equation}
\end{lemma}

\begin{proof}
	Applying Lemmas \ref{lm change of norm} and \ref{lm efeas bound dist} directly yields:
	\begin{equation}\label{eq L C 2}
		\begin{aligned}
			\dist_M(z^{n,0},\calZ^\star) & \ \le  \sqrt{2}c_0\cdot \dist(w^{n,0},\calW^\star)
			\\
			                             & \
			\le  \frac{3\sqrt{2} c_0 D_\delta}{r_\delta}\cdot  \max\{\dist(w^{n,0},V),\dist(w^{n,0},K)\}+  \frac{\sqrt{2}c_0 d^H_\delta}{\delta} \cdot \max\{\gap(w^{n,0}),\delta\} \ .
		\end{aligned}
	\end{equation}

	\noindent We consider two cases, depending on whether $\delta \ge \gap(w^{n,0})$ or $\delta < \gap(w^{n,0})$. We first consider the case where $\gap(w^{n,0})\ge \delta$.  From Lemma \ref{lm use gap to bound error in w} and \eqref{eq L C 2} it follows that
	\begin{equation}\label{eq L C 3}
		\dist_M(z^{n,0},\calZ^\star)
		\le \left( \frac{3\sqrt{2} c_0^2 D_\delta}{r_\delta} +  \frac{4c_0^2 d^H_\delta}{\delta}
		\cdot   \dist(0,\calW^\star)
		\right)
		\cdot  \rho(\| z^{n,0} - z^{n-1}\|_M ; z^{n,0} ) \ .
	\end{equation}
	Furthermore, applying the second inequality of \eqref{xcski} yields $
		\dist(0,\calW^\star) \le \max_{w\in\calW^\star}\|w\| \le \delta/r_\delta$ for any $\delta > 0$, and  substituting this into \eqref{eq L C 3} yields
	\begin{equation}\label{eq L C 4}
		\dist_M(z^{n,0},\calZ^\star)
		\le \left( \frac{3\sqrt{2} c_0^2 D_\delta}{r_\delta} +  \frac{4c_0^2 d^H_\delta}{r_\delta}
		\right)
		\cdot  \rho(\| z^{n,0} - z^{n-1}\|_M ; z^{n,0} ) \ .
	\end{equation}
	We can also have $d^H_\delta \le D_\delta$ and $3\sqrt{2} + 4 \le 8.25$, which when combined with \eqref{eq L C 4} shows that
	\begin{equation}\label{eq L C 5}
		\begin{aligned}
			\dist_M(z^{n,0},\calZ^\star) \le 8.25 \cdot c_0^2  \cdot \frac{D_\delta}{r_\delta} \cdot \rho(\| z^{n,0} - z^{n-1,0}\|_M; z^{n,0}) \ ,
		\end{aligned}
	\end{equation}
	which proves the result in this case.

	Let us now consider the case where $\gap(w^{n,0}) < \delta$. From Lemma \ref{lm use gap to bound error in w} and \eqref{eq L C 2} it follows that
	\begin{equation}\label{eq L C 6}
		\begin{aligned}
			\dist_M(z^{n,0},\calZ^\star) \le 3\sqrt{2} c_0^2 \cdot \frac{ D_\delta}{r_\delta} \cdot  \rho(\| z^{n,0} - z^{n-1,0}\|_M; z^{n,0}) + \sqrt{2}c_0 d^H_\delta \ ,
		\end{aligned}
	\end{equation}
	which proves the result in this case.	Depending on the case, we obtain either \eqref{eq L C 5} or \eqref{eq L C 6}, either of which implies \eqref{eq thm L C}.
\end{proof}

\subsubsection{Proof of Theorem \ref{thm overall complexity clp}}\label{gillianwelch}

At last we prove Theorem \ref{thm overall complexity clp}.
\begin{proof}[Proof of Theorem \ref{thm overall complexity clp}]
	From Lemma \ref{lm use gap to bound error in w} it follows that $w^{n,0}$ satisfies the $\eps$-tolerance requirement (Definition \ref{eq accuracy requirement clp}) if
	\begin{equation}\label{eq overall complexity clp 1}
		\rho(\| z^{n,0} - z^{n-1}\|_M ; z^{n,0} ) \le \min\left\{
		\frac{ \eps_{\mathrm{cons}}}{c_0},\
		\frac{ \eps_{\mathrm{gap}}}{2\sqrt{2} c_0 \cdot  \dist(0,\calW^\star) },\
		\frac{\eps_{\mathrm{obj}}}{14c_0\cdot\inf_{\gamma>0}\left(\frac{\gamma}{r_\gamma}\right)}
		\right\} \ .
	\end{equation}
	Now notice that the right-hand-side term is equal to $\frac{\err}{c_0}$ where $\err$ is defined in \eqref{def eq calN}.  Also, it follows from the choice of step-sizes in the theorem and the definition of $c_0$ in \eqref{icemelt} that $c_0 = \sqrt{\kappa}$.

	In the proof of Lemma \ref{thm L C} we see that \eqref{eq thm L C} holds for any $\delta > 0$, so Theorem \ref{thm: complexity of PDHG with adaptive restart} can be applied since the condition \eqref{eq restart L C condition} is satisfied using \eqref{eq thm L C} with $\calL = \frac{8.25 c_0^2 D_\delta}{r_\delta}$ and $\calC = \sqrt{2} c_0 d^H_\delta$. Therefore it follows from Theorem \ref{thm: complexity of PDHG with adaptive restart} that $T$ satisfies
	\begin{equation}\label{eq overall complexity clp 2}\begin{aligned}
			T & \le 23 \left(8.25 \cdot c_0^2\cdot \frac{D_\delta}{r_\delta} \right) \cdot  \ln\left(\frac{23 c_0\cdot \dist_M(z^{0,0},\calZ^\star)}{\err}\right)  +  \frac{35 \sqrt{2}c_0^2d^H_\delta}{\err}                                                                                                                                                                                          \\ \\
			  & = 23  \left(8.25 \cdot \kappa \cdot \frac{D_\delta}{r_\delta} \right) \cdot  \ln\left(\frac{23 \sqrt{\kappa} \cdot \dist_M(z^{0,0},\calZ^\star)}{\err} \right)  
			+  \frac{35 \sqrt{2}\cdot \kappa \cdot d^H_\delta}{\err}                                                                                                                                                                            \\ \\
			  & \le  23 \left(8.25  \kappa \cdot \frac{D_\delta}{r_\delta} \right) \cdot  \ln\left(\frac{23\sqrt{2} \kappa  \cdot \dist(0,\calW^\star)}{\err} \right)  
			+  \frac{35\sqrt{2}  \kappa \cdot d^H_\delta}{\err} \ ,
		\end{aligned}\end{equation}
	where the equality substitutes in the values of $c_0$, and the second inequality uses $\dist_M(z^{0,0},\calZ^\star) \le  \sqrt{2}c_0 \cdot \dist(w^{0,0},\calW^\star)\le  \sqrt{2}c_0 \cdot \dist(0,\calW^\star) = \sqrt{2\kappa}\cdot \dist(0,\calW^\star)$ from Lemma \ref{lm change of norm} and Proposition \ref{lm:distance_to_optimal_initial}.  Last of all, notice that $23 \times 8.25\le 190$, $23\sqrt{2}\le 33$ and $35\sqrt{2} \le 50$, and substituting this back into \eqref{eq overall complexity clp 2} yield  \eqref{eq overall complexity} and completes the proof.\end{proof}

\section{Linear Convergence of rPDHG for Linear Programming}\label{sec:complexity_lp}

In the case of linear programming problems (namely instances of \eqref{pro: general primal clp} wherein $K_p = \mathbb{R}^n_+$) we now show a global linear convergence bound that structurally improves on the bound in \cite{xiong2023computational}.  Our analysis uses the ``best suboptimal extreme point gap'' $\bar \delta$ whose formal definition we now state.\medskip
\begin{definition}[Best suboptimal extreme point gap]\label{def:best_suboptimal_gap}
	Let $\ep_\calF$ denote the set of extreme points of $\calF$. The best suboptimal extreme point gap $\bar{\delta}$ is defined as follows:
	\begin{equation}\label{eqdef:best_suboptimal_gap}
		\bar{\delta}:= \left\{
		\begin{array}{ll}
			\min\{\gap(w): w\in \ep_\calF\setminus \calW^\star\} & \quad \text{ if $\ep_\calF\setminus \calW^\star \neq \emptyset$ } \\
			+\infty                                              & \quad \text{ if $\ep_\calF\setminus \calW^\star = \emptyset$ .}   \\
		\end{array}
		\right.
	\end{equation}
\end{definition}
\noindent We now present our global linear convergence result for rPDHG for linear programming problem instances.
\begin{theorem}\label{thm overall complexity lp} Suppose that
	\eqref{pro: general primal clp} is a linear programming instance ($K_p = \mathbb{R}^n_+$) and $c\in\operatorname{Null}(A)$, and let $\bar\delta$ be as defined in \eqref{eqdef:best_suboptimal_gap}. Under Assumption \ref{assump:striclyfeasible}, suppose that Algorithm \ref{alg: PDHG with restarts} (rPDHG) is run starting from $z^{0,0} = (x^{0,0},y^{0,0} ) = (0,0)$ using the $\beta$-restart condition with $\beta := 1/e$, and the step-sizes are chosen as follows:
	\begin{equation}\label{eq:stepsize_in_theorem again}
		\tau = \frac{1}{\kappa}  \ \ \ \mathrm{and} \ \ \  \sigma = \frac{1}{\lambda_{\max}\lambda_{\min}} \ .
	\end{equation}
	Given $\eps := (\eps_{\mathrm{cons}}, \eps_{\mathrm{gap}},\eps_{\mathrm{obj}})\in\mathbb{R}_{++}^3$, let $T$ be the total number of \textsc{OnePDHG} iterations that are run in order to obtain $n$ for which $w^{n,0}=(x^{n,0},s^{n,0})$ satisfies the $\eps$-tolerance requirement (Definition \ref{eq accuracy requirement clp}).
	Then
	\begin{equation}\label{eq overall complexity LP}
		T \le 255 \kappa \cdot \left( \min_{0 < \delta \le \bar{\delta}}\frac{D_{\delta}}{r_{\delta}}\right) \cdot   \left[\ln\big(33 \kappa \cdot \dist(0,\calW^\star) \big) + \ln\left(\frac{1}{\err}\right) \, \right] \,  ,
	\end{equation}
	in which $\err$ is the weighted minimum of the target tolerance values $\eps$:
	\begin{equation}\label{def eq calN lp}
		\err := \min\left\{\eps_{\mathrm{cons}}, \
		\tfrac{\sqrt{2}}{4 \cdot \dist(0,\calW^\star)}\cdot \eps_{\mathrm{gap}}, \
		\tfrac{1}{14}  \left( \displaystyle\sup_{\gamma  > 0} \tfrac{ r_\gamma}{\gamma} \right) \cdot \eps_{\mathrm{obj}}
		\right\} \ .
	\end{equation}
\end{theorem}

Similar to Theorem \ref{thm overall complexity clp}, in Theorem \ref{thm overall complexity lp} we still use $\err$ for expositional convenience, and $\err$ defined in \eqref{def eq calN lp} is identical to that defined in \eqref{def eq calN}. But unlike the bound in Theorem \ref{thm overall complexity clp}, the computational bound \eqref{eq overall complexity LP} in Theorem \ref{thm overall complexity clp}
depends on $\err^{-1}$ only in the logarithm term. This means that if there exists a $\delta$-sublevel set for some $\delta \in (0,\bar{\delta}]$ such that the corresponding ratio $\frac{D_{\delta}}{r_{\delta}}$ is not too large, then rPDHG will have fast linear convergence for that LP instance.  Appendix \ref{taxtime} contains the proof of Theorem \ref{thm overall complexity lp}.

Theorem \ref{thm overall complexity lp} provides a stronger linear convergence guarantee than Corollary \ref{cor: linear convergence complexity clp} because $\left( \min_{0 < \delta \le \bar{\delta}}\frac{D_{\delta}}{r_{\delta}}\right) < \infty$ (which follows from Assumption \ref{assump:striclyfeasible}) even if there exist multiple optimal solutions. Consider as an example the LP problem \eqref{pro:example_lp_validate} with $\nu=0$. This instance has multiple primal optimal solutions, namely $x^\star =(0, \alpha, 1-\alpha)$ is primal optimal for all $\alpha \in [0,1]$. The left plot in Figure \ref{fig:two_thms_easy_hard_LP_paper} shows the iteration upper bounds from Theorems \ref{thm overall complexity clp} and \ref{thm overall complexity lp}, and the actual number of iterations of rPDHG that were needed to satisfy the $(\bar{\eps},\bar{\eps},\bar{\eps})$-tolerance requirement, for the LP problem \eqref{pro:example_lp_validate} with $\nu=0$.  The right plot of Figure \ref{fig:two_thms_easy_hard_LP_paper} shows the same results for the LP problem \eqref{pro:example_lp_validate} with $\nu=10^{-4}$. Examining the LP problem \eqref{pro:example_lp_validate} with $\nu=0$ and the left plot in the figure, we observe that the bound in Theorem \ref{thm overall complexity clp} itself is not revealing linear convergence, as a consequence of the fact that \eqref{pro:example_lp_validate} has multiple primal optimal solutions for $\nu = 0$. Examining the LP problem \eqref{pro:example_lp_validate} with $\nu=10^{-4}$ and the right plot in the figure, we observe the linear convergence rate in Theorem \ref{thm overall complexity lp}.  But also observe that the linear convergence rate does not correspond to the early-stage performance of rPDHG in practice, and for the early stage in fact the bound in Theorem \ref{thm overall complexity clp} much more closely corresponds to the actual performance of rPDHG for this instance.

\begin{figure}[htbp]
	\centering
	\includegraphics[width=0.65\linewidth]{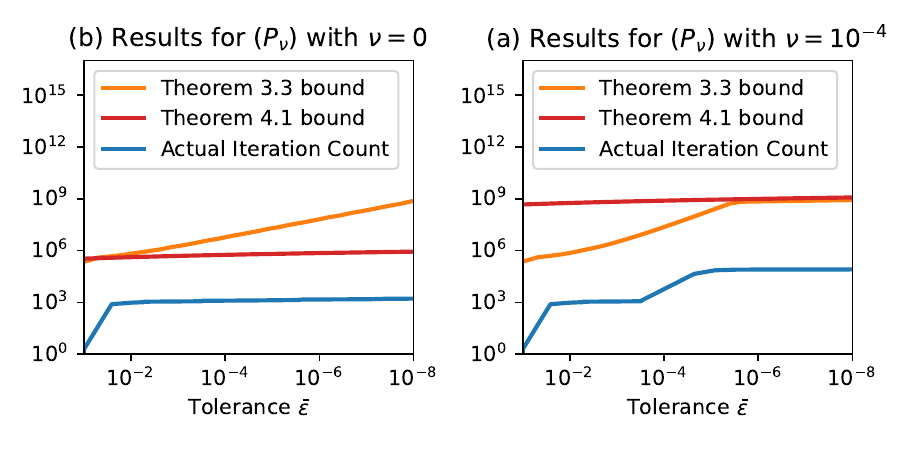}\vspace{-10pt}
	\caption{Plots of the guaranteed bounds on the number of iterations of rPDHG and the actual number of iterations of rPDHG for LP problem \eqref{pro:example_lp_validate} with $\nu=0$ (left) and $\nu=10^{-4}$ (right).}\label{fig:two_thms_easy_hard_LP_paper}
\end{figure}

It follows from \eqref{xcski} that $r_\delta \le \tfrac{\delta}{\max_{w\in\calW^\star}\|w\|}$ for any $\delta > 0$, so the term $\min_{0 < \delta \le \bar{\delta}} \frac{D_{\delta}}{r_{\delta}}$ in \eqref{eq overall complexity LP} has the following lower bound:
\begin{equation}
	\min_{0 < \delta \le \bar{\delta}}\frac{D_{\delta}}{r_{\delta}} \ge \min_{0 < \delta \le \bar{\delta}} \frac{D_\delta \cdot \max_{w\in\calW^\star}\|w\|}{\delta} \ .
\end{equation}
Now notice that both sides of the above inequality are decreasing in the value of $\bar{\delta}$. Therefore the iteration bound in Theorem \ref{thm overall complexity lp} is smaller when the value of the best suboptimal extreme point gap $\bar{\delta}$ is larger. This is consistent with what we observe in computational practice for the LP instances \eqref{pro:example_lp_validate} with $\nu=0$ and $\nu=10^{-4}$.  Table \ref{tabletable} shows the values of $\bar\delta$ as well as several other quantities of interest for these two LP instances. Notice that $\bar\delta$ is much larger in \eqref{pro:example_lp_validate} with $\nu=0$ compared to its value in \eqref{pro:example_lp_validate} with $\nu=10^{-4}$.  Furthermore, the two problems have fairly similar values of $D_{\bar{\delta}}$ and $\max_{w\in\calW^\star}\|w\|$ both of which contribute to the bound $\frac{D_{\bar{\delta}}}{r_{\bar\delta}}$ being much larger for \eqref{pro:example_lp_validate} with $\nu=10^{-4}$.  This is also reflected in the bounds and the actual iteration counts in Figure \ref{fig:two_thms_easy_hard_LP_paper}.

\begin{table}[htbp]
	\centering
	\begin{tabular}{l|cccccc}
		              & $\kappa$ & $\bar{\delta}$ & $D_{\bar{\delta}}$ & $r_{\bar{\delta}}$ & $\frac{D_{\bar{\delta}}}{r_{\bar{\delta}}}$ & {\small$\max_{w\in\calW^\star}\|w\|$} \\ \hline
		$\nu=0$       & 1.0e0    & 1.0e0         & 1.1e1              & 9.1e-2             & 1.2e2                                       & 1.0e1                                 \\
		$\nu=10^{-4}$ & 1.0e0    & 1.0e-4         & 1.4e0             & 8.9e-6             & 1.6e5                                       & 1.0e1
	\end{tabular}\vspace{-5pt}
	\caption{Values of some quantities of interest for LP instances \eqref{pro:example_lp_validate} with $\nu=0$ and $\nu=10^{-4}$. }\label{tabletable}

\end{table}

Moreover, a small value of $\bar{\delta}$ implies that the best suboptimal extreme point is nearly optimal and so the problem is ``close to'' having more optimal solutions (say, under small data perturbation). In this sense, Theorem \ref{thm overall complexity lp} can be interpreted as indicating that the performance of rPDHG is hurt when the problem is close to having more optimal solutions. However, unlike Theorem \ref{thm overall complexity clp}, having multiple optimal solutions does not necessarily hurt the linear convergence rate of rPDHG. This is similar to the observation of \cite{lu2023geometry} on the performance of PDHG without restarts, where they show that degeneracy in LP instances does not hurt the convergence rate, but being {\em close} to degeneracy does hurt the convergence rate. One difference between the computational guarantees in Theorem \ref{thm overall complexity lp} and in \cite{lu2023geometry} is that the bound in  \cite{lu2023geometry} depends on quantities involving the limiting points of the iterates of PDHG, which is not necessarily an inherent property of the problem instance as it might depend on the initial point.

We note that Theorem \ref{thm overall complexity lp} is not the first result of the linear convergence of PDHG on LP problems. \cite{applegate2023faster} uses the Hoffman constant of the KKT system to characterize the linear convergence, but the global Hoffman constant is often overly conservative and hard to analyze \cite{lu2023geometry,xiong2023computational}. \cite{lu2023geometry} studies PDHG without restarts applied to LP problems and uncovers a very nice ``two-phase'' behavior. Although the second phase yields linear convergence based on the Hoffman constant of a reduced system, the duration of the first phase is still not well understood. \cite{xiong2023computational,xiong2023relation} provide computational guarantees with linear convergence based on two condition measures of LP problems, ``limiting error ratios'' and LP sharpness. These condition measures are closely related to the inherent natural properties of the LP instances, but they can still take extreme values and lead to poor performance of PDHG in theory and in practice. In contrast, in addition to the simplicity of the linear convergence rate in Theorem \ref{thm overall complexity lp}, the bound in Theorem \ref{thm overall complexity lp} is based on the condition numbers related to sublevel-set geometry -- diameter, conic radius, and Hausdorff distance. As we will show in Section \ref{sec:geometric_enhancements}, the sublevel-set geometry can be improved by rescaling transformations, and in this way the theory in Theorem \ref{thm overall complexity lp} can lead to practical algorithm enhancements that can significantly improve the overall performance of rPDHG.

\section{How to Bound the Sublevel-Set Geometry using Central-Path Hessian-based Rescaling}\label{sec:geometric_enhancements}
In this section we show how to bound (and generally improve) the primal-dual sublevel-set geometry condition numbers -- $D_\delta$, $r_\delta$ and $d^H_\delta$ -- by doing a ``rescaling'' linear transformation of the variables based on Hessian information from a point on the central path of a self-concordant barrier. 
(Here the word ``rescaling'' is borrowed from the lexicon of interior-point literature for LP where nonnegative individual variables are rescaled.) A brief overview is as follows. From the theory of interior-point methods \cite{nesterov2018lectures,renegar2001mathematical}, it is known that the Hessian matrix of a point on the central path yields an ellipsoid that contains a given primal-dual sublevel set, and whose center is suitably interior to the feasible region.  Using a linear transformation based on the Hessian, one then can transform the primal-dual sublevel sets to easily bound the resulting sublevel-set geometry condition numbers in the transformed space, in such a way that sublevel-set geometry condition numbers in the transformed space only depend on the complexity value $\vartheta$ of the barrier and the duality gap of the central path point.  After rescaling, one can then apply rPDHG to the transformed problem, whose sublevel set condition numbers are more bounded/improved, yielding an improved iteration bound on rPDHG via Theorems \ref{thm overall complexity clp} and \ref{thm overall complexity lp}.

We begin by recalling some essential properties of self-concordant barrier functions and central-path solutions.

\subsection{Self-concordant barriers and central-path solutions}

A self-concordant barrier function for a cone $K_p$ is a special class of barrier functions defined on $\textsf{int}K_p$. This class of barrier functions plays a crucial role in Newton's method and in interior-point methods for convex optimization, see \cite{renegar2001mathematical}. Examples of (logarithmically homogeneous) self-concordant barrier functions include: $f(x):= -\sum_{i=1}^n\log(x_i)$ (whose complexity value is $\vartheta_f = n$) for $K_p = \mathbb{R}^n_+$, $f(x):=-\ln(x_{d+1}^2 - \|x_{1:d}\|^2)$ (whose complexity value is $\vartheta_f=2$) for $K_p = \mathbb{K}_{\textsf{soc}}^{d+1}$, and $f(X):= -\ln(\det(X))$ (whose complexity value is $\vartheta_f = d$) for $K_p = \mathbb{S}^{d\times d}_{+}$. A review of the definitions of (logarithmically homogeneous) self-concordant barrier functions is provided in Appendix \ref{appendix:geometric_enhancements}.

Let $f$ be a self-concordant barrier function for $K_p$, and let $H_f(x)$ denote its Hessian at $x$. The local norm of $f$ at $x$ is defined as $\|u\|_x := \|u\|_{H_f(x)} = \sqrt{\langle u,H_f(x) u \rangle}$ and the local-norm ball of radius $r$ centered at $x$ is $B_x(x, r) := \{\hat{x}: \|\hat{x} - x\|_x \le r\}$. Let $f^*$ denote the following conjugate function of $f$ defined by $f^*(s) := -\inf_{x \in K_p}\{s^\top x + f(x) \}$ (which is a slight variation of the standard definition of the conjugate function). Then $f^*$ is a self-concordant barrier function for the dual cone $K_d$, and $F(w) := f(x) + f^*(s)$ for $w = (x, s)$ is a self-concordant barrier function on the product cone $K$. We define the primal-dual central path on $K$ as follows.

\begin{definition}[Primal-Dual Central path]\label{italy}
	For $\eta >0$ define $w(\eta)$ as follows:
	\begin{equation}\tag{PD$_\eta$}\label{pro: general primaldual clp centralpath}
		w(\eta) :=	\arg\min_{w \in\mathbb{R}^{2n}}  \ \eta\cdot \gap(w) +F(w) \quad
		\operatorname{s.t.} \  w \in V_p \times V_d , \ w \in\emph{\textsf{int}} K_d\times \emph{\textsf{int}} K_p \ ,
	\end{equation}
	and the primal-dual central path is defined to be the set of parameterized solutions $\{w(\eta):\eta>0\}$.
\end{definition}
\noindent
Because $F(w) = f(x)+f^*(s)$, \eqref{pro: general primaldual clp centralpath} can be separated into primal and dual barrier problems, whose solutions $x(\eta)$ and $s(\eta)$ form the primal central path and dual central path, respectively, and $w(\eta) = (x(\eta),s(\eta))$.
The following are useful properties of $f$, $f^*$, and $F$ (see \cite{renegar2001mathematical} for proofs):
\begin{fact}\label{fact:SCfunction_main}
	The following properties hold for $f$, $f^*$ and $F$:
	\begin{enumerate}
		\item For any $x\in \emph{\textsf{int}}K_p$ it holds that $B_x(x,1)\subset K_p$. \label{item:fact:SCfunction_main:0}
		\item $f^*$ is a self-concordant barrier function for the dual cone $K_d$. Furthermore, if $f$ is logarithmically homogeneous, then $f^*$ is logarithmically homogeneous, and $\vartheta_{f^*}=\vartheta_f$ .\label{item:fact:SCfunction_main:1}
		\item $F(w)$ is a self-concordant barrier function for $K:=K_p\times K_d$, and $\vartheta_F = \vartheta_f + \vartheta_{f^*}$ . \label{item:fact:SCfunction_main:4}
	\end{enumerate}
\end{fact}\noindent Additionally, central-path solutions have the following useful properties.

\begin{fact}\label{fact:centralpath_main}
	The following properties hold for central-path solutions $w(\eta)=(x(\eta),s(\eta))$ defined in \eqref{pro: general primaldual clp centralpath}:
	\begin{enumerate}
		\item
		      For every $w \in \calF$ satisfying $\gap(w) \le \gap(w(\eta))$ it holds that $w \in B_{w(\eta)}(w(\eta),\vartheta_F + 2 \sqrt{\vartheta_F})$ \cite[Theorem 5.3.8]{nesterov2018lectures}. \label{item:fact:centralpath_main:3}
		\item If $f$ is logarithmically homogeneous, then $\gap(w(\eta))=\vartheta_f/\eta$ \cite[Section 3.4 and Equation (2.9)]{renegar2001mathematical}. \label{{item:fact:centralpath_main:2}}
	\end{enumerate}
\end{fact}

\noindent It follows from Fact \ref{fact:SCfunction_main} that $B_{w(\eta)}(w(\eta),1)$ is an inscribed ellipsoid of $K$ centered at $w(\eta)$, and from Fact \ref{fact:centralpath_main} that $B_{w(\eta)}(w(\eta),\vartheta_F + 2 \sqrt{\vartheta_F})$ circumscribes the sublevel set $\calW_{\gap(\eta)}$. Importantly, these two ellipsoids have the same center $w(\eta)$ and shape matrix $H_f({w(\eta)})$, and only differ in their scaling. To somewhat ease the notational burden, we use $H_{w(\eta)}:= H_F(w(\eta))$ to denote the Hessian of $F$ at $w(\eta)$. We have the following bounds on the condition numbers of the primal-dual sublevel set $\calW_\alpha$.
\begin{remark}\label{cor:H_bound_cond_numbers}
	For $\alpha = \gap(w(\eta))$, the diameter $D_\alpha$ and the Hausdorff distance $d^H_{\alpha}$ are at most as large as the diameter of $B_{w(\eta)}(w(\eta),\vartheta_F+2\sqrt{\vartheta_F})$, and the conic radius $r_\alpha$ is at least as large as the radius of $B_{w(\eta)}(w(\eta),1)$. Therefore
	\begin{equation}\label{eq:cor:H_bound_cond_numbers}
		d^H_{\alpha}\le D_\alpha \le \frac{2\vartheta_F+4\sqrt{\vartheta_F}}{\sqrt{\sigma_{\min}^+(H_{w(\eta)})}} \ , \ \ \text{ and } \ \  r_\alpha \ge \frac{1}{\sqrt{\sigma_{\max}^+(H_{w(\eta)})}} \ .
	\end{equation}
\end{remark}
\noindent (The first inequality in \eqref{cor:H_bound_cond_numbers} follows from Lemma \ref{lm D ge r}, and the second and third inequalities are derived from the fact that the diameter and radius of the ellipsoid $B_{w(\eta)}(w(\eta),r)$ are $\tfrac{2r}{\sqrt{\sigma_{\min}^+(H_{w(\eta)})}}$ and $\tfrac{r}{\sqrt{\sigma_{\max}^+(H_{w(\eta)})}}$, respectively.)

Remark \ref{cor:H_bound_cond_numbers} indicates that large values of $\sigma_{\max}^+(H_{w(\eta)})$ and small values of $\sigma_{\min}^+(H_{w(\eta)})$ might lead to a large ratio $D_\alpha / r_\alpha$ and a large Hausdorff distance $d^H_\alpha$, and thus to worse iteration bounds in Theorems \ref{thm overall complexity clp} and \ref{thm overall complexity lp}. Contrapositively, suppose that we have (or can easily construct) a linear transformation $\phi: \ \mathbb{R}^{2n} \mapsto \mathbb{R}^{2n}$ mapping primal-dual sublevel set $\calW_\alpha$ to $\tilde\calW_\alpha = \phi(\calW_\alpha)$ so that after transformation to the new variables $\tilde w = \phi(w)$ we have $\sigma_{\max}^+(H_{\tilde w(\eta)}) = \sigma_{\min}^+(H_{\tilde w(\eta)}) = \alpha$. Then from Remark \ref{cor:H_bound_cond_numbers} the bounds on the condition numbers of the transformed primal-dual sublevel sets will only involve $\vartheta_F$ and $\alpha$ and hence the iteration bounds in Theorems \ref{thm overall complexity clp} and \ref{thm overall complexity lp} will be similarly controlled.  Such a linear transformation $\phi$ is of course easily constructed from $H_{w(\eta)}$ itself, which we present in detail in the next subsection.

\subsection{``Rescaling'' linear transformation to improve the geometry of primal-dual sublevel sets}\label{subsec:rescaling_transformation}

Let $f$ be a logarithmically homogeneous self-concordant barrier function for $K_p$. Given a point $w(\eta)$ on the central path, we define the linear transformation (``rescaling'') $\phi$ of $w = (x,s) \in \mathbb{R}^{2n}$ as follows:
\begin{equation}\label{eq:def phi}
	\tilde w = (\tilde x, \tilde s) := \phi_\eta(w):= \left(\tfrac{1}{\eta}\cdot H_{w(\eta)}\right)^{1/2} w   \ = \left( \begin{array}{cc} \tfrac{1}{\sqrt{\eta}}\cdot H_{x(\eta)}^{1/2} & \\ & \tfrac{1}{\sqrt{\eta}}\cdot H_{s(\eta)}^{1/2} \end{array} \right) \left(\begin{array}{c} x \\ s \end{array} \right) =  \left( \begin{array}{c} \tfrac{1}{\sqrt{\eta}}\cdot H_{x(\eta)}^{1/2} x  \\  \tfrac{1}{\sqrt{\eta}}\cdot H_{s(\eta)}^{1/2} s \end{array} \right)  \ ,
\end{equation}
where we use $M^{1/2}$ to denote the symmetric square root of a symmetric positive semi-definite matrix $M$. Let $\alpha$ denote the duality gap of $w(\eta)$, whereby $\alpha := \gap(w(\eta)) = \vartheta_f/\eta$ from item (\ref{{item:fact:centralpath_main:2}}.) of Fact \ref{fact:centralpath_main}. Under the linear transformation $\phi_\eta$ in \eqref{eq:def phi}, we study the geometry of the rescaled sublevel sets and related objects:
\begin{equation}\label{congrats}\tilde{\calW}_{\alpha} := \phi_\eta \left(\calW_{\alpha} \right) \ , \ \tilde{\calW}^\star := \phi_\eta \left( \calW^\star \right) \ , \ \text{ and } \tilde{K} := \phi_\eta(K) \ . \end{equation}
\begin{theorem}\label{thm:conditionnumber_after_transformation}
	Let $f$ be a logarithmically homogeneous self-concordant barrier function for $K_p$, and let $\alpha:= \gap(w(\eta)) = \vartheta_f/\eta$.  Under the linear transformation \eqref{eq:def phi} and following the notation \eqref{congrats}, the following bounds hold:
	\begin{align}
		 & \tilde{D}_{\alpha}  := \max_{u,v\in \tilde{\calW}_{\alpha}} \|u-v\| \le \frac{4\vartheta_f + 4 \sqrt{2\vartheta_f}}{\sqrt{\eta}}  \label{eq:tilde_D}                                                                          \\
		 & \tilde{r}_{\alpha} :=\left(
		\begin{array}{cl}
				\max_{w\in \mathbb{R}^{2n}, \ r\ge 0 } & r                                                         \\
				\operatorname{s.t.}                    & w \in \tilde{\calW}_{\alpha}, \ B(w,r)\subseteq \tilde{K}
			\end{array}
		\right) \ge \frac{1}{\sqrt{\eta}} \label{eq:tilde_r}                                                                                                                                                                             \\
		 & \tilde{d}^H_{\alpha}  := D^H(\tilde{\calW}_{\alpha}, \tilde{\calW}^\star) := \max_{w\in \tilde{\calW}_{\alpha}}\dist(w,\tilde{\calW}^\star) \le \frac{4\vartheta_f + 4 \sqrt{2\vartheta_f}}{\sqrt{\eta}}   \label{eq:tilde_d} \\
		 & \dist(0,\tilde{\calW}^\star)  \le \frac{2\vartheta_f + 3 \sqrt{2\vartheta_f}}{\sqrt{\eta}} \ , \text{ and } \label{eq:tilde_dist}                                                                                             \\
		 & \frac{\tilde{D}_{\alpha} }{\tilde{r}_{\alpha}}   \le 4\vartheta_f + 4 \sqrt{2\vartheta_f} \ .  \label{eq:ratioratio}
	\end{align}
\end{theorem}
\noindent Theorem \ref{thm:conditionnumber_after_transformation} states that the condition numbers of the rescaled sublevel set $\tilde{\calW}_{\alpha}$ have upper or lower bounds that only involve $\vartheta_f$ and $\eta$. Note that all these condition number bounds are decreasing in $\eta$. Moreover, the ratio $\frac{\tilde{D}_\alpha}{\tilde{r}_\alpha}$ remains bounded above by $4\vartheta_f + 4\sqrt{2\vartheta_f}$, which is likely to be significantly smaller than the original ratio $\frac{D_\alpha}{r_\alpha}$ (which could be arbitrarily large). So long as $\eta$ is not itself too small, the rescaled sublevel set $\tilde{\calW}_{\alpha}$ may have better geometry than the original sublevel set $\calW_{\alpha}$.
(Of course, if the original ratio $D_\alpha/r_\alpha$ is already sufficiently small, then the rescaling transformation may not be beneficial.)  The proof of Theorem \ref{thm:conditionnumber_after_transformation} is presented in Appendix \ref{appendix:geometric_enhancements}.

Furthermore, it follows from the barrier calculus of self-concordant functions that the linear map $\phi_\eta$ can be expressed quite simply using just the Hessian of $f$ as is shown in the following proposition, whose proof is presented in Appendix \ref{appendix:geometric_enhancements}. 
\begin{proposition}\label{lm:feature_of_log_hessian} Let $f$ be a logarithmically homogeneous self-concordant barrier function for $K_p$, and define $D_1:=\sqrt{\eta}\cdot H_{x(\eta)}^{-1/2}$. For $w=(x,s)$, $\phi_{\eta}(w)$ can be expressed  as
	\begin{equation}\label{eq_of_lm:feature_of_log_hessian}
		\phi_\eta(w) = \left(D_1^{-1} x, D_1 s\right) \ .
	\end{equation}
\end{proposition}
\noindent

\noindent Using Proposition \ref{lm:feature_of_log_hessian} we define $\tilde{K} = \phi_\eta(K) = \tilde{K_p} \times \tilde{K}_d  := D_1^{-1}\cdot K_p \times D_1 \cdot K_d$, and notice as well that $\tilde{K_p}$ and $ \tilde{K}_d$ are indeed dual cones. Note also that
$\tilde{\calW}_{\alpha}  = \phi_\eta(\calW_{\alpha}) $, so we can write $\tilde{\calW}_\alpha = $
{\small\begin{equation*}
			\begin{aligned}
				\left\{ \big(D_1^{-1} x, D_1 s\big):
				\begin{array}{l}
					Ax = b,                                                     \\
					\exists \ y \in\mathbb{R}^m \text{ s.t. } A^\top y + s = c, \\
					x\in K_p, s\in K_d                                          \\
					c^\top x - b^\top y \le \alpha
				\end{array}
				\right\} = \left\{ \left(\tilde{x}, \tilde{s}\right):
				\begin{array}{l}
					AD_1 \tilde{x} = b,                                                                   \\
					\exists \ y \in\mathbb{R}^m \text{ s.t. } D_1^\top A^\top y + \tilde{s} = D_1^\top c, \\
					D_1 \tilde{x}\in K_p, \  D_1^{-1}\tilde{s}\in K_d                                     \\
					(D_1^\top c)^\top x - b^\top y \le \alpha
				\end{array}
				\right\} \\
			\end{aligned}
		\end{equation*}}where the equality substitutes $\big(D_1^{-1} x, D_1 s\big)$ with $\left(\tilde{x}, \tilde{s}\right)$. Note that here $D_1^{-1}\cdot K_p = \tilde{K}_p$ and $D_1 \cdot K_d = \tilde{K}_d$.

Additionally, let $D_2\in\mathbb{R}^{m\times m}$ be a given full-rank matrix.  Since any full-rank row transformation of the linear constraints does not change the feasible sets of $x$ and $s$, it follows that $\tilde{\calW}_\alpha$ is the $\alpha$-sublevel set of the following primal and dual paired problems:
\begin{equation}\tag{$\tilde{\mathrm{P}}$}\label{pro:rescaled problem}
	\text{(Primal)} \quad
	\min_{x\in\mathbb{R}^n} \  (D_1^\top c)^\top x \quad
	\operatorname{s.t.} \ D_2 A D_1 x = D_2 b, \  x \in  \tilde{K}_p \ \end{equation}
\begin{equation}\tag{$\tilde{\mathrm{D}}_{y,s}$}\label{pro:dual rescaled problem}
	\text{(Dual)} \quad
	\max_{y\in\mathbb{R}^m, s\in\mathbb{R}^n}  \  (D_2 b)^\top y \quad
	\operatorname{s.t.}  \ (D_2 A D_1)^\top y + s = D_1^\top c, \  s \in  \tilde{K}_d \ .
\end{equation}
Therefore the favorable geometry described in Theorem \ref{thm:conditionnumber_after_transformation} can be equivalently achieved by the appropriate rescaling of the problem data, which we state formally as follows.
\begin{fact}\label{fact:correct_data_rescaling} Let $D_2\in\mathbb{R}^{m\times m}$ be a given full-rank matrix, and define:
	\begin{equation}\label{eq:defD_1} D_1:=\sqrt{\eta}\cdot H_{x(\eta)}^{-1/2} \ . \end{equation}
	Then the $\alpha$-sublevel sets, the optimal solution sets, and the underlying primal-dual cone of \eqref{pro:rescaled problem} and \eqref{pro:dual rescaled problem} are $\tilde{\calW}_{\alpha}$, $\tilde{\calW}^\star$ and $\tilde{K}$, and satisfy the bounds in Theorem \ref{thm:conditionnumber_after_transformation}.
\end{fact}

Despite the potential improvement in condition numbers as shown in Theorem \ref{thm:conditionnumber_after_transformation}, it might be the case that Algorithm \ref{alg: PDHG with restarts} (rPDHG) cannot be easily applied to \eqref{pro:rescaled problem} because projections onto the rescaled cone $\tilde{K}_p$ might be more expensive as compared to projections onto the original cone $K_p$. However, the following classic result (from Section 3.5 of \cite{renegar2001mathematical}, based on \cite{YuTodd97}) states that if $K_p$ is a self-scaled cone, the rescaled cone $\tilde{K}_p$ coincides with $K_p$ and so projections to $\tilde K_p$ are no more cumbersome than projections onto $K_p$.
\begin{lemma}\label{lm:symmetric_cone}
	If $K_p$ is a self-scaled cone, then $K_p = \tilde{K}_p = K_d = \tilde{K}_d$.
\end{lemma}

Hence, if $K_p$ is self-scaled, the rescaled cone $\tilde{K}_p$ is identical to the original cone $K_p$, and projections onto $\tilde{K}_p$ are the same as projections onto $K_p$. Although the family of self-scaled cones is in some sense limited \cite{Guler_96}, it includes the three most important cones in practice, namely $\mathbb{R}^n_+$, $\mathbb{K}_{\textsf{soc}}^{d+1}$, and $\mathbb{S}^{d\times d}_+$ and their Cartesian products \cite{Guler_96,renegar2001mathematical}.

We denote the affine subspaces of the rescaled primal/dual problem by $\tilde{V}_p$ and $\tilde{V}_d$, respectively. Similarly, let $\tilde{V}:= \tilde{V}_p\times \tilde{V}_d$ and $\tilde{\calF}:=\tilde{V}\cap \tilde{K}$.

In summary up to this point, we have shown after the rescaling the sublevel-set condition numbers can be nicely bounded, and the complexity of doing projections onto the rescaled cone $\tilde{K}_p$ remains unchanged if $K_p$ is self-scaled. In the next subsection we present a computational guarantee for obtaining a solution of the original problem \eqref{pro: general primal clp} by first solving the rescaled problem \eqref{pro:rescaled problem} using rPDHG and then transforming back to \eqref{pro: general primal clp}.

\subsection{Complexity of rPDHG via solving the rescaled problem}\label{subsec:complexity_of_getting_good_solution}

Algorithm \ref{alg: PDHG for rescaled problem} describes our overall scheme for using a central-path Hessian rescaling of the original problem \eqref{pro: general primal clp} and then using rPDHG to solve the rescaled problem \eqref{pro:rescaled problem}, and then transforming the solution back to the original problem \eqref{pro: general primal clp}. In Line \ref{alg: amc} the rescaled problem \eqref{pro:rescaled problem} is constructed based on the Hessian matrix $D_1$ and also $D_2$ as inputted from Line \ref{alg: inp}.  In Line \ref{alg:line:apply_rPDHG} the rescaled problem is solved using rPDHG, and in Line \ref{alg:line:recover_solution} the solution $\tilde{w}^{n,0} = (\tilde{x}^{n,0},\tilde{s}^{n,0})$ of \eqref{pro:rescaled problem} is transformed back to the original problem \eqref{pro: general primal clp} using $w^{n,0} := (D_1 \tilde x^{n,0}, D_1^{-1} \tilde s^{n,0}) $.  Line \ref{noodles} describes an alternate way to transform $\tilde{w}^{n,0} = (\tilde{x}^{n,0},\tilde{s}^{n,0})$ back to the original problem, namely by first projecting $\tilde x^{n,0}$ onto $\tilde V_p$ using $\hat x^{n,0} := P_{\tilde V_p}(\tilde x^{n,0})$, and then applying: $w^{n,0} := (D_1 \hat x^{n,0}, D_1^{-1} \tilde s^{n,0}) $.  This alternative version is used in our complexity analysis in Theorem \ref{thm overall complexity rescaled} as it leads to a more streamlined bound and a more streamlined proof.  

\begin{algorithm}[htbp]
	\SetAlgoLined
	{\bf Input:} Hessian matrix $H_{x(\eta)}$ of a central-path solution $w(\eta)=(x(\eta),s(\eta))$ of a logarithmically homogeneous self-concordant barrier function $f$ for $K_p$, and a full-rank matrix $D_2\in\mathbb{R}^{m\times m}$  . Define $\alpha  := \gap(w(\eta)) = \vartheta_f /\eta$ \; \label{alg: inp}
	Construct the rescaled problem \eqref{pro:rescaled problem} with $D_1 = \sqrt{\eta}\cdot H_{x(\eta)}^{-1/2}$ and $D_2$ as given in Line \ref{alg: inp} \; \label{alg: amc}
	Apply Algorithm \ref{alg: PDHG with restarts} (rPDHG) to the rescaled problem \eqref{pro:rescaled problem} for $n$ outer iterations to obtain $\tilde{w}^{n,0} = (\tilde{x}^{n,0},\tilde{s}^{n,0})$ \; \label{alg:line:apply_rPDHG}
	Transform solution to the original problem: $w^{n,0} := (D_1 \tilde x^{n,0}, D_1^{-1} \tilde s^{n,0}) $ \;\label{alg:line:recover_solution}
	(Alternate pre-projection version) First project $\tilde x^{n,0}$ onto $\tilde V_p$ using $\hat x^{n,0} := P_{\tilde V_p}(\tilde x^{n,0})$, and then transform to the original problem: $w^{n,0} := (D_1 \hat x^{n,0}, D_1^{-1} \tilde s^{n,0}) $ \; \label{noodles}
	{\bf Output:} $w^{n,0}$\label{alg:line:projection_w}
	\caption{Scheme for Hessian rescaling and applying rPDHG to \eqref{pro:rescaled problem} }\label{alg: PDHG for rescaled problem}
\end{algorithm}

Theorem \ref{thm overall complexity rescaled} below describes our computational guarantee for the rescaling scheme (Algorithm \ref{alg: PDHG for rescaled problem}).  Before presenting the theorem, we first go over some notational and related matters. First, note that any translation of the objective function vector $D_1^\top c$ of the rescaled problem \eqref{pro:rescaled problem} along a direction in the space of $\operatorname{Im}(D_1^\top A^\top D_2^\top)$ does not change the optimal solution set $\tilde{\calW}^\star$. Therefore, similar to the setting of Theorem \ref{thm overall complexity clp}, we will presume that $D_1^\top c \in \operatorname{Null}(D_2AD_1)$. This condition can be satisfied by replacing $D_1^\top c$ with its projection $\bar c$ onto $\operatorname{Null}(D_2 A D_1)$, namely $\bar{c}:=\arg\min_{\hat{c}\in\operatorname{Null}(D_2 A D_1)}\|\hat{c}-D_1^\top c\|$.

For the rescaled problem \eqref{pro:rescaled problem} we define the notation $\tilde{\lambda}_{\max}:= \sigma_{\max}^+\left(D_2 A  D_1\right)$, $\tilde{\lambda}_{\min}:= \sigma_{\min}^+\left(D_2 A D_1\right)$, and $\tilde{\kappa} := \frac{\tilde{\lambda}_{\max}}{\tilde{\lambda}_{\min}}$.
We denote the iterates of rPDHG on the rescaled problem as $\tilde{z}^{n,k}=(\tilde{x}^{n,k},\tilde{y}^{n,k})$ and $\tilde{w}^{n,k} = \left(\tilde{x}^{n,k},\tilde{s}^{n,k}:=\bar{c} - (D_2 A D_1)^\top\tilde{y}^{n,k}\right)$.  

\begin{theorem}\label{thm overall complexity rescaled}
	Under Assumption \ref{assump:striclyfeasible}, suppose that the rescaled objective function vector $\bar{c} := D_1^\top c$ satisfies $D_2 A D_1 \bar{c} = 0$, and that Algorithm \ref{alg: PDHG for rescaled problem} is applied using Line \ref{noodles} to transform solutions back to the original problem.  Furthermore assume rPDHG is implemented in Line \ref{alg:line:apply_rPDHG} of Algorithm \ref{alg: PDHG for rescaled problem} using the same set-up as in Theorem \ref{thm overall complexity clp}, and let $\alpha  := \gap(w(\eta)) = \vartheta_f /\eta$. Given $\eps := (\eps_{\mathrm{cons}}, \eps_{\mathrm{gap}},\eps_{\mathrm{obj}})\in\mathbb{R}_{++}^3$,
	let $T$ be the total number of \textsc{OnePDHG} iterations that are run in Line \ref{alg:line:apply_rPDHG} of Algorithm \ref{alg: PDHG for rescaled problem} in order to obtain $n$ for which $w^{n,0}$ satisfies the $\eps$-tolerance requirement of \eqref{pro: general primal clp}.
	Then it holds that
	\begin{equation}\label{eq overall complexity rescaled clp}
		T \le  \tilde{\kappa}  (\vartheta_f + \sqrt{2\vartheta_f})  \cdot \left\{ 760\ln\left(33\tilde{\kappa} \left(2\vartheta_f + 3 \sqrt{2\vartheta_f}\right) \right) + 760\ln\left(\frac{1}{\err^{\alpha}}\right)
		+ \frac{200 }{\err^{\alpha}} \right\}\ ,
	\end{equation}
	where $\err^{\alpha}$ is the following (weighted) minimum of the target tolerance values $\eps$:
	\begin{equation}\label{def eq Merr rescaled clp}
		\err^{\alpha} :=  \min\left\{\frac{\eps_{\mathrm{cons}}}{\sqrt{2}\cdot D_{\alpha}} , \
		\frac{\eps_{\mathrm{gap}}}{14\alpha}  ,\  \frac{\eps_{\mathrm{obj}}}{14\alpha}
		\right\}\ .¢¢
	\end{equation}
	Furthermore, in the case where $K_p = \mathbb{R}^n_+$, then
	\begin{equation}\label{eq overall complexity rescaled lp}
		T \le \tilde{\kappa}  \left(n + \sqrt{2n}\right)\cdot  \max\left\{1,\frac{\alpha}{\bar{\delta}}\right\} \cdot   1020\left[\ln\left(33\tilde{\kappa} \left(2n + 3 \sqrt{2n}\right) \right) + \ln\left(\frac{1}{\err^{\alpha}}\right) \, \right]\,  ,
	\end{equation}
	where $\bar{\delta}$ is the best suboptimal extreme point gap defined in \eqref{eqdef:best_suboptimal_gap}.
\end{theorem} 

The proof of Theorem \ref{thm overall complexity rescaled} is presented in Appendix \ref{appendix:geometric_enhancements}.  In practice, it is likely more efficient to use Line \ref{alg:line:recover_solution} rather than Line \ref{noodles} of Algorithm \ref{alg: PDHG for rescaled problem} to determine the solution $w^{n,0}$ for the original problem, as the pre-projection operation might itself be expensive for very large problem instances. Depending on the practical usefulness of the pre-projection, it could be implemented using the conjugate gradient method.  Similar remarks hold for the projection of $D_1^\top c$ onto $\operatorname{Null}(D_2AD_1)$.

The error tolerance quantity $\err^{\alpha}$ is a weighted minimum of the three target tolerances in Theorem \ref{thm overall complexity rescaled}, and plays a similar role as $\err$ did in Theorem \ref{thm overall complexity clp}.  Here the different multipliers inside the minimum in the expression for $\err^{\alpha}$ are $\tfrac{1}{\sqrt{2}\cdot D_\alpha}$, $\tfrac{1}{14\alpha}$, and  $\tfrac{1}{14\alpha}$, where $\alpha=\gap(w(\eta)) = \vartheta_f/\eta$. Furthermore, if we use $D_2 := (A D_1^2 A^\top)^{-1/2}$ then $\tilde{\kappa} =1$, then the smaller $\gap(w(\eta))$ is, the larger the required $\err^{\alpha}$ is, resulting in a smaller overall complexity bound. Furthermore, unlike the bounds in \eqref{def eq calN} or \eqref{def eq calN lp}, the multipliers on $\eps_{\mathrm{gap}}$ and $\eps_{\mathrm{obj}}$ in $\err^{\alpha}$ are the same, which means that achieving a small objective function error now here is no more difficult than achieving a small duality gap (at least in theory).

In Appendix \ref{appendix:geometric_enhancements} we actually prove a stronger result than what is stated in Theorem \ref{thm overall complexity rescaled}. One can observe from our proof in Appendix \ref{appendix:geometric_enhancements} that the solution $w^{n,0}$ is guaranteed to satisfy $\dist(w^{n,0},\calF) \le \eps_{\mathrm{cons}}$, which is a (potentially significantly) stronger statement than  $\max\{\dist(w^{n,0},K), \dist(w^{n,0},V)\}\le \eps_{\mathrm{cons}}$.

Comparing \eqref{eq overall complexity rescaled clp} to \eqref{eq overall complexity}, we see that the geometric condition numbers (which for some instances might be extremely large) have been replaced by $O(\vartheta_f + \sqrt{2\vartheta_f})$. This is significant as it is better controlled; for instance, it is as small as $O(n)$ for LP problem instances.  Also, the dependence on $\tilde{\kappa}$ suggests that if not too expensive one should construct $D_2$ to decrease $\tilde{\kappa}$ as much as possible.

We also remark on the overall dependence on $\err^{\alpha}$ in \eqref{eq overall complexity rescaled lp} for LP problem instances.  Notice that the constant in the front of the logarithm term in \eqref{eq overall complexity rescaled lp} is $O\left(\tilde{\kappa}\cdot n\cdot  \max\left\{1,\frac{\alpha}{\bar{\delta}}\right\} \right)$, and therefore smaller $\alpha$ leads to a lower iteration bound. Furthermore, for appropriate choice of $D_1$ and $D_2$, it holds that rPDHG achieves linear convergence as follows.
\begin{corollary}\label{cor:best_rescaling_lp}
	When $K_p = \mathbb{R}^n_+$, let $w(\eta)$ have a sufficiently small duality gap, namely $\gap(w(\eta)) \le \bar{\delta}$.  Then setting $D_1 = \sqrt{\eta}\cdot H_{x(\eta)}^{-1/2}$ and $D_2 = (A D_1^2 A^\top)^{-1/2}$ yields
	$$
		T \le \left(n + \sqrt{2n}\right)\cdot   1020\left[\ln\left(33\left(2n + 3 \sqrt{2n}\right) \right) + \ln\left(\frac{1}{\err^{\alpha}}\right) \, \right] \ .
	$$
\end{corollary}
\noindent We point out that some classic interior-point methods, such as long-step barrier methods, have an overall iteration complexity of $O(n\cdot \ln(1/\eps))$ \cite{renegar2001mathematical} iterations.  Corollary \ref{cor:best_rescaling_lp} states that with a suitable rescaling scheme, rPDHG also achieves $O(n\cdot \ln(1/\eps))$ iterations. However, since the per-iteration cost of rPDHG is significantly lower than that of an IPM, the overall complexity of rPDHG with a central-path Hessian rescaling is therefore much lower than the corresponding IPM.

Last of all, we note that the central-path solution can be obtained using a first-order version of an IPM, namely where the linear equations solved to yield the Newton step are solved using first-order methods on an auxiliary least-squares optimization problem. Such first-order versions of interior-point methods have been studied and implemented, see for example \cite{lin2021admm} which develops an ADMM-based interior-point method, and \cite{deng2024enhanced} which shows it is as competitive as rPDHG in computing a solution of moderate accuracy for LP instances. Although they are not as competitive as rPDHG in solving the problem itself, these other first-order method schemes can produce a central-path solution that is of good enough quality for the purposes of doing Hessian-based rescaling. In Section \ref{sec:exp}, we will present a scheme for determining a central-path Hessian by using the conjugate gradient method to solve the central path equations.

\subsection{Computational Effectiveness of Central-path Hessian Rescaling}\label{subsec:APDHG_exp}

Theorem \ref{thm overall complexity rescaled} shows that (at least in theory) the central-path Hessian rescaling scheme (Algorithm \ref{alg: PDHG for rescaled problem}) can improve the worst-case computational burden of rPDHG for solving a conic linear optimization problem instance.  Theorem \ref{thm overall complexity rescaled} also shows that (in theory) using a Hessian scaling based on a point further along the central path (namely with lower duality gap $\alpha$) further improves the worst-case computational burden. This naturally leads to two computational questions:  (i) is the actual performance of rPDHG improved by using the Hessian scaling?, and if so (ii) is the actual performance of rPDHG enhanced by using a Hessian rescaling further along the central path?  In this subsection we present numerical experiments designed to answer these two questions, and we find that the answers to both questions are affirmative.

We conducted experiments on LP instances from the LP relaxations of the MIPLIB 2017 \cite{gleixner2021miplib} dataset, which is a collection of 1065 mixed-integer programs derived from real-world applications. In line with the experimental set-ups in \cite{applegate2021practical} and \cite{lu2023cupdlp-c}, we first applied the presolver PaPILO \cite{gleixner2023papilo} (an open-source presolving library) to all instances, which performs a variety of computations to identify inconsistent bounds, eliminate empty rows and columns of the constraint matrix, and remove variables with identical lower and upper bounds. All problems were then converted to standard-form (re)-formulation as in  \eqref{pro: general primal clp}.  We use the relative error: 
\begin{equation}\label{wimp} \mathcal{E}_r(x, y):=\max\left\{\frac{\|A x^{+}-b\|}{1+\|b\|},\frac{\|(c-A^{\top} y)^{-}\|}{1+\|c\|},\frac{|c^{\top} x^{+}-b^{\top} y|}{1+|c^{\top} x^{+}|+|b^{\top} y|}\right\} \end{equation} as the tolerance metric of solutions, which is standard in solver software \cite{lu2023cupdlp-c,applegate2021practical,mosek}.  
We studied two settings which we now describe.

\begin{itemize}
	\item \textbf{rPDHG (Central-$\boldsymbol{\delta}$)}: In this setting we run rPDHG on \eqref{pro:rescaled problem} with $D_1$ constructed using the central-path Hessian $D_1 := \eta \cdot H_{x(\eta)}^{-1/2}$ (as in Theorem \ref{thm overall complexity rescaled}) using an endogenously determined value of $\eta$ for which the relative error satisfies $ \mathcal{E}_r(x(\eta),y(\eta)) \approxeq \delta$. The diagonal entries of $D_1$ are then clipped so that they are in the range $[10^{-5}, 10^5]$ to avoid numerical issues. Here we test $\delta = 0.5$, $0.1$, and $0.01$. Then $D_2$ is set to $D_2 := (AD_1^2 A^\top)^{-1/2}$ which yields $\tilde \kappa = 1$. This is done in order to isolate the role of the rescaling matrix  $D_1$ on the performance of rPDHG by nullifying any effect of $\tilde \kappa$. 
	\item \textbf{rPDHG (EasyColumn)}: In this setting we run rPDHG on \eqref{pro:rescaled problem} with $D_1$ constructed as the diagonal matrix whose diagonal entries are the reciprocals of the $\ell_\infty$ norms of the columns of $A$. We set $D_2 := (AD_1^2 A^\top)^{-1/2}$.
\end{itemize}

 We refer to the choice of $D_2 := (AD_1^2 A^\top)^{-1/2}$ as the ``complete preconditioner''. All computations were performed on the MIT Engaging Cluster, allocated with one Intel E5-2690 v4 CPU and 16GB RAM per task. All algorithms were implemented in Julia 1.8.5. The central-path points were computed using the conic optimization solver of Mosek \cite{mosek}. In alignment with Theorem \ref{thm overall complexity rescaled}, the objective vector $D_1^\top c$} of \eqref{pro:rescaled problem} was replaced with its projection $\bar{c} := \arg\min_{\hat{c}\in\operatorname{Null}(D_2 A D_1)}\|\hat{c}-D_1^\top c\|$.
Using the complete preconditioner for $D_2$ yields $\tilde{\lambda}_{\max} = \tilde{\lambda}_{\min} = 1$ and the step-sizes in Theorem \ref{thm overall complexity rescaled} are both equal to $1$.  Instead of using these step-sizes, we selected the step-sizes of rPDHG more conservatively as $\tau=\sigma=0.8$ which made it easier to compute the approximate normalized duality gap, see Appendix \ref{appendix:compute_rho} for further details. Additionally, we adopted the ``flexible adaptive restart'' strategy as developed in \cite{applegate2023faster} with the $\beta =1/e$-restart condition. (The flexible restart ends an inner loop of rPDHG if the normalized duality gap of either (a) the current iterate or (b) the average of the loop's iterates, satisfies the $\beta$-restart condition \cite{applegate2023faster}.) 
The initial iterates of rPDHG were chosen as $x^{0,0} = 0$ and $y^{0,0}=0$. For each instance, we set the time limit of rPDHG to be 10,000 seconds.

\begin{figure}[b]
	\centering
	\includegraphics[width=1\linewidth]{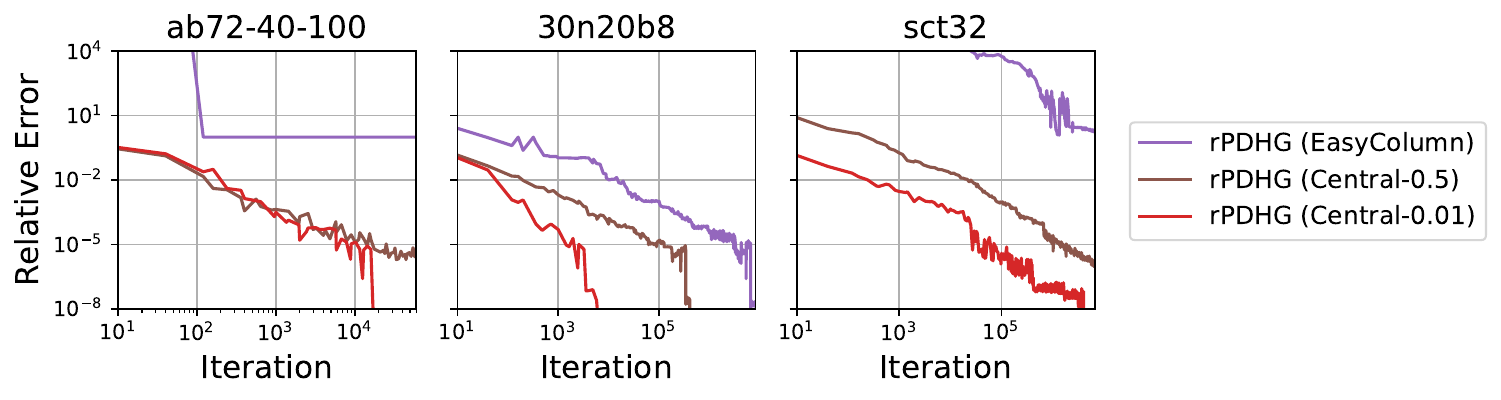}
	\caption{Performance comparison of rPDHG (Algorithm \ref{alg: PDHG with restarts}) using different column rescalings (and using complete preconditioner for $D_2$), for the problem instances \text{ab72-40-100}, \text{30n20b8}, and \text{sct32}.}\label{fig:apdhg_4_lp_instances}
\end{figure}

Figure \ref{fig:apdhg_4_lp_instances} illustrates the performance of the Hessian rescaling for two different points on the central path, namely $\delta = 0.5$ and $\delta = 0.01$, in comparison to using the EasyColumn rescaling, on three representative problem instances, namely \text{ab72-40-100}, \text{30n20b8}, and \text{sct32}. The plots are in log-log scale, with the number of iterations of \textsc{OnePDHG} on the horizontal axis and the relative error $\mathcal{E}_r(x, y)$ on the vertical axis. For these three instances we see that rPDHG converges faster using a central-path Hessian rescaling than using the EasyColumn rescaling, and the Hessian rescaling convergence is faster for $\delta=0.01$ than for $\delta = 0.5$, which is consistent with our theory (and with intuition). Although rPDHG with EasyColumn rescaling theoretically achieves linear convergence (Theorem \ref{thm overall complexity lp}), such linear convergence is not observable in the first 10,000 iterations for instance \text{ab72-30-100}, which indicates that the corresponding constant factor for the linear convergence rate is quite large for this instance. And for this instance we see a large improvement in the convergence from using a central-path Hessian rescaling.  For the problem instance \text{30n20b8}, we observe that the EasyColumn rescaling yields sublinear-like convergence performance in the early stages followed by linear convergence after around $10^6$ iterations, in synch with Theorem \ref{thm overall complexity clp}. The two central-path Hessian rescalings lead to faster linear convergence on this problem as well. For problem instance \text{sct32}, none of the three methods reach the linear convergence stage within $10^7$ iterations, but the central-path Hessian rescalings still have significant speedups over the EasyColumn rescaling. This is in synch with Theorem \ref{thm overall complexity rescaled} which indicates that a central-path Hessian rescaling yields improved convergence in both the sublinear and linear convergence stages.

To compare the different rescalings on a broad set of problem instances we chose the subset of the MIPLIB 2017 LP relaxation instances for which (i) Mosek can successfully compute central-path points, (ii) the problem instance is not too small, namely $mn \ge 10^7$, and (iii) the problem instance is not too large, namely $mn \le 5\times 10^8$ (so that we have a sufficiently large family of instances that are solved in the 10,000 seconds time limit). This yielded 222 problem instances, which we solved using different rescalings. We consider a problem instance to be solved if it achieves a relative error satisfying $\mathcal{E}_r(x, y) \le 10^{-8}$, which is a standard threshold for LP applications but is an extremely strict threshold in the context of first-order methods. Figure \ref{fig:apdhg_proportionplot} shows the fraction of the solved problems (among the 222 instances) on the horizontal axis, and the maximum iterations (leftmost plot) and the maximum runtime (rightmost plot) to achieve the fraction of the problems solved. We observe that the central-path Hessian rescaling markedly reduces the iteration counts and the runtime, in comparison with the  EasyColumn rescaling. For example, rPDHG (Central-0.5) requires 100 times fewer iterations to solve $60\%$ of the problem instances, compared to rPDHG with EasyColumn rescaling. Also, smaller value of $\delta$ (which corresponds primarily to a smaller duality gap on the central path) yields a larger fraction of problems solved. And if obtaining a low-accuracy point on the central-path is not computationally expensive, this suggests that minor efforts to obtain a central-path point have the potential to greatly improve the computational performance of rPDHG.  This last point is further developed and tested in Section \ref{sec:exp}. 

\begin{figure}[htbp]
	\centering
	\includegraphics[width=1\linewidth]{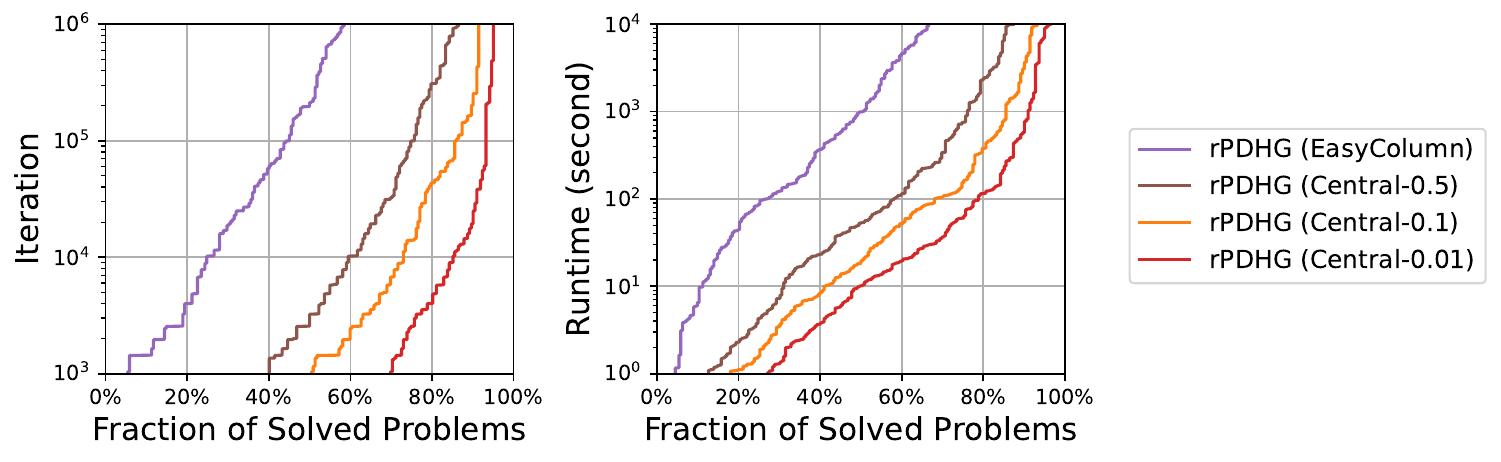}
	\caption{Performance comparison of rPDHG (Algorithm \ref{alg: PDHG with restarts}) using different column rescalings (and using complete preconditioner for $D_2$), on the 222 LP relaxation problem instances from the MIPLIB 2017 dataset}\label{fig:apdhg_proportionplot}
\end{figure}

In this subsection we have illustrated how to alleviate rPDHG's dependence on the sublevel set geometry by using central-path Hessian rescalings.  While in one sense Figures \ref{fig:apdhg_4_lp_instances} and \ref{fig:apdhg_proportionplot} show this rather nicely, in another sense it ignores two important aspects, namely (i) it does not account for the computational cost of computing a point on the central path, and (ii) it presumes that the problem is pre-conditioned using the complete preconditioner $D_2 := (AD_1^2 A^\top)^{-1/2}$, which is likely to be computationally burdensome for very large-scale problem instances. Moreover, many problem instances do not satisfy Assumption \ref{assump:striclyfeasible} and thus the central path does not exist. In Section \ref{sec:exp}, we develop a solution scheme that uses the conjugate gradient method to solve the IPM Newton steps and so compute a ``reasonably good'' interior-point solution that can be used to construct an approximate central-path Hessian rescaling, which will form the basis of a practical rescaling scheme for rPDHG for conic linear optimization problem instances.

\section{An Adaptive Hessian Rescaling Scheme for rPDHG and its Experimental Evaluation}\label{sec:exp}

Both the theory and the computational experiments in Section \ref{sec:geometric_enhancements} show that it can be very valuable to utilize a central-path Hessian rescaling transformation and then run rPDHG on the rescaled problem \eqref{pro:rescaled problem}.  The same theory and computational experiments also show that the Hessian rescaling of points further along the central path yields better performance for rPDHG both in theory and in actual computation time. However, a practical incorporation of such a Hessian rescaling must address at least two challenges, namely (i) which target point on (or near) the central path to compute?, and (ii) how to most efficiently compute the target point?  Regarding the choice of the target point, there are various strategies one can develop for determining where to aim for in computing a point near the central path.  For example, a very simple heuristic is to fix $\eta$ at some predetermined value, say $\eta := 0.1/\vartheta$, and then approximately compute $w(\eta)$.  A more complicated heuristic would be to first run rPDHG for a certain number of iterations (or for a certain amount of time) and then somehow use the information from iterates of rPDHG to choose a suitable value of $\eta$ and then compute $w(\eta)$.
	
Even if one has a clear strategy for determining which target point $w(\eta)$ to try to compute, it is still a significant challenge to efficiently compute a good approximation of $w(\eta)$. In order for rPDHG to be competitive with the best methods for linear programming instances (pivoting methods and IPMs), the scheme for approximately computing a specific $w(\eta)$ should involve some form of Newton step computation, but should not use a direct method for solving linear equations (such as Cholesky factorization).  However, computation with IPMs where iterative methods are used for solving the associated Newton step -- such as the conjugate gradient method -- are not numerically reliable for computing points far along the central path (i.e., with very small duality gap), see \cite{karmarkar1991computational,mehrotra1992implementations}.

We address these two challenges directly in this section. To address the first challenge, we present in Section \ref{subsec:adaptive_rescaling} a practical scheme for {\em adaptively} computing a good Hessian rescaling that is designed to efficiently balance the extra computation time to determine a good rescaling with the computational savings from using rPDHG on the rescaled problem. To address the second challenge, in Section \ref{subsec:\cpipm} we describe our implementation of a Newton-step-based algorithm that is designed to compute a low-accuracy point on the central path (namely, a point that is not very far along the path) where the conjugate gradient method (CGM) is used to approximately solve for the Newton step.  We call this method \cpipm \ since it computes a point on the central path using the conjugate gradient method (to solve the Newton step equations). In Section \ref{fivemiles} we present a comparison of our method with two standard methods for LP, namely an implementation of rPDHG using the scaling technique developed in PDLP \cite{applegate2021practical}, and an implementation of a standard IPM for LP almost exactly as in \cite{nocedal2006numerical}.  We test all three methods on LP problem instances from the MIPLIB 2017 dataset, and note that all computational experiments in this section follow the same general setup as in Section \ref{subsec:APDHG_exp} except as otherwise noted.  

\subsection{Scheme for rPDHG with adaptive Hessian rescaling }\label{subsec:adaptive_rescaling}

Algorithm \ref{alg:ada_rescaling} describes our scheme for running rPDHG with adaptive Hessian rescaling, which we denote as \adarescaling. Before doing a line-by-line explanation of Algorithm \ref{alg:ada_rescaling}, we first explain the underlying strategy at a more informal level. The strategy is to start by first spending a small amount of time $t$ (we set $t=0.5$ seconds) to compute a very low-accuracy approximate central path point using \cpipm. (Recall that \cpipm \ denotes the method that computes a point on the central path using the conjugate gradient method to solve the Newton step equation system, the details of which are further described in Section \ref{subsec:\cpipm} and Appendix \ref{clifficlerun}.)  We then use the approximate central path point to construct the rescaling transformation and then run rPDHG on the rescaled problem for $\omega t$ seconds (we used $\omega = 6.0$).  We then test to determine if our current rescaling is good enough or not, and proceed accordingly as follows:
\begin{itemize}
\item If the relative error of the current rPDHG solution is not too much larger than the target relative error, we declare the current rescaling to be good enough, and we keep running rPDHG until we achieve the target accuracy.
\item If the relative error of the current rPDHG solution is worse than the relative error of the previous round of rPDHG (using the previous rescaling), and if the previous relative error itself was reasonably good enough, we discard the current rescaling, and we declare the previous rescaling to be good enough.  We then keep running rPDHG using the previous rescaling until we achieve the target accuracy.
\item If neither of the above two conditions is satisfied, we replace $t \leftarrow 2t$ and run (or continue running) \cpipm \ for an additional $t$ seconds, and repeat all of the above.  
\end{itemize}

We continue to use the notation of previous sections of this paper for iteration $k$ solutions of rPDHG, namely $z^k$, $w^k$, $x^k$, $y^k$, $s^k$, etc. To distinguish between these iterates and central path solutions produced by \cpipm \ that are used in the rescaling transformation, we use a different font and let ${\textsf z}^k$, ${\textsf w}^k$, ${\textsf x}^k$, ${\textsf y}^k$, ${\textsf s}^k$, etc., denote the corresponding objects at iteration $k$ of \cpipm. With this notation in mind, we now do a line-by-line explanation of Algorithm \ref{alg:ada_rescaling}. In Line \ref{fender} the initial point for rPDHG is set to $z^0$, and the initial point for \cpipm \ is set to ${\textsf w}^0$. In addition to the target relative error $\eps$, the input is also composed of two adaptivity parameters $\hat{\eps}$ and $\bar{\eps}$ that will play a role in the logic of Line \ref{yikes} and will be explained later.  For now, it is best to think of $\hat{\eps}$ and $\bar{\eps}$ as tolerance values that are larger than $\eps$, and in fact we use $\bar{\eps} = \sqrt{\eps}$, $\hat{\eps} = \sqrt[5]{\eps}$ in our implementation. 
In Line \ref{bender} we run (or continue running) \cpipm \ starting from ${\textsf w}^k$ for $t$ seconds to produce ${\textsf w}^{k+1} = ({\textsf x}^{k+1},{\textsf s}^{k+1})$.
In Line \ref{render}, we construct the rescaled problem \eqref{pro:rescaled problem}$_{k+1}$ using the Hessian of the solution ${\textsf x}^{k+1}$ and an additional rescaling from a certain pair of $\bar{D}_1$ and $\bar{D}_2$ designed to balance the geometry condition numbers as well as $\tilde{\kappa}$.
In Line \ref{lender} we run rPDHG on the new rescaled problem for $\omega t$ seconds, and we output the transformed solution $z^{k+1}$ and its associated relative error $\eps_{k+1}:= \err(x^{k+1}, y^{k+1})$.  In Line \ref{sender} the counter is updated and the time parameter is doubled.  This entire process is then repeated unless certain conditions are satisfied, which are denoted condition (a) or condition (b) in Line \ref{yikes}. 

Condition (a) in Line \ref{yikes} tests if the current output of rPDHG has ``reasonably good'' relative error (at most $\bar\eps$), in which case we declare the current rescaled \eqref{pro:rescaled problem}$_{k}$ has a good enough rescaling, and so in Line \ref{morning} we permanently fix the new rescaling and continue running rPDHG until we attain the target relative error $\eps$.  

Condition (b) in Line \ref{yikes} tests if the current output of rPDHG is sufficiently poor compared to the previous output from rPDHG ($\eps_k > \eps_{k-1} $), and if also the previous output from rPDHG has ``fairly good'' relative error ($\eps_{k-1} \le \hat{\eps}$) then we declare the previous rescaling is good enough, and so in Line \ref{mourning} we permanently fix the rescaled problem to its previous form \eqref{pro:rescaled problem}$_{k-1}$ and continue running rPDHG until we attain the target relative error $\eps$.

The overall design strategy in Algorithm \ref{alg:ada_rescaling} is to minimize the time invested in obtaining a reasonably effective rescaling and the time spent utilizing the rescaling in running rPDHG.

\begin{algorithm}[htbp]
	\SetAlgoLined
	{\bf Input:} Initial iterate $z^{0}:=(x^{0,0}, y^{0,0})$, initial point ${\textsf w}^0$ for \cpipm , time parameter $t$, time multiplier $\omega$, target relative error $\eps$, and adaptivity error parameters $\hat{\eps}$ and $\bar{\eps}$. Define $k  := 0$, $\eps_0 := +\infty$ \label{fender} \;
	\Repeat{ either (a) $\eps_k\le \bar{\eps}$, or (b) $\eps_k > \eps_{k-1} $ and \ $\eps_{k-1} \le \hat{\eps}$  \label{yikes}}{
		Run (or continue running) \cpipm \ from ${\textsf w}^{k}$ for $t$ seconds. Output ${\textsf w}^{k+1}$ \label{bender} \;
		Construct new rescaled problem: define $\eta^{k+1}:= ({\textsf s}^{k+1})^\top {\textsf x}^{k+1}$, $\tilde D_1 := \sqrt{\eta^{k+1}} H^{-1/2}_{{\textsf x}^{k+1}}$ and $\tilde D_2 = I$.  Optionally do further rescaling by introducing $\bar D_1$ and $\bar D_2$, and setting $D_1 := \bar{D}_1 \tilde{D}_1$ and $D_2 := \bar{D}_2 \tilde{D}_2$. Then construct the new rescaled problem \eqref{pro:rescaled problem}$_{k+1}$ using $D_1$ and $D_2$ \label{render} \; 
		Run rPDHG on rescaled problem \eqref{pro:rescaled problem}$_{k+1}$ for $\omega t$ seconds. Output the transformed solution $z^{k+1}$ and the relative error $\eps_{k+1}:= \err(x^{k+1}, y^{k+1})$ \label{lender}\;
		$k\gets k+1$ and $t \gets 2t$ \label{sender} \;
	}
	If (a) holds, then fix the new rescaling: run rPDHG on \eqref{pro:rescaled problem}$_{k}$ until a  solution $z=(x,y)$ is computed for which $\err(x, y) \le \eps$ \label{morning} \;
	If (b) holds, then revert to and fix the previous rescaling: run rPDHG on \eqref{pro:rescaled problem}$_{k-1}$ until a solution $z=(x,y)$ is computed for which $\err(x, y) \le \eps$ \label{mourning} \;	
\caption{Scheme for rPDHG with Adaptive Hessian Rescaling (\adarescaling)}\label{alg:ada_rescaling}
\end{algorithm}

\subsection{\cpipm: An interior-point method for computing a point on the central path utilizing the conjugate gradient method}\label{subsec:\cpipm}

\cpipm \ is an implementation of the practical IPM presented in Nocedal and Wright \cite[Section 14.2]{nocedal2006numerical}, and which itself is based on and is similar to Mehrotra's classic predictor-corrector primal-dual path-following method \cite{mehrotra1992implementation}. At each iteration, the IPM performs one Newton step to decrease the barrier parameter and another Newton step to return to the central path. Computing the two Newton steps requires solving two linear systems at each iteration. The method works very well in practice, although no good convergence guarantee exists for this particular (and practical) IPM. 

\cpipm \ differs from the practical IPM in \cite[Section 14.2]{nocedal2006numerical} in two ways.  First, unlike the practical IPM which directly solves the two linear equation systems at each iteration using matrix factorization, instead \cpipm \ uses an ``inner'' iterative method (at each ``outer'' iteration) to solve the (reduced) normal equations \cite[(14.44)]{nocedal2006numerical}.  The normal equation matrices are of the form $A D^2 A^\top$ for an iteration-dependent positive diagonal matrix $D$.  \cpipm \ uses the Jacobi preconditioned CGM \cite{barrett1994templates} to solve the normal equations at each iteration. This approach allows \cpipm \ to avoid formulating and factorizing $A D^2 A^\top$; and therefore the primary computational burden of \cpipm \ is the (many) matrix-vector multiplications within the CGM. The CGM stops when either (i) it reaches a sufficiently large number of iterations, or (ii) the corresponding original linear system \cite[(14.41)]{nocedal2006numerical} is appropriately solved. 

The second way that \cpipm \ differs from the practical IPM in \cite[Section 14.2]{nocedal2006numerical} has to do with pre-scaling.  At the outset, \cpipm \ runs 10 iterations of the Ruiz scaling method \cite{ruiz2001scaling} before commencing with the interior-point steps. This rescaling is designed with the aim of controlling the condition number of the matrix $A$ and related linear systems. More details on the differences between the practical IPM in \cite[Section 14.2]{nocedal2006numerical} and \cpipm \ are described in Appendix \ref{clifficlerun}. 

We note that \cpipm \ is just one type of simple first-order implementations of a classic IPM. Improved preconditioners for the CGM for IPMs have been studied as early as \cite{karmarkar1991computational,mehrotra1992implementations}. Moreover, there are other more recent matrix-factorization-free approaches for computing a point near the central path; for example,  \cite{lin2021admm,deng2024enhanced} develops an ADMM-based IPM that is as competitive as PDLP in obtaining a solution of moderate accuracy, and \cite{vladu2023interior} studies the complexity of an IPM based on quasi-Newton iterations. We chose to use \cpipm \ in our experiments because \cpipm \ is sufficiently representative of the classic first-order implementations of IPMs and it in fact demonstrates some of the benefits of central-path Hessian rescaling that is tied to general IPMs.

\subsection{Computational comparison of methods}\label{fivemiles}

We present computational experiments where we compare \adarescaling \ (Algorithm \ref{alg:ada_rescaling}) with two standard methods for LP -- namely rPDHG and a standard interior-point method -- and so we compare three methods in all as follows:
\begin{itemize}
\item {\bf \adarescaling}: This is Algorithm \ref{alg:ada_rescaling}, where we set $t=0.5$, $\omega = 6$, $\bar{\eps} = \sqrt{\eps}$, and $\hat{\eps} = \sqrt[5]{\eps}$, where $\eps = 10^{-8}$ is the relative error tolerance. We performed the optional ``further scaling'' (Line \ref{render} of Algorithm \ref{alg:ada_rescaling}) using the rescaling methods of Ruiz \cite{ruiz2001scaling} and Pock-Chambolle \cite{pock2011diagonal} as implemented in \cite{applegate2021practical}.
\item {\bf rPDHG(RuizPC)}: This an implementation of rPDHG using the scaling technique developed in PDLP \cite{applegate2021practical} which itself incorporates the scaling methods of Ruiz \cite{ruiz2001scaling} and Pock-Chambolle \cite{pock2011diagonal}. 
\item {\bf PIPM}: This is an implementation of the \textbf{P}ractical \textbf{I}nterior \textbf{P}oint \textbf{M}ethod for LP exactly as in \cite{nocedal2006numerical} that uses Cholesky factorization to solve the Newton step equations.\end{itemize}

In synch with the setup of Theorem \ref{thm overall complexity rescaled}, we replace the objective function vector $D_1^\top c$ of \eqref{pro:rescaled problem} with its projection $\bar c$ onto $\operatorname{Null}(D_2 A D_1)$, namely $\bar{c}:=\arg\min_{\hat{c}\in\operatorname{Null}(D_2 A D_1)}\|\hat{c}-D_1^\top c\|$. This is done before each time when rPDHG is run (either as a subroutine in \adarescaling \ or in rPDHG(RuizPC)).  The projection is computed by running a maximum of $1,000$ iterations of CGM. Regarding the setting of the step-sizes $\sigma$ and $\tau$ in rPDHG, since $\tilde{\lambda}_{\min}$ is not easily computable, we use the heuristic in \cite{applegate2023faster,xiong2023computational} to learn a reasonably good ratio $\tau/\sigma$ when running rPDHG. We consider five possible step-sizes pairs: $(\tau, \sigma)=\big(10^{\ell} / 2 \tilde{\lambda}_{\max }, 10^{-\ell} / 2 \tilde{\lambda}_{\max }\big)$ for $\ell=-2,-1, 0, 1, 2$. For each of these step-size pairs we run rPDHG for 10,000 iterations from the same initial point (hence 50,000 iterations in total), and then we choose which of the five step-sizes to use based on the smallest relative error it achieves. Last of all, to better take advantage of the solution $\textsf{w}^{k+1}$ of \cpipm \ when running \adarescaling, we use $\textsf{w}^{k+1}$ as a warm start for rPDHG in Line \ref{lender} of \adarescaling. 

For our implementation of PIPM we solve the normal equations via sparse Cholesky factorization using the CHOLMOD library in SuiteSparse. Whenever numerical issues occur in solving the normal equations, the diagonal entries of the matrix are shifted by $10^{-10}$, $10^{-9}$, $10^{-8}$, and so forth, until no numerical issues are present.

Finally, we note that both rPDHG and PIPM can be augmented and amended to include various other heuristics and advanced implementations that could potentially significantly enhance their practical performance. We intentionally limited our comparison to the most basic heuristics because our primary goal is to understand the general potential of rPDHG using central-path Hessian rescaling.

 The computational environment is mostly the same as described in Section \ref{subsec:APDHG_exp}, except in the choice of which problem instances are chosen from the MIPLIB 2017 dataset. Unlike in Section \ref{subsec:APDHG_exp}, we chose all MIPLIB 2017 LP relaxation instances that were not too small ($mn \ge 10^6$) but were not too dense (the number of non-zeroes \textsf{nnz} satisfies $\textsf{nnz} \le 10^5$).  This yielded 413 instances in total for which almost all methods execute a sufficiently large number of iterations. (Also note that these 413 problems do not necessarily satisfy Assumption \ref{assump:striclyfeasible}.) For each problem instance and each method we set a time limit of 5 hours. As in Section \ref{subsec:APDHG_exp}, we consider a problem instance to be solved if it achieves a relative error satisfying $\mathcal{E}_r(x, y) \le 10^{-8}$.

Figure \ref{fig:adapdhg_proportionplot} shows the fraction of solved instances (among the 413 problem instances) on the horizontal axis, and the maximum number of matrix-vector products (leftmost plot) and the maximum runtime (rightmost plot) to achieve the fraction of the problems solved. Since the main computational expense of PIPM is not the matrix-vector products, the left plot omits the information for PIPM. From Figure \ref{fig:adapdhg_proportionplot} we observe that the performance of \adarescaling \ dominates that of rPDHG(RuizPC) and all the more so for the more challenging problems.  We also see that PIPM dominates \adarescaling \ except for the easiest and the hardest problems, but that the trend indicates that \adarescaling \ outperforms PIPM on the harder problems. Overall, these results demonstrate that the Hessian rescaling has benefits even if the problem does not satisfy Assumption \ref{assump:striclyfeasible}. 

(Additionally, we also ran some experiments using \cpipm \ and we observed that \cpipm \ solved fewer than 70\% of the problem instances, which indicates that \cpipm \ alone is far less competitive than rPDHG, and indeed the advantages of \adarescaling \ mainly derive from the rescaling and not from the advanced starting point provided by \cpipm.)

\begin{figure}[htbp]
	\centering
	\includegraphics[width=1\linewidth]{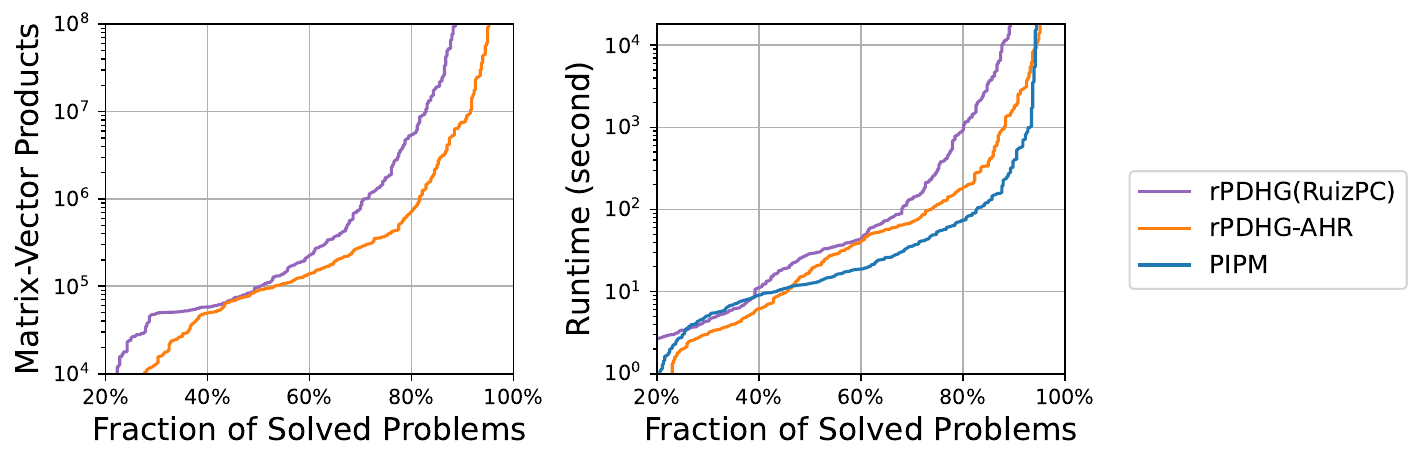}
	\caption{Performance of the three different methods on the 413 problem instances from the MIPLIB 2017 dataset.}\label{fig:adapdhg_proportionplot} 
\end{figure} 

We present more detailed comparisons between \adarescaling \ and rPDHG(RuizPC) in Table \ref{tbl:RuizPCvsAda} and Figure \ref{fig:RuizPCvsAda}, and between \adarescaling \ and PIPM in Table \ref{tbl:IPMvsAda} and Figure \ref{fig:IPMvsAda}. From Table \ref{tbl:RuizPCvsAda} we see that \adarescaling \ solved 95.2\% of the instances compared with 89.4\% for rPDHG(RuizPC). The computational bottleneck for both methods is the matrix-vector products, either within the conjugate gradient method or in the \textsc{OnePDHG} iterations. Therefore in Figure \ref{fig:RuizPCvsAda} we use the ratio of matrix-vector products required by rPDHG(RuizPC) to those required by \adarescaling \ as a measure of the speedup ratio achieved by \adarescaling. Of the 362 instances solved by both methods, Figure \ref{fig:RuizPCvsAda} shows that the advantage of \adarescaling \ over rPDHG(RuizPC) is greater for problems requiring more matrix-vector products, and the ratio is generally greater than $1.0$ for the more difficult problems, and grows to $10^2 - 10^4$. 

\vspace{10pt}
\noindent
\begin{minipage}{.45\linewidth} 
	\centering
	\vspace{31pt}
	\begin{adjustbox}{width=1\linewidth,center}
		\begin{tabular}{|cc||cc|}
		\hline
	        &   & \multicolumn{2}{c|}{\small\adarescaling}         \\ \cline{3-4}
	        &   & \multicolumn{1}{c|}{Solved}                   & Not Solved     \\ \hline\hline
		\multicolumn{1}{|c|}{\multirow{2}{*}{\small\begin{tabular}[c]{@{}c@{}}rPDHG\\ (RuizPC)\end{tabular}}} 
		& Solved & \multicolumn{1}{c|}{87.7\%}        & 1.7\% \\ \cline{2-4}
			\multicolumn{1}{|c|}{}     & Not Solved & \multicolumn{1}{c|}{7.5\%}   & 3.1\% \\ \hline
		\end{tabular}
	\end{adjustbox}
	\vspace{25pt}
	\captionof{table}{Fraction of the 413 problem instances solved and not-solved by \adarescaling \ and rPDHG(RuizPC).}\label{tbl:RuizPCvsAda}
\end{minipage}
\hfill 
\begin{minipage}{.53\linewidth} 
	\centering
	\includegraphics[width=1\linewidth]{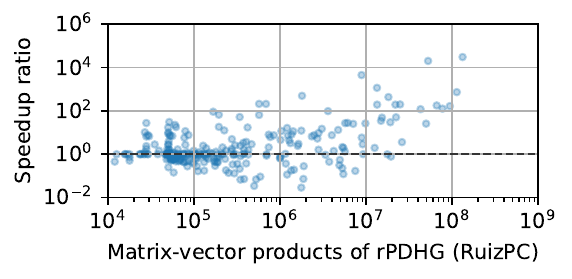}\vspace{-10pt}
	\captionof{figure}{Speedup ratio of \adarescaling \ compared with rPDHG(RuizPC) for the 362  problem instances solved by both methods.}\label{fig:RuizPCvsAda}
\end{minipage}
\vspace{20pt}

From Table \ref{tbl:IPMvsAda} we see that \adarescaling \ solved 95.2\% of the instances compared with 94.4\% for PIPM, so the two methods have nearly identical performance on this metric. In Figure \ref{fig:IPMvsAda} we use the ratio of the runtime required by rPDHG(RuizPC) to that required by \adarescaling \ as a measure of the speedup ratio achieved by \adarescaling \ over PIPM. Of the 384 instances solved by both methods, Figure \ref{fig:IPMvsAda} shows that PIPM is dominant on the easier problems, that the two methods are competitive on the intermediate-difficulty problems, and that \adarescaling \ dominates PIPM generally for problems requiring at least $100$ seconds of runtime by PIPM. This suggests in the very least that \adarescaling \ is competitive with PIPM.  We also mention that none of the 413 problems are huge-scale, and in fact are actually classified as ``small problems'' in the PDHG implementation \cite{lu2023cupdlp}.  It is therefore reasonable to expect that \adarescaling \ will exhibit more advantages over PIPM when applied to larger-scale problems.

\vspace{10pt}
\noindent
\begin{minipage}{.45\linewidth} 
	\centering
	\vspace{33pt}
	\begin{adjustbox}{width=0.9\linewidth,center}
		\begin{tabular}{|cc||cc|}  \hline
	 &   & \multicolumn{2}{c|}{\small\adarescaling}         \\ \cline{3-4}
	&   & \multicolumn{1}{c|}{Solved}                   & Not solved     \\ \hline\hline
\multicolumn{1}{|c|}{\multirow{2}{*}{\small{PIPM}}} & Solved & \multicolumn{1}{c|}{93.0\%}  & 1.5\% \\ \cline{2-4}
	\multicolumn{1}{|c|}{}                             & Not Solved & \multicolumn{1}{c|}{2.2\%}   & 3.4\% \\ \hline
		\end{tabular}
	\end{adjustbox}
	\vspace{25pt}
	\captionof{table}{Fraction of the 413 problem instances solved and not-solved by \adarescaling \ and PIPM.}\label{tbl:IPMvsAda}
\end{minipage}%
\hfill 
\begin{minipage}{.53\linewidth} 
	\centering
	\includegraphics[width=1\linewidth]{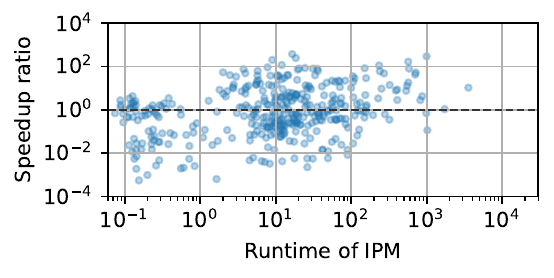}\vspace{-10pt}
	\captionof{figure}{Speedup ratio of \adarescaling \ compared with PIPM for the 384  problem instances solved by both methods.}\label{fig:IPMvsAda}
\end{minipage}
\vspace{20pt}

\subsection{An ``ideal'' central-path Hessian rescaling}\label{subsec:interior_point_rescaling}

By its design, \adarescaling \ (Algorithm \ref{alg:ada_rescaling}) heuristically tries to balance the cost of obtaining a good rescaling with the benefit of computational savings when running rPDHG. As such, \adarescaling \ is one of many different strategies that might be used to better (or best) achieve such balance of cost/benefit.  In this subsection we aim to demonstrate what we might expect from a better (or ideal) heuristic for balancing the cost/benefit. Specifically, we investigate using a Hessian rescaling for rPDHG that ideally balances the computational cost of \cpipm \ and the computational benefit for rPDHG. To do so, we first run \cpipm \ for $t$ seconds, for a variety of different values of $t = t_1, t_2, \ldots, t_M$, yielding $M$ different candidate Hessian-rescaled problems.  Then, for each $i=1, \ldots, M$ we run rPDHG for the associated rescaled problem, and we record the time $\hat t_i = \hat t_1, \hat t_2, \ldots, \hat t_M$ used to a obtain a solution with relative error at most $10^{-8}$.  Last of all, we choose the value of $t_i$ that achieves the best overall runtime $t_i + \hat t_i$.  Here is the description of our method, which we denote as \textbf{\bestrescaling}.

 \begin{itemize}
	\item \textbf{\bestrescaling}: This is essentially Algorithm \ref{alg:ada_rescaling} modified to have a single outer loop that calls \cpipm \ to run for $t_i$ seconds, where $t_i = 2^{i}/4$ for $i=1, \ldots, M$. Then we continue with Line \ref{render} of Algorithm \ref{alg:ada_rescaling} after choosing the Hessian rescaling associated with the value of $t_i$ that achieves the shortest overall runtime (for running \cpipm \ and then running rPDHG) to obtain a solution with relative error at most $10^{-8}$. As in \adarescaling, we perform the optional ``further scaling'' using the rescaling methods of Ruiz \cite{ruiz2001scaling} and Pock-Chambolle \cite{pock2011diagonal} as implemented in \cite{applegate2021practical}.
\end{itemize}
Additionally, \bestrescaling \ uses other logic just as in \adarescaling. For examples, \bestrescaling \ replaces the objective vector with the projection onto $\operatorname{Null}(D_2 A D_1)$ and uses the same logic as in \adarescaling \ to learn a reasonably good ratio between $\tau$ and $\sigma$ when running rPDHG. 

By selecting the ``best'' value of $t_i$ in \bestrescaling, we aim to get a sense of what is the limit of what a good/better heuristic might achieve. We note that by its intention \bestrescaling \ is not a practical method.  However, it does allow us to gauge the potential performance of a possible improved heuristic for central-path Hessian rescaling. 

Figure \ref{fig:ideal_proportionplot} shows the fraction of solved instances using \bestrescaling \ (among the same 413 problem instances tested in Section \ref{fivemiles}) on the horizontal axis, and the maximum runtime to achieve the fraction of the problems solved. From Figure \ref{fig:ideal_proportionplot} we observe that an ``ideal'' choice of the central-path Hessian rescaling is consistently 2 to 3 times faster than \adarescaling. There is thus room for improving \adarescaling \ to close this performance gap. Furthermore, \bestrescaling \ also consistently outperforms PIPM in terms of the computation time (on most problems) and the number of problems solved. This result shows the potential significant benefits of central-path Hessian rescaling if one can further improve the balance between the cost of computing the rescaling and the benefit of shortening the runtime of rPDHG.

\begin{figure}[htbp]
	\centering
	\includegraphics[width=0.65\linewidth]{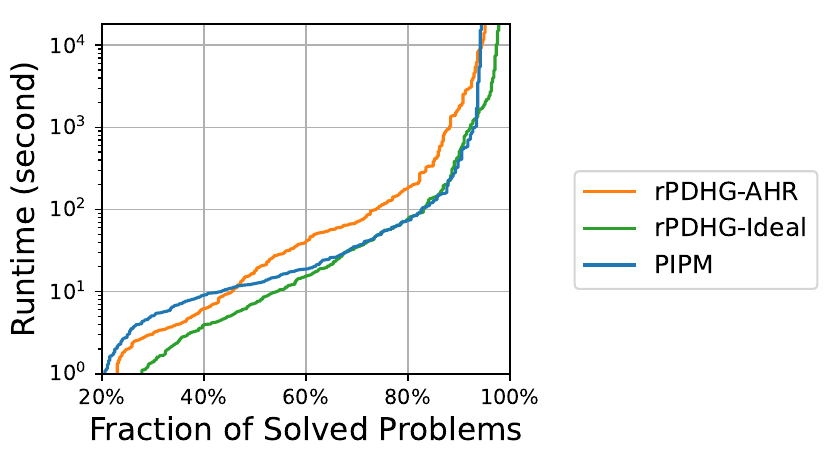}
	\caption{Performance of \ideal, \adarescaling,  and PIPM on the 413 problem instances from the MIPLIB 2017 dataset.}\label{fig:ideal_proportionplot} 
\end{figure}

\section{Final Remarks, Questions, and Research Directions}

In this paper we have extended the theoretical, algorithmic, and practical performance of rPDHG for general CLP, and we have performed computational experiments for the special case of LP.  For the lens of theory, we have presented three geometric condition numbers of the primal-dual sublevel set $\calW_\delta$: the diameter $D_\delta$, the conic radius $r_\delta$, and the Hausdorff distance to the optimal solution set $d_\delta^H$, and we have shown how these condition numbers inform the convergence rate of rPDHG both in theory and in practice. For the lens of algorithmic development, we have proposed to use central-path Hessian rescaling transformations to improve the geometry of the primal-dual sublevel sets with the overarching aim of improving the convergence rate of rPDHG -- both in theory and practice. Last of all, we have presented computational results that verify the theory and  and demonstrate how such central-path Hessian rescaling can significantly improve the performance of rPDHG for LP in practice. 

We end this paper with a short list of open questions for further investigation:

\textbf{1. Extensions to other primal-dual algorithms.} We expect that the use of the sublevel set geometric condition numbers for informing convergence is not limited to just PDHG. Indeed, the only properties of rPDHG that we used were those in Lemmas \ref{lm: nonexpansive property}, \ref{lm: R in the opt gap convnergence} and \ref{lm: original sublinear PDHG}. Similar properties also hold for other primal-dual first-order algorithms such as ADMM and EGM \cite{applegate2023faster,ryu2022large}.  For this reason we expect that much of our analysis should extend to these other primal-dual first-order methods as well.

\textbf{2. Condition numbers and analysis for more general nonlinear programs.} In this paper we studied conic linear problems (CLP) which are a (rather important) subclass of constrained convex optimization. It would be interesting to explore how the sublevel-set geometry influences the convergence rate of first-order methods for more general convex and nonconvex constrained optimization, which might also lead to methods for improving the sublevel set geometry as we have done here. 

\textbf{3. Condition numbers and analysis for the infeasibility detection problem.} It has been shown in \cite{applegate2024infeasibility} that rPDHG can also be used to detect infeasibility for LP problem instances. It would be interesting to study how the geometry of the problem informs the speed of infeasibility detection, and how to possibly improve the geometry to enhance practical algorithm performance.

\textbf{4. Other methods for improving the geometry of sublevel sets.} In this paper we have used \cpipm \ to obtain a good Hessian rescaling, but of course one can devise many other types of schemes to obtain good/better Hessian rescalings.  It would be interesting to see if there are other methods that can find a good Hessian rescaling faster than \cpipm \ or achieve a better balance between the cost of computing the rescaling and the benefit of improving the runtime of rPDHG. Perhaps other types of rescalings and associated problem transformations may improve the geometry and enhance the convergence rate of rPDHG, especially for cones (symmetric or not) for which we do not have logarithmically homogeneous self-concordant barriers.

\section*{Acknowledgements} We thank Haihao Lu, Kim-Chuan Toh, and Haoyue Wang for useful discussions and feedback on this work. 

\section*{Appendix}

\appendix

\section{Computing the Normalized Duality Gap}\label{appendix:compute_rho}

We first show in Section \ref{subsec:compute_rho} a methodology for computing the normalized duality gap $\rho(r;z)$.  In Section \ref{subsec:approx_rho} we modify this methodology by changing the norm from the $M$-norm $\|\cdot\|_M$ to a different norm $\|\cdot\|_N$ which we call the $N$-norm, and whose associated normalized duality gap is denoted by $\rho^N(r;z)$. Throughout this appendix we assume that  $\sigma, \tau$ satisfy \eqref{eq:general_stepsize}, which then implies that $M \succeq 0$ \eqref{robsummer}. The strategy presented herein for computing and approximating $\rho(r;z)$ is a generalization of the method developed by \cite{applegate2023faster} for the particular case of LP.

\subsection{Computing the normalized duality gap}\label{subsec:compute_rho}
Computing $\rho(r;z)$ for  $z\in\bar{K}:= K_p \times \mathbb{R}^m$ and $r > 0$ is basically equivalent to solving the following convex optimization problem:
\begin{equation}\label{pro:compute_rho}
	\max_{ \hat{z} = (\hat{x},\hat{y}): \ \hat{x}\in K_p, \|\hat{z} -z\|_M \le r }\left[
		L(x,\hat{y}) - L(\hat{x},y)
		\right] = \left(
	\begin{array}{ll}
			\max_{\hat{z}} \       & h^\top (\hat{z} - z)                           \\
			\operatorname{s.t.} \  & \hat{z}\in \bar{K}, \ \|z-\hat{z}\|_M^2\le r^2
		\end{array}
	\right)
\end{equation}
in which $z = (x,y)$ and $h = \begin{pmatrix}
	h_1 \\
	h_2
\end{pmatrix}:=
\begin{pmatrix}
	A^\top y - c \\
	b - A x
\end{pmatrix}\in\mathbb{R}^{n+m}$. Suppose that $\hat{z}^\star$ is an optimal solution of \eqref{pro:compute_rho}; then $\rho(r;z)$ is obtained by
$$
\rho(r;z) = \frac{h^\top \big(\hat{z}^\star - z\big)}{r}\ .
$$

We now show how to construct an optimal solution of \eqref{pro:compute_rho}. Consider the following parameterized optimization problem over the parameter $t \ge 0$, with optimal solution $z(t)$ :
\begin{equation}\label{eq:z_t}
	z(t):= \arg\max_{\tilde{z}\in \bar{K}}  \ 
	t\cdot h^\top (\tilde{z} - z) - \|\tilde{z} - z\|_M^2 \ .
\end{equation} 
This problem essentially replaces the norm constraint $\|\hat{z} - z\|_M^2 \le r^2$ in \eqref{pro:compute_rho} with a penalty term $-\tfrac{\|\hat{z} - z\|_M^2}{t}$ in the objective function. The following lemma shows that solving \eqref{pro:compute_rho} is essentially a root-finding problem of the univariate function $f(t):=\|z - z(t)\|_M - r$ defined on $t \in [0,\infty)$.

\begin{lemma}\label{lm:compute_rho}
	Suppose $t^\star$ satisfies $t^\star>0$ and $f(t^\star) = 0$. Then $z(t^\star)$ is an optimal solution of \eqref{pro:compute_rho}. 
\end{lemma}
\noindent
Based on this lemma, we consider using the bijection method to find the root of $f(t)$ in the region $t \in (0,\infty)$. Notice that $f(0) < 0$.  We can compute $f(t)$ for an increasing sequence of values of $t$, for example $t_k := 2^k$ for $k=1, 2, \ldots$, until we obtain $k$ for which $f(t_k) \ge 0$. Let $K$ denote the first value of $k$ for which $f(t_k) \ge 0$. Then $f(t_{K-1})$ has a different sign than $f(t_{K})$, and so $[t_{K-1},t_K]$ contains a root of $f(t)$ which can be computed using the bijection method. In the special case that none of the $t_k$ satisfy $f(t_k)\ge 0$, the following lemma shows that $\frac{h^\top (z(t_k) - z)}{r}$ itself converges to $\rho(r;z)$ at a conveniently bounded rate.

\begin{lemma}\label{lm:compute_rho_2}  If $t>0$ and $f(t) < 0$, then $\left|\frac{h^\top (z(t) - z)}{r} - \rho(r;z)\right|\le \frac{r}{t}$ . 
\end{lemma}

The main computational bottleneck of the above scheme is solving \eqref{eq:z_t}. If $M$ is full-rank (i.e., \eqref{eq:general_stepsize} holds strictly), the objective function of \eqref{eq:z_t} is strongly convex and smooth. Furthermore, let us presume that the task of computing the projection onto $\bar{K}$ under the Euclidean norm is reasonable (as it is for the nonnegative orthant and the cross-product of second-order cones for example).  Then the projected gradient descent method and its accelerated versions can be applied to \eqref{eq:z_t} with linear convergence rates, see \cite{lan2020first}. Moreover, since the restart condition does not need to be checked frequently in practice, the cost of computing the normalized duality gap can be further reduced if not computed very often. 

\subsubsection{Proofs of Lemmas \ref{lm:compute_rho} and \ref{lm:compute_rho_2}}
We now proceed with the proofs of Lemmas \ref{lm:compute_rho} and \ref{lm:compute_rho_2}. We first recall the optimality conditions for problems \eqref{pro:compute_rho} and \eqref{eq:z_t}.
Since strong duality holds for \eqref{eq:z_t}, for each $t \ge 0$ the optimal solution $z(t)$ of \eqref{eq:z_t} must satisfy the following conditions:
	\begin{equation}\label{eq:z_t_opt}
	z(t) \in \bar{K}, \ s:= 2M z(t) - 2Mz - t\cdot h \in  \bar{K}^*, \text{ and  } (z(t))^\top s = 0 \ .
	\end{equation} 
Regarding problem \eqref{pro:compute_rho}, $\hat{z}^\star$ is an optimal solution of \eqref{pro:compute_rho} if there exists a scalar multiplier $\lambda^\star$ that together with $\hat{z}^\star$ satisfy the KKT optimality conditions:
	\begin{equation}\label{eq:rho_KKT}
		\text{Inclusions: } \hat{z}^\star \in \bar{K}, \ s^\star :=    2 \lambda^\star M\hat{z}^\star - 2 \lambda^\star Mz -h \in \bar{K}^*, \ \lambda^\star \ge 0,  \  \|z-\hat{z}^\star\|_M^2\le r^2,   \text{ and }
	\end{equation}
	\begin{equation}\label{eq:rho_KKT_1}
		\text{Complementarity: } \  (\hat{z}^\star)^\top s^\star = 0 , \ \text{ and }  \lambda^\star \cdot \left( r^2 -  \|z-\hat{z}^\star\|_M^2\right) = 0 \ . \ \ \ \ \ \ \ \ \ \ \ \ \ \ \ \ \ \ \ \ \ \ \ \ \ \ \ \ 
	\end{equation}

\begin{proof}[Proof of Lemma \ref{lm:compute_rho}]
	If $t^\star>0$ and $\|z - z(t^\star)\|_M = r$, then it follows from \eqref{eq:z_t_opt} and $f(t^\star) = 0 $ that $z^\star  := z(t^\star)$ and $\lambda^\star := \frac{1}{t^\star}$ satisfy the optimality conditions \eqref{eq:rho_KKT} and \eqref{eq:rho_KKT_1}, and therefore $z(t^\star)$ is an optimal solution of \eqref{pro:compute_rho}.	
\end{proof}

\begin{proof}[Proof of Lemma \ref{lm:compute_rho_2}]	
	We first suppose that $M$ is positive definite, in which case using a standard Lagrangian construction one can derive the following dual problem of \eqref{pro:compute_rho}:
\begin{equation}\label{bertsekas}
	\min_{ \lambda \ge0, \ s \in \bar K^* } \frac{1}{4\lambda}\|s+h\|_{M^{-1}}^2 + z^\top s + \lambda r^2\end{equation}
Now define $\lambda:=\frac{1}{t}$ and $s:= 2\lambda M z(t) - 2\lambda Mz -  h$, whereby from \eqref{eq:z_t_opt} it follows that $(\lambda,s)$ is feasible for \eqref{bertsekas}.  Also, since $f(t)<0$ we have $z(t)$ is feasible for \eqref{pro:compute_rho} and the duality gap of this pair of primal and dual solution works out to be exactly $\lambda(r^2 - \|z-z(t) \|_M^2) = \tfrac{r^2 - \|z-z(t) \|_M^2}{t} $ which can be verified by simple arithmetic manipulation. Let the optimal objective value of \eqref{pro:compute_rho}  be $g^\star$; then $|h^\top (z(t) - z)  - g^\star | \le \tfrac{r^2 - \|z-z(t) \|_M^2}{t}$, from which it follows that $\left|\frac{h^\top (z(t) - z)}{r} - \rho(r;z)\right| = \left|\frac{h^\top (z(t) - z)}{r} - \frac{g^\star}{r}\right| \le \tfrac{r^2 - \|z-z(t) \|_M^2}{rt} \le \tfrac{r^2}{rt} = \tfrac{r}{t} $. This proves the result for the case when $M$ is positive definite.  
	
If $M$ is not positive definite, then under the assumption that $\sigma, \tau$ satisfy \eqref{eq:general_stepsize} we have $M \succeq 0$ \eqref{robsummer}. In this case the dual problem of \eqref{pro:compute_rho} no longer has the very convenient expression \eqref{bertsekas}, but all of the properties of the proof in the previous paragraph follow nevertheless. \end{proof}

Note that the proofs of Lemmas \ref{lm:compute_rho} and \ref{lm:compute_rho_2} are also valid if we replace $M$ by another positive semidefinite matrix $\tilde{M}$ to define a $\tilde M$-norm, and let us denote the normalized duality gap using the $\tilde M$-norm as $\rho^{\tilde{M}}(r;z)$. In Section \ref{subsec:approx_rho} we will show that with a proper choice of $\tilde{M}$ that $\rho^{\tilde{M}}(r;z)$ will provide a good approximation of $\rho^M(r;z)$ but with significantly lower computational cost of solving \eqref{eq:z_t} .

\subsection{Approximating  the normalized duality gap}\label{subsec:approx_rho}

In Section \ref{subsec:compute_rho} we showed that computing $\rho(r;z)$ can be accomplished by parametrically solving the optimization problem \eqref{eq:z_t}, which is equivalent to a certain projection onto the cone $\bar{K}=K_p\times \mathbb{R}^m$ in the $M$-norm. Although PDHG is premised on the notion that Euclidean projections onto $\bar{K}$ are simple to compute (see Lines \ref{line:update_x} and \ref{line:update_y} of \textsc{OnePDHG} in Algorithm \ref{alg: one PDHG}), projections onto $\bar K$ under the $M$-norm might be significantly more difficult.  In this subsection we describe how to efficiently approximate $\rho(r;z)$ by working with a different matrix norm, namely the $N$-norm which was introduced in the proof of Lemma \ref{lm change of norm}, and for which the equivalent optimization problem \eqref{eq:z_t} works out to be a Euclidean projection onto $K_p$.

The $N$-norm is the matrix norm $\|z\|_N$ in which $N := \Big(\begin{smallmatrix}
		\frac{1}{\tau}I_n &                     \\
		                  & \frac{1}{\sigma}I_m
	\end{smallmatrix}\Big)$, which was introduced in the proof of Lemma \ref{lm change of norm}. Let $\rho^N(r;z)$ denote the corresponding normalized duality gap function in $N$-norm.  We now show that in the $N$-norm, solving $z(t)$ of \eqref{eq:z_t} is simply a Euclidean projection onto $K_p$. Because $\bar{K} = K_p \times \mathbb{R}^m$ and $\|z\|_N^2 = \frac{1}{\tau}\|x\|^2 + \frac{1}{\sigma} \|y\|^2$, \eqref{eq:z_t} can be separated into two independent problems:
\begin{equation}\label{eq:z_t_2}
\begin{aligned}
	z(t) = (x(t),y(t)) & = \Big(
	\arg\max_{\tilde{x} \in K_p} \ t\cdot h_1^\top \tilde{x} - \tfrac{1}{\tau}\|\tilde{x} - x\|^2,  
	\arg\max_{\tilde{y} \in \mathbb{R}^m} \ t\cdot h_2^\top \tilde{y} - \tfrac{1}{\sigma}\|\tilde{y} - y\|^2\Big) \\
	& = \left(P_{K_p}\left(x + \tfrac{t\tau}{2} \cdot h_1\right), y + \tfrac{t\sigma}{2}\cdot h_2 \right) \ .
\end{aligned}
\end{equation} 
Hence the primary computational cost of computing $z(t)$ in the $N$-norm is just the Euclidean projection onto $K_p$, which is no more of a computational burden than Line \ref{line:update_x} of \textsc{OnePDHG} in Algorithm \ref{alg: one PDHG}, and might be considerably easier than the $M$-norm projection onto $\bar{K}$.

Furthermore, the following proposition shows that $\rho^N(r;z)$ is equivalent to $\rho(r;z)$ up to a constant factor so long as the step-sizes are chosen a bit conservatively.
\begin{proposition}\label{pr:rhoNrho}
	If $\tau,\sigma$ satisfy \eqref{eq:general_stepsize} strictly, then for any $z\in \bar{K}$ and $r>0$ it holds that:
	\begin{equation}\label{eq:lm:rhoNrho}
		\frac{1}{\sqrt{2}}\cdot \rho^N\left( r;z\right) \le
		\rho(r;z) \le \frac{1}{\sqrt{1-\sqrt{\tau \sigma}\lambda_{\max }}}\cdot \rho^N\left( r;z\right)  \ .
	\end{equation}
\end{proposition}
\noindent
For example, if $\sqrt{\sigma\tau} = \tfrac{1}{2\lambda_{\max}}$, then \eqref{eq:lm:rhoNrho} becomes $\frac{1}{\sqrt{2}}\cdot \rho^N\left( r;z\right) \le  \rho(r;z) \le \sqrt{2}\cdot \rho^N\left( r;z\right)$. In practice we have used $\rho^N\left( r;z\right)$ instead of $\rho\left( r;z\right)$ to evaluate the restart condition, and it can be proven that a computational guarantee of rPDHG still holds. This technique has also been used in \cite{applegate2023faster,applegate2021practical}. 

The proof of Proposition \ref{pr:rhoNrho} uses the following two lemmas.
\begin{lemma}{\bf(Proposition 2.8 of \cite{xiong2023computational})}\label{lm:MNnorms}
	If $\tau,\sigma$ satisfy \eqref{eq:general_stepsize}, then for any $z\in\mathbb{R}^{n+m}$ it holds that $\sqrt{1-\sqrt{\tau \sigma} \lambda_{\max }} \cdot\|z\|_N \leq\|z\|_M \leq \sqrt{2}\|z\|_N$.
\end{lemma}
\begin{lemma}{\bf(Proposition 5 of \cite{applegate2023faster})}\label{lm:monotonicity_of_rho}
	It holds for any $z$ that $\rho^N(r;z)$ is a monotonically non-increasing for $r\in[0,\infty)$.
\end{lemma}
\begin{proof}[Proof of Proposition \ref{pr:rhoNrho}]
	From Lemma \ref{lm:MNnorms} we have $\big\{\hat{z}:\|\hat{z} - z\|_N\le  \frac{r}{\sqrt{2}}
	\big\}\subseteq \left\{\hat{z}:\|\hat{z} - z\|_M\le r
	\right\} \subseteq \big\{\hat{z}:\|\hat{z} - z\|_N\le  \frac{r}{\sqrt{1-\sqrt{\tau \sigma}\lambda_{\max }}}
	\big\}$, which leads to
	\begin{equation}\label{eq:lm:rhoNrho 2}
		\textstyle
		\frac{r}{\sqrt{2}}\cdot \rho^N\left( \frac{r}{\sqrt{2}};z\right) \le
		r\cdot \rho(r;z) \le \frac{r}{\sqrt{1-\sqrt{\tau \sigma}}}\cdot \rho^N\Big( \frac{r}{\sqrt{1-\sqrt{\tau \sigma}\lambda_{\max }}};z\Big)  \
	\end{equation}
	for any $r$ and $z$, because of the inclusion relationship between the feasible sets of the corresponding optimization problems. Finally, applying Lemma \ref{lm:monotonicity_of_rho} to \eqref{eq:lm:rhoNrho 2} yields \eqref{eq:lm:rhoNrho}.
\end{proof}

\section{Proof of \eqref{xcski}}\label{app:proof of lm r calW small}

We prove the following lemma regarding lower and upper bounds on $\sup_{\gamma>0} \frac{r_\gamma}{\gamma}$.
\begin{lemma}\label{lm r_delta over delta}
	Under Assumption \ref{assump:striclyfeasible}, it holds that
	\begin{equation}\label{eq of lm r calW small}
		\frac{\width_K}{\max_{w\in\calW^\star}\|w\|}\le \sup_{\gamma >0}\frac{r_\gamma}{\gamma} \le \frac{1}{\max_{w\in\calW^\star}\|w\|} \ ,
	\end{equation}
	in which $\width_K$ is the width of the cone $K:=K_p\times K_d$ defined in \eqref{eq:width_K}. \end{lemma}

Before presenting the proof we first discuss a related result.
For the general primal and dual problems \eqref{pro: general primal clp} and \eqref{pro: general dual clp on s}, \cite{freund2003primal} proves a geometric relationship between the primal and dual sublevel sets of \eqref{pro: general primal clp} and \eqref{pro: general dual clp on s} which we now describe. Under Assumption \ref{assump:striclyfeasible} the problems \eqref{pro: general primal clp} and \eqref{pro: general dual clp on s} have a common optimal objective value $f^\star$.  Then for any $\eps,\delta\in\mathbb{R}_+$, define
\begin{equation}\label{eq:def_R_r}
	\bar{R}_\eps:=
	\left(\begin{array}{cc}
			\max\limits_{x}     & \|x\|                         \\
			\operatorname{s.t.} & x \in V_p \cap K_p            \\
			                    & c^\top x \le  f^\star + \eps
		\end{array}
	\right)
	\ \text{ and } \ \bar{r}_\delta:=  \left(\begin{array}{cc}
			\max\limits_{s}     & \max\limits_{r:B(s,r) \subseteq K_d}r \\
			\operatorname{s.t.} & s \in V_d \cap K_d                    \\
			                    &  - q^\top s + q_0 \ge f^\star - \delta
		\end{array}
	\right) \ .
\end{equation}
Recall the definition of the width $\width_K$ of cone $K$ in \eqref{eq:width_K}, and then the product of $\bar{R}_\eps$ and $\bar{r}_\delta$ has both lower bound and upper bounds.
\begin{lemma}{\bf (Geometric relationship between primal and dual sublevel sets, Theorem 3.2 of \cite{freund2003primal})}\label{lm:primal-dual-sublevelsets}
	Under Assumption \ref{assump:striclyfeasible}, for any $\eps,\delta\in\mathbb{R}_+$ it holds that
	\begin{equation}
		\width_{K_d} \cdot \min\{\delta, \eps\} \le  \bar{r}_\delta\bar{R}_\eps  \le \delta + \eps \ .
	\end{equation}
\end{lemma}
\noindent Lemma \ref{lm:primal-dual-sublevelsets} leads directly to the following corollary regarding the conic radius $r_\delta$ of the sublevel set $\calW_\delta$:
\begin{corollary}\label{cor: r R in symmetric form}
	For any $\eps, \delta \in\mathbb{R}_+$, it holds that
	\begin{equation}\label{eq:r R in symmetric form}
		\width_K \cdot \min\{\delta,\eps\} \le r_\delta \cdot \left(\max_{w\in\calW_\eps}\|w\|\right) \le \delta + \eps \ .
	\end{equation}
\end{corollary}
\begin{proof}
	The idea of the proof is to relate $r_\delta$ and $\max_{w\in\calW_\eps}\|w\|$ to the quantities $\bar{r}_\delta$ and $\bar{R}_\eps$ defined in \eqref{eq:def_R_r} for a certain pair of primal and dual problems. Since any change of the objective vector $c$ in $\operatorname{Im}(A^\top)$ does not change the sublevel sets $\calW_\delta$ or $\calW_\eps$, without loss of generality we presume that $c \in \operatorname{Null}(A)$. We can combine \eqref{pro: general primal clp} and \eqref{pro: general dual clp on s} into a single conic linear optimization problem of the following form:
	\begin{equation}\label{pro:primal-dual LP problem}
		\begin{aligned}
			\min_{w\in\mathbb{R}^{2n}} \  & z_0^\top w \quad
			\text{s.t.} \                 & w \in V = \linV + w_0, \ w \in K = K_p \times K_d \ , 
		\end{aligned}
	\end{equation}
where $z_0 := (c,q) \in \linV$ and $w_0 := (q,c)$, yielding $V = \linV + w_0 = (\linVp + q )\times (\linVd + c) = V_p \times V_d$.  Using a similar approach for deriving the dual problem with that in Section \ref{sec:get_general_clp_dual}, it is known that its symmetric dual problem (see also Section 3.1 of \cite{renegar2001mathematical} for details) is the following problem:
	\begin{equation}\label{pro:dual of the primal-dual LP problem}
		\begin{aligned}
			\min_{z\in\mathbb{R}^{2n}} \  & w_0^\top z \quad
			\text{s.t.} \                 & z \in \tilde{V} := \linV^\bot + z_0 \ , \ z \in K^* = K_d \times K_p \ . 
		\end{aligned}
	\end{equation}
	Note that  $\lin{V} = \linV_p \times \linV_d$ in which $\linV_p$ and $\linV_d$ are orthogonal complementary subspaces of $\mathbb{R}^n$, and so $\linV_d \times \linV_p$ and $\linV_p\times \linV_d$ are orthogonal complementary subspaces of $\mathbb{R}^{2n}$ and thus $\linV^\bot$ = $\linV_d \times \linV_p$. Moreover, because $K_p$ and $K_d$ are dual cones of each other, then $K^* = K_p^* \times K_d^* = K_d \times K_p$. Therefore, now we can see the feasible sets of \eqref{pro:primal-dual LP problem} and \eqref{pro:dual of the primal-dual LP problem} are related by simply exchanging the order of the $n$-dimensional variable components, namely from $w=(x,s)$ to $z=(s,x)$.  And also the objective vectors $z_0$ and $w_0$ are similarly related by exchanging the order of $c$ and $q$.

It thus follows that the quantity $\bar{r}_\delta$ associated with \eqref{pro:primal-dual LP problem} is identical to the quantity $r_\delta$ of $\calW_\delta$, and quantity $\bar{R}_\eps$ associated with \eqref{pro:dual of the primal-dual LP problem} is identical to $\max_{w\in\calW_\eps}\|w\|$. Furthermore, because $K = K_p\times K_d$ and $K^* = K_d \times K_p$, then $\tau_K = \tau_{K^*}$.
	Therefore, directly applying Lemma \ref{lm:primal-dual-sublevelsets} on \eqref{pro:primal-dual LP problem} and  \eqref{pro:dual of the primal-dual LP problem}  yields \eqref{eq:r R in symmetric form}.
\end{proof}

Furthermore, the following monotonicity result is presented in Remark 2.1 of \cite{freund2003primal}.
\begin{lemma}{\bf (Monotonicity of $\bar{r}_\delta/\delta$)}\label{lm:monotonicity of delta r}
	For any $\delta' > \delta >0$, it holds that $\frac{\bar{r}_{\delta} }{\delta}\ge \frac{\bar{r}_{\delta'}}{\delta'}$.
\end{lemma}

Using Corollary \ref{cor: r R in symmetric form} and Lemma \ref{lm:monotonicity of delta r}, we now prove Lemma \ref{lm r_delta over delta}.
\begin{proof}[Proof of Lemma \ref{lm r_delta over delta}]
	We first prove the second inequality of \eqref{eq of lm r calW small}.
	According to the second inequality of \eqref{eq:r R in symmetric form} with $\eps = 0$, it holds that $r_\delta \cdot \max_{w \in \calW^\star}\|w\| \le \delta$ for any $\delta > 0$, which directly implies the second inequality of \eqref{eq of lm r calW small}.

As for the first inequality of \eqref{eq of lm r calW small}, the first inequality of \eqref{eq:r R in symmetric form} yields $r_\delta \cdot \max_{w \in \calW_\eps}\|w\| \ge \width_K \cdot \min\{\delta,\eps\}$ for any $\delta ,\eps \in \mathbb{R}_{++}$. Taking $\eps = \delta$ yields $\frac{r_\delta}{\delta}\ge \frac{\width_K}{\max_{w \in \calW_\delta}\|w\|} $ for any $\delta > 0$. And using Lemma \ref{lm:monotonicity of delta r} we have $\sup_{\delta>0}\frac{r_\delta}{\delta} = \lim_{\delta\searrow 0}\frac{r_\delta}{\delta} \ge \lim_{\delta \searrow 0} \frac{\width_K}{\max_{w \in \calW_\delta}\|w\|} = \frac{\width_K}{\max_{w\in \calW^\star}\|w\|}$, which proves the first inequality of \eqref{eq of lm r calW small}.
\end{proof}

\section{Proof of Theorem \ref{thm overall complexity lp}}\label{taxtime}

Linear programming problems enjoy a ``sharpness'' property that is not guaranteed for the more general conic optimization problem \eqref{pro: general primal clp}.  Similar in spirit to \cite{xiong2023computational} we define the (primal-and-dual) \textit{PD sharpness} $\mu$ as follows:
\begin{equation}\label{eqdef:LPsharpness}
	\mu:= \inf_{w \in \calF \setminus \calW^\star}\frac{\dist(w, V \cap \{w: \gap(w)= 0\})}{\dist(w,\calW^\star)} \ .
\end{equation}

Recalling the definition of the best suboptimal extreme point gap $\bar\delta$ from Definition \ref{def:best_suboptimal_gap}, a key property that characterizes $\mu$ using $\bar\delta$ is the following lemma which is taken from \cite{xiong2023computational}, albeit using different notation.
\begin{lemma}\label{lm:LPsharpness} {\bf (essentially Theorem 5.5 of \cite{xiong2023computational})}
	If both \eqref{pro: general primal clp} and \eqref{pro: general dual clp on s} have feasible and nonoptimal solutions, then for any $\delta \in (0, \bar{\delta}]$ it holds that
	\begin{equation}\label{eq:LPsharpness_keyslice}
		\mu = \inf_{w \in \calF \cap \{w: \gap(w) = \delta\}}\frac{\dist(w, V \cap \{w: \gap(w)= 0\})}{\dist(w,\calW^\star)} \ .
	\end{equation}
\end{lemma}
\noindent The following lower bound of $\mu$ using $d^H_\delta$ and $r_\delta$ is a direct implication of Lemma \ref{lm:LPsharpness}.
\begin{lemma}\label{lm:lowerbound_mu}
	Under the hypothesis of Lemma \ref{lm:LPsharpness}, for any $\delta\in(0,\bar{\delta}]$ it holds that
	\begin{equation}\label{eq:lowerbound_mu}
		\mu \ge  \frac{\delta}{d^H_\delta}\cdot \frac{1}{\sqrt{\|P_{\linVp}(c)\|^2 + \|q\|^2}} \ .
	\end{equation}
\end{lemma}
\begin{proof}
	For $ w \in V$ satisfying $\gap(w) = \delta$, the numerator of \eqref{eq:LPsharpness_keyslice} has the closed form 		$$
		\dist(w, V \cap \{w: \gap(w)= 0\}) =  \frac{\delta}{\|P_{\linV}([c, q])\|} = \frac{\delta}{\sqrt{\|P_{\linVp}(c)\|^2 + \|q\|^2}} \ .
	$$
	Also, from the definition of $d^H_\delta$ in Definition \ref{def distance to optima} we have:
	$$
		d^H_\delta = \max_{w \in \calF \cap\{w:\gap(w)\le \delta\}} \dist(w,\calW^\star) \ge \max_{w \in \calF \cap\{w:\gap(w) = \delta\}} \dist(w,\calW^\star) \ ,
	$$
	whereby from
	\eqref{eq:LPsharpness_keyslice}  we have for any $\delta \in (0, \bar{\delta}]$ that
	$$
		\mu =  \frac{\delta}{\sqrt{\|P_{\linV}(c)\|^2 + \|q\|^2}} \cdot \inf_{w \in \calF \cap \{w: \gap(w) = \delta\}}\frac{1}{\dist(w,\calW^\star)}  \ge
		\frac{\delta}{\sqrt{\|P_{\linV}(c)\|^2 + \|q\|^2}} \cdot \frac{1}{d^H_\delta} \ ,
	$$
	which is exactly \eqref{eq:lowerbound_mu}.
\end{proof}

\begin{lemma}\label{lm c123 sharp LP}
	Suppose that $c\in\operatorname{Null}(A)$. For any $\delta \in ( 0,\bar{\delta}]$ and any $w:=(x,s)$ it holds that
	\begin{equation}\label{eq lm c123 sharp LP}
		\dist(w,\calW^\star) \le 	\frac{ d^H_\delta   }{\delta} \cdot\gap(w) +  \frac{5D_\delta}{r_\delta} \cdot \max\{ \dist(w,K) , \dist(w, V)\} \ .
	\end{equation}
\end{lemma} 
\begin{proof}
	We first consider the case when $\gap(w)\le \delta$. Let $\bar{w} := P_V(w)$. Then because $c \in \operatorname{Null}(A) = \linVp$ and $q \in \operatorname{Im}(A^\top) = \linVd$, we have $\gap(w) = \gap(\bar{w})$ and hence $\gap(\bar{w}) = \gap(w)\le \delta$.
	Let $(w_\delta, r_\delta)$ be the conic center and conic radius of $\calW_\delta$ as defined in Definition \ref{def radius}. Because $\bar{w} \in V$, then from Lemma \ref{lm error bound R r} we have:
	\begin{equation}\label{eq lm c123 sharp LP 1}
		\frac{\| \bar{w} - \calF(\bar{w};w_\delta) \|}{\dist(\bar{w}, K)} \le \frac{\| w_\delta - \calF(\bar{w};w_\delta) \|}{r_\delta} \ .
	\end{equation}

	Also, because $\bar{w}$, $\calF(\bar{w};w_\delta)$, and $w_\delta$ are collinear, it follows that
	\begin{equation}\label{eq lm c123 sharp LP 2}
		\begin{aligned}
			\big| \gap(\bar{w}) - \gap(\calF(\bar{w};w_\delta)) \big| & = \big|
			\gap(\calF(\bar{w};w_\delta)) - \gap(w_\delta)
			\big|
			\cdot \frac{\| \bar{w} - \calF(\bar{w};w_\delta)\|}{\|\calF(\bar{w};w_\delta) - w_\delta\|} \\
			                                                          & \le  \big|
			\gap(\calF(\bar{w};w_\delta)) - \gap(w_\delta)
			\big|
			\cdot \frac{\dist(\bar{w},K)}{r_\delta} \le \delta \cdot  \frac{\dist(\bar{w},K)}{r_\delta} \ ,
		\end{aligned}
	\end{equation}
	where the first inequality uses \eqref{eq lm c123 sharp LP 1} and the second inequality follows since $w_\delta$ and $\calF(\bar{w};w_\delta)$ are in $\calW_\delta$.

	Next, because $\calF(\bar{w};w_\delta) \in \calW_\delta$, then from the definition of the PD sharpness $\mu$ in \eqref{eqdef:LPsharpness} and the lower bound on $\mu$ from Lemma \ref{lm:lowerbound_mu}, we have:
	\begin{equation}\label{eq lm c123 sharp LP 3}
		\begin{aligned}
			\dist(\calF(\bar{w};w_\delta),\calW^\star) & \ \le   \frac{\dist(\calF(\bar{w};w_\delta), V \cap \{w:\gap(w)  = 0\}         ) }{ \mu} \\
			                                           & \ =
			\frac{\gap(\calF(\bar{w};w_\delta))}{\sqrt{\|c\|^2 + \|q\|^2}\cdot \mu} \le \frac{ d^H_\delta   }{\delta} \cdot \gap(\calF(\bar{w};w_\delta))\ .
		\end{aligned}
	\end{equation}

	\noindent We also have the following bound on $\dist(\bar{w},\calW^\star)$:
	\begin{equation}\label{eq lm c123 sharp LP 4}
		\begin{aligned}
			\dist(\bar{w},\calW^\star) & \le \dist(\calF(\bar{w};w_\delta),\calW^\star)  + \| \calF(\bar{w};w_\delta) - \bar{w}\| \\  & \le  \frac{ d^H_\delta   }{\delta} \cdot \gap(\calF(\bar{w};w_\delta))  +\frac{D_\delta}{r_\delta} \cdot \dist(\bar{w},K) \ .
		\end{aligned}
	\end{equation}
	where the second inequality uses \eqref{eq lm c123 sharp LP 3} as well as the inequality $\|   \calF(\bar{w};w_\delta) - \bar{w}\| \le \frac{D_\delta}{r_\delta} \cdot \dist(\bar{w},K)$ which itself follows from Lemma \ref{lm error bound R r}.

	From \eqref{eq lm c123 sharp LP 2} we have
	\begin{equation}\label{eq lm c123 sharp LP 6}
		\begin{aligned}
			\gap(\calF(\bar{w};w_\delta)) & \ \le  \gap(\bar{w}) + \delta\cdot \frac{\dist(\bar{w},K)}{r_\delta} \ ,
		\end{aligned}
	\end{equation}
	whereby \eqref{eq lm c123 sharp LP 4} implies
	\begin{equation}\label{eq lm c123 sharp LP 7}
		\begin{aligned}
			\dist(\bar{w},\calW^\star) & \le  \frac{ d^H_\delta   }{\delta} \cdot \left( \gap(\bar{w}) + \delta\cdot \frac{\dist(\bar{w},K)}{r_\delta}    \right)  +\frac{D_\delta}{r_\delta} \cdot \dist(\bar{w},K) \\
			                           & \le \frac{ d^H_\delta   }{\delta} \cdot\gap(\bar{w}) +  \frac{2D_\delta}{r_\delta} \cdot \dist(\bar{w},K) \ ,
		\end{aligned}
	\end{equation}
	in which the second inequality uses $d^H_\delta \le D_\delta$ from Lemma \ref{lm D ge r}.

	Finally, we use the upper bound on $\dist(\bar{w},\calW^\star)$ to obtain an upper bound on $\dist(w,\calW^\star)$:
	\begin{equation}\label{eq lm c123 sharp LP 8}\begin{aligned}
			\dist(w,\calW^\star) & \le 	\dist(\bar{w},\calW^\star) + \| \bar{w} - w\|  = 	\dist(\bar{w},\calW^\star) +\dist(w, V) \\ & \le \frac{ d^H_\delta   }{\delta} \cdot\gap(w) + \frac{2D_\delta}{r_\delta} \cdot \dist(\bar w, K) + \dist(w, V)
			\\ & \le \frac{ d^H_\delta   }{\delta} \cdot\gap(w) + \frac{2D_\delta}{r_\delta} \cdot \left(\dist(w, K) + \dist(w,V) \right) + \dist(w, V)
			\\ & \le 	\frac{ d^H_\delta   }{\delta} \cdot\gap(w) +  \frac{5D_\delta}{r_\delta} \cdot \max\{ \dist(w,K) , \dist(w, V)\} \ ,
		\end{aligned}\end{equation}
	where the second inequality uses \eqref{eq lm c123 sharp LP 7} and $\gap(\bar{w}) = \gap(w)$, the third inequality uses $\dist(\bar{w},K) \le \dist(w,K) + \|w - \bar{w}\| = \dist(w,K)  + \dist(w,V)$, and the fourth inequality uses $D_\delta \ge r_\delta$ from Lemma \ref{lm D ge r}. This proves \eqref{eq lm c123 sharp LP} in the case when $\gap(w) \le \delta$.

	Let us now consider the case where  $\delta \le \gap(w)$.  Here we will make use of Proposition \ref{lm:convexity_ineq} to complete the proof.  Let $w^\star := P_{\calW^\star}(w) = \arg\min_{\bar{w}\in\calW^\star}\|\bar{w} - w\|$ and define $w_t := w^\star + t\cdot (w - w^\star)$ for $t\in[0,\infty)$. Then define the following functions of $t$ :
	$$
		f(t):= \dist(w_t,\calW^\star) \ , \text{ and }g(t):=\frac{ d^H_\delta   }{\delta} \cdot\gap(w_t) +  \frac{5D_\delta}{r_\delta} \cdot \max\{ \dist(w_t,K) , \dist(w_t, V)\}  \ .
	$$
	Then $f(t)$ is a nonnegative  linear  function on $[0,\infty)$, and $f(0) = 0$. And $g(t)$ is convex and nonnegative on $[0,\infty)$, and $g(0) = 0$. In addition, because $\gap(\cdot)$ is a linear function and $\gap(w_t) = t \cdot \gap(w)$, then setting $u := \delta / \gap(w)$ we obtain $ \gap(w_u) = u\cdot \gap(w) = \delta$. We can then invoke \eqref{eq lm c123 sharp LP 8} using $w_u$ in the place of $w$, which yields $g(u) \ge f(u)$. Now it follows from Proposition \ref{lm:convexity_ineq} with $v := 1\ge  u$ that  $g(1) \ge f(1)$, which is precisely \eqref{eq lm c123 sharp LP 8} in the case $\delta \le \gap(w)$, and completes the proof.
\end{proof}

\begin{lemma}\label{thm L sharp LP} Suppose that $c\in\operatorname{Null}(A)$. Under Assumption \ref{assump:striclyfeasible}, suppose that Algorithm \ref{alg: PDHG with restarts} (rPDHG) is run starting from $z^{0,0} = (x^{0,0},y^{0,0} ) = (0,0)$, and the step-sizes $\sigma$ and $\tau$ satisfy the step-size inequality \eqref{eq:general_stepsize}. Then for all $n\ge 1$ and any $\delta  \in (0,\bar{\delta}]$ it holds that
	\begin{equation}\label{eq thm L sharp LP}
		\dist_M(z^{n,0},\calZ^\star) \le    (5\sqrt{2} + 4) \cdot c_0^2 \cdot \frac{ D_\delta}{r_\delta}  \cdot \rho(\| z^{n,0} - z^{n-1,0}\|_M; z^{n,0}) \ .
	\end{equation}
\end{lemma}

\begin{proof}[Proof of Lemma \ref{thm L sharp LP}]
	Directly using Lemmas \ref{lm change of norm} and \ref{lm c123 sharp LP} we have
	\begin{equation}\label{eq L sharp LP 1}
		\begin{aligned}
			\dist_M(z^{n,0},\calZ^\star) & \le  \sqrt{2}c_0 \cdot	\dist(w^{n,0},\calW^\star)                                                                                                                           \\
			                             & \le  \	\frac{ \sqrt{2}c_0 d^H_\delta   }{\delta} \cdot\gap(w^{n,0}) +  \frac{5\sqrt{2}c_0  D_\delta}{r_\delta} \cdot \max\{ \dist(w^{n,0},K) , \dist(w^{n,0}, V)\} \ ,
		\end{aligned}
	\end{equation}
	and applying Lemma \ref{lm use gap to bound error in w} yields
	\begin{equation}\label{eq L sharp LP 2}			\begin{aligned}
			\dist_M(z^{n,0},\calZ^\star) \le \  & \left(  \frac{5\sqrt{2}c_0^2 D_\delta}{r_\delta}  + \frac{4 c_0^2 d^H_\delta   }{\delta} \cdot \dist(0,\calW^\star)    \right) \cdot \rho(\| z^{n,0} - z^{n-1,0}\|_M; z^{n,0}) \ .
		\end{aligned}
	\end{equation}
	From Lemma \ref{lm r_delta over delta} we have
	$r_\delta \cdot \dist(0,\calW^\star) \le r_\delta \cdot \max_{w\in\calW^\star}\|w\| \le \delta$, whereby it follows from \eqref{eq L sharp LP 2} that
	\begin{equation}\label{eq L sharp LP 3}
		\dist_M(z^{n,0},\calZ^\star)
		\le \left( \frac{5\sqrt{2} c_0^2 D_\delta}{r_\delta} +  \frac{4c_0^2 d^H_\delta}{r_\delta}
		\right)
		\cdot  \rho(\| z^{n,0} - z^{n-1}\|_M ; z^{n,0} ) \ .
	\end{equation}
	Last of all, notice that $d^H_\delta \le D_\delta$ (from Lemma \ref{lm D ge r}), and so \eqref{eq thm L sharp LP} follows from \eqref{eq L sharp LP 3}.
\end{proof}

We now prove Theorem \ref{thm overall complexity lp}. 

\begin{proof}[Proof of Theorem \ref{thm overall complexity lp}]
	From Lemma \ref{lm use gap to bound error in w} it follows that $w^{n,0}$ satisfies the $\eps$-tolerance requirement (Definition \ref{eq accuracy requirement clp}) if
	\begin{equation}\label{eq overall complexity lp 1}
		\rho(\| z^{n,0} - z^{n-1}\|_M ; z^{n,0} ) \le \min\left\{
		\frac{ \eps_{\mathrm{cons}}}{c_0},\
		\frac{ \eps_{\mathrm{gap}}}{2\sqrt{2} c_0 \cdot  \dist(0,\calW^\star) },\
		\frac{\eps_{\mathrm{obj}}}{14c_0\cdot\inf_{\gamma>0}\left(\frac{\gamma}{r_\gamma}\right)}
		\right\} \ .
	\end{equation}
	Now notice that the right-hand-side term is equal to $\frac{\err}{c_0}$ where $\err$ is defined in \eqref{def eq calN lp}.  Also, it follows from the choice of step-sizes in the theorem and the definition of $c_0$ in \eqref{icemelt} that $c_0 = \sqrt{\kappa}$.

	In the proof of Lemma \ref{thm L sharp LP} we see that \eqref{eq thm L sharp LP} holds for any $\delta \in (0,\bar\delta]$, so Theorem \ref{thm: complexity of PDHG with adaptive restart} can be applied since the condition \eqref{eq restart L C condition} is satisfied using \eqref{eq thm L sharp LP} with $\calL = (5\sqrt{2} + 4) \cdot c_0^2 \cdot \frac{ D_\delta}{r_\delta}$ and $\calC = 0$. Therefore it follows from Theorem \ref{thm: complexity of PDHG with adaptive restart} that $T$ satisfies
	\begin{equation}\label{eq overall complexity lp 2}
		T \le 23 \cdot \left((5\sqrt{2}+4)\cdot c_0^2\cdot \frac{D_\delta}{r_\delta} \right) \cdot  \ln\left(\frac{23 c_0\cdot  \dist_M(z^{0,0},\calZ^\star)}{\err}\right) \ .
	\end{equation}		Here we have $c_0^2 = \kappa$, and $\dist_M(z^{0,0},\calZ^\star) \le  \sqrt{2}c_0 \cdot \dist(w^{0,0},\calW^\star)\le  \sqrt{2}c_0 \cdot \dist(0,\calW^\star) = \sqrt{2\kappa} \cdot \dist(0,\calW^\star)$ from Lemma \ref{lm change of norm} and Proposition \ref{lm:distance_to_optimal_initial}.
	Additionally, $23 \cdot (5\sqrt{2}+4)  \le 255$ and $23\sqrt{2}\le 33$.     Therefore \eqref{eq overall complexity lp 2} yields \eqref{eq overall complexity LP}.
\end{proof}

\section{Proofs of Results in Section \ref{sec:geometric_enhancements}}\label{appendix:geometric_enhancements}

\subsection{Proof of Theorem \ref{thm:conditionnumber_after_transformation}}
We first recall definitions of a (logarithmically homogeneous) self-concordant barrier function $f$ and the complexity value $\vartheta_f$.

\begin{definition}[Sections 2.3.1 and 2.3.5 of \cite{renegar2001mathematical}]\label{def:SCbarrier}
	Let $f$ be a function defined on $\emph{\textsf{int}} K_p \subseteq \mathbb{R}^n $.
	The function $f$ is a self-concordant barrier function for $K_p$ if for all $x \in \emph{\textsf{int}} K_p$, the unit local norm ball $B_{x}(x,1)$ satisfies $B_{x}(x,1) \subset K_p$ and
	\begin{equation}\label{eq:sef:SCbarrier}
		1- \|y-x\|_x \le \frac{\|v\|_y}{\|v\|_x}\le \frac{1}{1 - \|y-x\|_x} \ \text{ for all }   y \in \emph{\textsf{int}}B_{x}(x,1) \text{ and } v\neq 0 \ ,
		\text{ and }
	\end{equation}
	\begin{equation}\label{eq:def:complexity_value}
		\vartheta_f:=\sup_{x\in \emph{\textsf{int}} K_p} \|\nabla f(x)\|_x^2 < \infty \ .
	\end{equation}
	Here $\vartheta_f$ denotes the complexity value of $f$. Additionally, $f$ is a logarithmically homogeneous self-concordant barrier function if for all $x\in\emph{\textsf{int}}K_p$ and all $t>0$ it holds that
	$f(tx) = f(x) - \vartheta_f \ln(t)$.
\end{definition}
\noindent
Logarithmically homogeneous self-concordant barriers have some very special properties, see \cite{nesterov1994interior,renegar2001mathematical}, among which is the following equality taken from \cite[Theorem 2.3.9]{renegar2001mathematical}:
\begin{equation}\label{patsday} \|x\|_x = \sqrt{\vartheta_f} \ \text{for all } \ x \in \textsf{int} K_p \ .
\end{equation}
(Any self-concordant barrier function for $K_p$ can be expanded to a logarithmically homogeneous self-concordant barrier function on $\{(x,t):t>0,x\in t \cdot K_p\}$, see \cite[Proposition 5.1.4]{nesterov1994interior}.)

For simplicity of notation, in this section we use $\tilde{w}(\eta):= \phi_\eta(w(\eta))$ to denote the rescaled central path solution for the parameter $\eta$.  We now prove Theorem \ref{thm:conditionnumber_after_transformation}.

\begin{proof}[Proof of Theorem \ref{thm:conditionnumber_after_transformation}]
	We first show that the local-norm ball at $w(\eta)$ maps to a Euclidean ball under the rescaling transformation $\phi_\eta$. For any $r>0$, the rescaling of the local-norm ball $\phi_\eta\left(
		B_{w(\eta)}(w(\eta),r)
		\right)$ satisfies	\begin{equation}\label{eq:local_to_Euclidean}
		\begin{aligned}
			\phi_\eta\left(
			B_{w(\eta)}(w(\eta),r)
			\right) = & \left\{ \tfrac{1}{\sqrt{\eta}} H_{w(\eta)}^{1/2}\cdot \hat{w}: \sqrt{(\hat{w}-w(\eta))^\top H_{w(\eta)} (\hat{w}-w(\eta))} \le r\right\} \\
			=         &
			\left\{v: \left\|v -\tfrac{1}{\sqrt{\eta}} H_{w(\eta)}^{1/2}\cdot w(\eta)\right\| \le\tfrac{r}{\sqrt{\eta}} \right\} = B\left(\tilde{w}(\eta),\tfrac{r}{\sqrt{\eta}}\right) \ ,
		\end{aligned}
	\end{equation}
	where the second equality uses the substitution $\hat{w} = \sqrt{\eta} H_{w(\eta)}^{-1/2}\cdot v$.

	Let us first prove \eqref{eq:tilde_D}. From item ({\it\ref{item:fact:centralpath_main:3}.}) of Fact \ref{fact:centralpath_main} we have
	$\calW_{\alpha} \subseteq	B_{w(\eta)}(w(\eta),\vartheta_F + 2\sqrt{\vartheta_F})$.	From \eqref{eq:local_to_Euclidean} we have
	\begin{equation}\label{eq:tw_in_tB}
		\begin{aligned}
			 & \tilde{\calW}_{\alpha}  \subseteq  \	\phi_\eta\left(B_{w(\eta)}(w(\eta),\vartheta_F + 2\sqrt{\vartheta_F})\right) = B\left(\tilde{w}(\eta),\tfrac{\vartheta_F + 2\sqrt{\vartheta_F}}{\sqrt{\eta}}\right) \ ,
		\end{aligned}
	\end{equation}
	and therefore $\tilde{D}_\alpha \le\frac{2\vartheta_F +4\sqrt{\vartheta_F}}{\sqrt{\eta}}$, which yields \eqref{eq:tilde_D} since logarithmic homogeneity implies $\vartheta_f = \vartheta_{f^*}$ and consequently $\vartheta_F = \vartheta_f + \vartheta_{f^*} = 2 \vartheta_f$.

	Next observe that \eqref{eq:tilde_d} directly follows from \eqref{eq:tilde_D} by using Lemma \ref{lm D ge r}.

	Regarding the lower bound on $\tilde{r}_{\alpha}$ in \eqref{eq:tilde_r}, we note that $ K \supseteq  B_{w(\eta)}(w(\eta),1)$ because $F$ is a self-concordant barrier function on $K$. Therefore
	\begin{equation}\label{eq:thm:conditionnumber_after_transformation_2}
		\begin{aligned}
			\tilde{K} = \phi_\eta (K)  \supseteq  \phi_\eta\left(B_{w(\eta)}(w(\eta),1)\right)  =
			B\left(\tilde{w}(\eta),\tfrac{1}{\sqrt{\eta}}\right) \ ,
		\end{aligned}
	\end{equation}where the last equality uses \eqref{eq:local_to_Euclidean}.
	And since $\tilde{w}(\eta)\in \tilde{\calW}_{\alpha}$, \eqref{eq:thm:conditionnumber_after_transformation_2} implies that $\tilde{r}_{\alpha} \ge \tfrac{1}{\sqrt{\eta}}$, which proves \eqref{eq:tilde_r}.

	Concerning the proof of \eqref{eq:tilde_dist}, we have the following expression for $\|\tilde{w}(\eta)\|$:
	\begin{equation}\label{eq:norm_of_tildew}
		\begin{aligned}
			\|\tilde{w}(\eta)\| & = \left\|\tfrac{1}{\sqrt{\eta}} H_{w(\eta)}^{1/2} \cdot w(\eta)\right\| =  \tfrac{1}{\sqrt{\eta}}\cdot\|w(\eta)\|_{w(\eta)} = \sqrt{\tfrac{\vartheta_F}{\eta}} \ ,
		\end{aligned}
	\end{equation}
	where the last equality uses \eqref{patsday}. Hence
	\begin{equation}\label{eq:dist_of_tildew}
		\dist(0,\tilde{\calW}^\star) \le \|\tilde{w}(\eta)\| + \dist(\tilde{w}(\eta),\tilde{\calW}^\star	) \le \sqrt{\tfrac{\vartheta_F}{\eta}} + \tfrac{\vartheta_F +2\sqrt{\vartheta_F}}{\sqrt{\eta}}  =  \tfrac{\vartheta_F +3\sqrt{\vartheta_F}}{\sqrt{\eta}}  = \tfrac{2\vartheta_f +3\sqrt{2\vartheta_f}}{\sqrt{\eta}} \ ,
	\end{equation}
	where the second inequality uses \eqref{eq:norm_of_tildew} and \eqref{eq:tw_in_tB}, and the last equality uses $\vartheta_F = \vartheta_f + \vartheta_{f^*} = 2 \vartheta_f$.
	Last of all, \eqref{eq:ratioratio} follows immediately from \eqref{eq:tilde_D} and \eqref{eq:tilde_r}.
\end{proof}

\subsection{Proof of Proposition \ref{lm:feature_of_log_hessian}}
The proof actually is nothing more than manipulation of basic barrier calculus. We first state the following properties of self-concordant barriers as presented in \cite{renegar2001mathematical}.
\begin{fact}\label{fact:SCfunction}
	The following statements hold for the self-concordant barrier function $f$ for $K_p$:
	\begin{enumerate}
		\item For all $x \in K_p$ it holds that $-\nabla f(x) \in K_p^*$.  Furthermore $\emph{\textsf{int}}K_p^* = \{-\nabla f(x):x\in\emph{\textsf{int}}K_p\}$ \cite[Proposition 3.3.3]{renegar2001mathematical}. For any $x\in\emph{\textsf{int}} K_p$, let $s = -\nabla f(x)$, then   $\nabla^2 f(x) = (\nabla^2 f^*(s))^{-1}$ \cite[Theorem 3.3.4]{renegar2001mathematical}.\label{item:fact:SCfunction:2}
		\item If $f$ is logarithmically homogeneous, then for any $t>0$, $\nabla^2 f(tx) = \frac{1}{t^2}\cdot \nabla^2 f(x)$ \cite[Theorem 2.3.9]{renegar2001mathematical}. \label{item:fact:SCfunction:3}
		\item  Let $w(\eta)=(x(\eta),s(\eta))$ be the optimal solution of \eqref{pro: general primaldual clp centralpath}. If $f$ is logarithmically homogeneous, then  $s(\eta) = -\frac{1}{\eta}\nabla f(x(\eta))$ \cite[Section 3.4]{renegar2001mathematical}. \label{item:fact:SCfunction:4}
	\end{enumerate}
\end{fact}
With these facts in mind, we can now easily prove the proposition.
\begin{proof}[Proof of Proposition \ref{lm:feature_of_log_hessian}]
	It follows from the definition of $\phi_\eta$ in \eqref{eq:def phi} and definition of $D_1$ in the proposition that
	$\phi_\eta(w)= ( D_1^{-1} x , \tfrac{1}{\sqrt{\eta}} \cdot H_{s(\eta)}^{1/2} s ) $, so it suffices to prove that $D_1 = \tfrac{1}{\sqrt{\eta}} \cdot H_{s(\eta)}^{1/2}$.  We have
	\begin{equation}\label{peaky} H_{s(\eta)} := \nabla^2 f^*(s(\eta)) = \nabla^2 f^*(-\tfrac{1}{\eta}\nabla f(x(\eta))) = \eta^2 \nabla^2 f^*(-\nabla f(x(\eta))) = \eta^2 [\nabla^2 f(x(\eta))]^{-1} \ , \end{equation}
	where the second equality uses item (\ref{item:fact:SCfunction:4}.) of Fact \ref{fact:SCfunction}, the third equality uses item (\ref{item:fact:SCfunction:3}.) of Fact \ref{fact:SCfunction}, and the fourth equality uses item (\ref{item:fact:SCfunction:2}.) of Fact \ref{fact:SCfunction}.  Last of all, it follows from \eqref{peaky} that
	$$H_{s(\eta)} = \eta^2 [H_{x(\eta)}]^{-1} = \eta D_1^2  \ , $$ which provides the proof after rearranging and taking square roots. \end{proof}

\subsection{Proof of Theorem \ref{thm overall complexity rescaled}}
The proof of Theorem \ref{thm overall complexity rescaled} is more involved than simply applying the complexity guarantees of Theorems \ref{thm overall complexity clp} and \ref{thm overall complexity lp} due to the fact that the $\eps$-tolerance condition for the rescaled problem \eqref{pro:rescaled problem} does not correspond to the $\eps$-tolerance condition of the original problem \eqref{pro: general primal clp}.

We begin with the following affine invariance properties for $w\in V$.

\begin{proposition}\label{lm:gap_invariant}
	Let $\widetilde{\gap}(\cdot)$ and $\widetilde{\eobj}(\cdot)$ denote the duality gap and the objective function error of the rescaled problem \eqref{pro:rescaled problem}. Then for any $w \in V$ and $\eta > 0$, it holds that  $
		\gap(w) = \widetilde{\gap}(\phi_\eta(w))$ and $\eobj(w) = \widetilde{\eobj}(\phi_{\eta}(w))$.
\end{proposition}
\begin{proof}
	In the proof we will use $\tilde{w} = (\tilde{x},\tilde{s})$ to denote $\phi_\eta(w)$. Since $\phi_\eta(w)\in\phi_\eta(V) = \tilde{V}$ and the rescaled affine subspace $\tilde{V}$ is $\{x: AD_1 x = b\}\times \{s: D_1^\top A^\top y + s = D_1^\top c\}$, there exists $\tilde{y}\in \mathbb{R}^m$ such that $(D_2 A D_1)^\top \tilde{y} + \tilde{s} = D_1^\top c$, which means  $\widetilde{\gap}(\tilde{w})=(D_1^\top c)^\top \tilde{x} - (D_2 b)^\top \tilde{y}$.

	Notice that from Lemma \ref{lm:feature_of_log_hessian}, $\tilde{s} = D_1 s$ so $(D_2 A D_1)^\top \tilde{y} + D_1 s = D_1^\top c$ and thus $A^\top (D_2^\top \tilde{y}) + s = c$. Therefore, $\gap(w)$ can be written as $c^\top x - b^\top (D_2^\top \tilde{y})$.

	According to Lemma \ref{lm:feature_of_log_hessian}, $\tilde{x} = D_1^{-1} x$, so substituting this into the expressions of $\widetilde{\gap}(\tilde{w})$ and $\gap(w)$ derived above yields $\widetilde{\gap}(\tilde{w}) = \gap(w)$.

	Let $\widetilde{\eobj}(\cdot)$ denote the objective error of the rescaled problem. Because the duality gap is invariant under the rescaling, and $\widetilde{\eobj}(\tilde{w} ) = |\widetilde{\gap}(\tilde{x},\tilde{s}^\star)| + |\widetilde{\gap}(\tilde{x}^\star,\tilde{s})|$ for $(\tilde{x}^\star,\tilde{s}^\star)\in\calW^\star$, the objective error also remains invariant for $w\in V$, namely $\eobj(w) = \widetilde{\eobj}(\phi_{\eta}(w))$.
\end{proof}

Notice that standard measures of feasibility error such as the distances to $K$ and $V$, are not generically invariant under the rescaling transformation.  We therefore are led to consider the following new measure of feasibility error, as a means to the end of proving Theorem \ref{thm overall complexity rescaled}. For any $\delta > 0$ and $w \in V$, the relative feasibility ratio on $\eqref{pro: general primal clp}$ is defined as
\begin{equation}\label{eq def relative feasibility error}
	\efeas(w;\delta) := \min_{\bar{w} \in \calW_\delta} \frac{\| \calF(w;\bar{w})  - w \|  }{\| \calF(w;\bar{w})  - \bar{w} \|} \
\end{equation}
where recall the definition of $\calF(w,\bar{w})$ in \eqref{eq  def v}. The relative feasibility error of $w$ with respect to $\calW_\delta$ can be interpreted as a measure of the extent to which $\calW_\delta$ can be expanded, centered at a specific inner point $\bar{w}$, such that the expanded $\calW_\delta$ encompasses $w$. Similarly, the relative feasibility error of the rescaled problem for $\tilde{w}\in\tilde{V}$ is denoted by $\widetilde{\efeas}(\tilde{w};\delta) := \min_{\hat{w} \in \tilde{\calW}_\delta} \frac{\| \tilde{\calF}(\tilde{w};\hat{w})  - \tilde{w} \|  }{\| \tilde{\calF}(\tilde{w};\hat{w})  - \hat{w} \|}$.
Note that
$$
	\frac{\| \calF(w;\bar{w})  - w \|  }{\| \calF(w;\bar{w})  - \bar{w} \|}  = \frac{\|\phi_\eta(\calF(w;\bar{w}))  - \phi_\eta(w) \|  }{\| \phi_\eta(\calF(w;\bar{w}))  - \phi_\eta(\bar{w}) \|}
	= \frac{\|\tilde{\calF}(\phi_\eta(w);\phi_\eta(\bar{w}))  - \phi_\eta(w) \|  }{\| \tilde{\calF}(\phi_\eta(w);\phi_\eta(\bar{w}))  - \phi_\eta(\bar{w}) \|}
$$
holds for any $w\in V$ and $\bar{w}\in\calF$, because $w$, $\calF(w;\bar{w})$, and $\bar{w}$ are collinear. Therefore the relative feasibility ratio is invariant under $\phi_\eta$, i.e., $\efeas(w;\delta) = \widetilde{\efeas}(\phi_\eta(w);\delta)$ for any $w\in V$ and $\eta > 0 $. 

The following lemma shows that the relative feasibility error can be bounded from above and below by the distance to $K$ up to certain factors involving $D_\delta$ and $r_\delta$.
\begin{lemma}\label{lm:distance_to_cone_bounded}
	Let $w$ be any point in $V$. For any $\delta \ge \gap(w)$, the following inequalities hold:
	\begin{equation}\label{eq:lm:distance_to_cone_bounded}
		\frac{\dist(w,\calF)}{D_\delta}\le \efeas(w;\delta)\le \frac{\dist(w,K)}{r_\delta} \ .
	\end{equation}
\end{lemma}
\noindent
Note from \eqref{eq:lm:distance_to_cone_bounded} that a small value of $\efeas(w;\delta)$ implies a small value of $\dist(w,\calF)$, which then implies a small value of $\dist(w, K)$ since $K \subseteq \calF$.
\begin{proof}
	To prove the first inequality, let $\bar{w}\in \calW_\delta$ be the point that achieves the minimum in \eqref{eq def relative feasibility error}, so that $\efeas(w;\delta) = \frac{\| \calF(w;\bar{w})  - w \|  }{\| \calF(w;\bar{w})  - \bar{w} \|}$. Then we have
	$$\dist(w,\calF) \le \| \calF(w;\bar{w})  - w \| = \| \calF(w;\bar{w})  - \bar{w} \|\cdot \efeas(w;\delta)  \le D_\delta \cdot \efeas(w;\delta) \ , $$
	where the last inequality follows from the fact that both $\calF(w;\bar{w})$ and $\bar{w}$ belong to $\calW_\delta$.

	To prove the second inequality, let $w_\delta$ be the conic center of  $\calW_\delta$ (defined in Definition \ref{def radius}). By the definition of the relative feasibility error, $\frac{	\| \calF(w;w_\delta) - w \|}{\| \calF(w;w_\delta)- w_\delta\| } \ge \efeas(w;\delta)$. Additionally, from Lemma \ref{lm error bound R r} we have $\frac{	\| \calF(w;w_\delta) - w \|}{\| \calF(w;w_\delta)- w_\delta\| } \le \frac{\dist(w,K)}{r_\delta}$, and combining these two inequalities furnishes the proof of the second inequality in \eqref{eq:lm:distance_to_cone_bounded}.
\end{proof}

The following lemma provides criteria for a candidate solution to the rescaled problem to be transformed back to a suitably nearly-optimal solution to the original problem.

\begin{lemma}\label{lm:acc_to_get}
	Suppose that $\alpha:=\gap(w(\eta)) \ge \min\{\eps_{\mathrm{obj}}, \eps_{\mathrm{gap}} \}$. Let $\tilde{w}=(\tilde{x},\tilde{s})$ satisfy $\tilde{x}\in \tilde{K}_p$ and $\tilde{s}\in \tilde{V}_d$, and also
	\begin{equation}\label{eq:lm:acc_to_get:1}
		\max\{\dist(\tilde{w},\tilde{V}),\dist(\tilde{w},\tilde{K})\}\le \frac{\eps_{\mathrm{cons}}\cdot \tilde{r}_\alpha}{\sqrt{2}\cdot D_{\alpha}} , \ \text{ and } \  \widetilde{\eobj}(\tilde{w})\le \min\{\eps_{\mathrm{obj}}, \eps_{\mathrm{gap}} \} \
	\end{equation}
	for the rescaled problems \eqref{pro:rescaled problem} and \eqref{pro:dual rescaled problem}.  Then $w:=\phi_\eta^{-1}(P_{\tilde{V}}(\tilde{w}))$ satisfies
	\begin{equation}\label{eq:lm:acc_to_get:2}
		\dist(w,\calF)\le \eps_{\mathrm{cons}}  \ \ \text{ and } \ \eobj(w)\le \min\{\eps_{\mathrm{obj}}, \eps_{\mathrm{gap}}\}  \ .
	\end{equation}
\end{lemma}
\begin{proof}
	To ease the notational burden let $\hat{w} :=P_{\tilde{V}}(\tilde{w}) = (P_{\tilde{V}_p}(\tilde{x}),\tilde{s})$. Then $w = \phi_\eta^{-1}(\hat{w})\in V$, and since $D_2AD_1 \bar{c} = 0$, we have $\widetilde{\gap}(\tilde{w}) = \widetilde{\gap}(\hat{w})$ and $\widetilde{\eobj}(\tilde{w}) = \widetilde{\eobj}(\hat{w})$. Furthermore, since the duality gap and objective error are both invariant under the rescaling $\phi_\eta$, it follows that
	$$
		\begin{array}{c}
			\widetilde{\gap}(\tilde{w})  = \widetilde{\gap}(\hat{w}) = \gap(w) \text{ and } 	\widetilde{\eobj}(\tilde{w}) = 	\widetilde{\eobj}(\hat{w})=\eobj(w)\ ,
		\end{array}
	$$
	and therefore $\eobj(w) = \widetilde{\eobj}(\tilde{w})  \le \min\{\eps_{\mathrm{obj}}, \eps_{\mathrm{gap}}\}$. This shows the second inequality in \eqref{eq:lm:acc_to_get:2}.

	Again because the relative feasibility error is invariant under the rescaling, we have $\widetilde{\efeas}(\hat{w};\alpha) = \efeas(w;\alpha)$. Furthermore, since $\gap(w) = \widetilde{\gap}(\tilde{w})\le \widetilde{\eobj}(\tilde{w})  \le \min\{\eps_{\mathrm{obj}}, \eps_{\mathrm{gap}} \} \le \alpha$, it follows from Lemma \ref{lm:distance_to_cone_bounded} that
	\begin{equation}\label{eq:acc_to_get:1}
		\dist(w,\calF)\le D_{\alpha}\cdot \efeas(w;\alpha) = D_{\alpha}\cdot \widetilde{\efeas}(\hat{w};\alpha) \ .
	\end{equation}
	Note that $\dist(\hat{w},\tilde{K}) \le \|(\hat{x} - \tilde{x}, \hat{s} - P_{\tilde{K}_d}(\hat{s})) \| = \|(\hat{x} - \tilde{x}, \tilde{s} - P_{\tilde{K}_d}(\tilde{s})) \|$ and $\|\hat{x} - \tilde{x}\| = \dist(\tilde{x},\tilde{V}_p) = \dist(\tilde{w},\tilde{V})$, and furthermore $\|\tilde{s} - P_{\tilde{K}_d}(\tilde{s})\| = \dist(\tilde{s},\tilde{K}_d) = \dist(\tilde{w},\tilde{K})$. It then follows that $\dist(\hat{w},\tilde{K}) \le  \sqrt{2}\cdot \max\{\dist(\tilde{w},\tilde{V}),\dist(\tilde{w},\tilde{K})\} \le \frac{\eps_{\mathrm{cons}}\cdot \tilde{r}_\alpha}{D_{\alpha}}$. And since $\widetilde{\gap}(\hat{w}) \le \alpha$, it follows using Lemma \ref{lm:distance_to_cone_bounded} that $\widetilde{\efeas}(\hat{w};\alpha) \le \frac{\eps_{\mathrm{cons}}}{D_{\alpha}}$. Finally, substituting this inequality into \eqref{eq:acc_to_get:1} yields $\dist(w,\calF) \le \eps_{\mathrm{cons}}$, which completes the proof.
\end{proof}
 
We now present the proof of Theorem \ref{thm overall complexity rescaled}.  We remark that the iteration bound in the theorem yields a solution $w^{n,0}$ satisfying \eqref{eq:lm:acc_to_get:2}, which is a stricter requirement (and might be significantly stricter) than the $\eps$-tolerance requirement.

\begin{proof}[Proof of Theorem \ref{thm overall complexity rescaled}]
	We first consider the general case of CLP. From Lemma \ref{lm:acc_to_get}, once $\tilde{w}^{n,0}$ satisfies \eqref{eq:lm:acc_to_get:1}, then $w^{n,0}$ satisfies \eqref{eq:lm:acc_to_get:2} and hence satisfies the $\eps$-tolerance requirement, since it always holds that $\gap(w) \le \eobj(w)$ for any $w$ and hence $\gap(w^{n,0}) \le \eobj(w^{n,0}) \le \min\{\eps_{\mathrm{gap}},\eps_{\mathrm{obj}}\} \le \eps_{\mathrm{gap}}$. Let us directly apply Theorem \ref{thm overall complexity clp} to the rescaled problem \eqref{pro:rescaled problem}. Then the number of \textsc{OnePDHG} iterations required in order to obtain a solution $\tilde{w}^{n,0}$  satisfying \eqref{eq:lm:acc_to_get:1} is at most
	\begin{equation}\label{eq thm overall complexity rescaled 1}
		\begin{array}{c}
			T \le  T_\alpha := 190\tilde{\kappa} \cdot \frac{\tilde{D}_\alpha}{\tilde{r}_\alpha}  \cdot \left[\ln\left(33\tilde{\kappa}\cdot \dist(0,\tilde{\calW}^\star) \right) + \ln\left(\frac{1}{\widetilde{\err}}\right) \, \right]
			+ \frac{50\tilde{\kappa} \cdot \tilde{d}^H_\delta}{\widetilde{\err}}	 \ ,
		\end{array}
	\end{equation}
	where $$\widetilde{\err} :=
		\min\left\{\frac{\eps_{\mathrm{cons}}\cdot \tilde{r}_\alpha}{\sqrt{2}\cdot D_{\alpha}} \ , \ +\infty \ , \
		\tfrac{1}{14}  \left( \displaystyle\sup_{\gamma  > 0} \tfrac{ \tilde{r}_\gamma}{\gamma} \right) \cdot \min\{\eps_{\mathrm{gap}},\eps_{\mathrm{obj}}\}
		\right\} \ , $$ where the middle term above is $+\infty$ because \eqref{eq:lm:acc_to_get:1} does not directly require a bound  on $\widetilde{\gap}(\tilde{w})$.
	Note that Theorem \ref{thm:conditionnumber_after_transformation} implies the following inequalities:
	\begin{equation}\label{eq thm overall complexity rescaled 3}
		\begin{array}{c}
			\frac{\tilde{D}_\alpha}{\tilde{r}_\alpha}\le 4\vartheta_f + 4 \sqrt{2\vartheta_f} \ , \  \dist(0,\tilde{\calW}^\star)  \le \frac{2\vartheta_f + 3 \sqrt{2\vartheta_f}}{\sqrt{\eta}} \ , \ \tilde{r}_{\alpha} \ge \frac{1}{\sqrt{\eta}} \ , \ \text{and }  \ \tilde{d}^H_{\alpha}  \le \frac{4\vartheta_f + 4 \sqrt{2\vartheta_f}}{\sqrt{\eta}} \ ,
		\end{array}
	\end{equation}
	which also implies
	\begin{equation}\label{eq thm overall complexity rescaled 3-1}
		\begin{array}{c}
			\displaystyle\sup_{\gamma  > 0} \frac{ \tilde{r}_\gamma}{\gamma} \ge \frac{ \tilde{r}_\alpha}{\alpha} \ge \frac{1}{\sqrt{\eta}\alpha} \ ,
		\end{array}
	\end{equation}
	and hence the following bound on $T$:
	\begin{equation}\label{eq thm overall complexity rescaled 4}
		\begin{array}{c}
			T \le  \tilde{\kappa} \cdot (\vartheta_f + \sqrt{2\vartheta_f})  \cdot \left(760\left[\ln\left(33\tilde{\kappa}\cdot \left(2\vartheta_f + 3 \sqrt{2\vartheta_f}\right) \right) + \ln\left(\frac{1}{\sqrt{\eta}\cdot \widetilde{\err}}\right) \, \right]
			+ \frac{200 }{\sqrt{\eta} \cdot \widetilde{\err}} \right)\ ,
		\end{array}
	\end{equation}
	and the following lower bound on $\sqrt{\eta}\cdot \widetilde{\err}$:
	\begin{equation}\label{eq thm overall complexity rescaled 5}
		\sqrt{\eta}\cdot \widetilde{\err}\ge \min\left\{\frac{\eps_{\mathrm{cons}}}{\sqrt{2} \cdot D_{\alpha}} , \
		\frac{\eps_{\mathrm{gap}}}{14\alpha}  ,\  \frac{\eps_{\mathrm{obj}}}{14\alpha}
		\right\} \ .
	\end{equation}
	Denoting the right-hand side of \eqref{eq thm overall complexity rescaled 5} by $\err^{\alpha}$, then the bound  \eqref{eq overall complexity rescaled clp} follows directly from \eqref{eq thm overall complexity rescaled 4} and
	\eqref{eq thm overall complexity rescaled 5}.

	Let us now consider the case where $K_p=\mathbb{R}^n_+$. Since the duality gap is invariant under the rescaling transformation, $\bar{\delta}$ is the best suboptimal extreme point gap for both \eqref{pro: general primal clp} and \eqref{pro:rescaled problem}, and so from Theorem \ref{thm overall complexity lp} we have
	\begin{equation}\label{eq thm overall complexity rescaled 6}
		\begin{array}{c}
			T \le 255 \kappa \cdot \left( \min_{0 < \delta \le \bar{\delta}}\frac{\tilde{D}_{\delta}}{\tilde{r}_{\delta}}\right) \cdot   \left[\ln\big(33 \tilde{\kappa}\cdot \dist(0,\tilde{\calW}^\star) \big) + \ln\left(\frac{1}{\widetilde{\err}}\right) \, \right] \ .
		\end{array}
	\end{equation}
	For $\alpha \ge \bar{\delta}$ we have $\tilde{D}_{\bar{\delta}}\le \tilde{D}_\alpha$ and $\tilde{r}_{\bar{\delta}}\ge \frac{\bar{\delta}}{\alpha}\cdot \tilde{r}_{\alpha}$ (using Lemma \ref{lm:monotonicity of delta r}), and hence $\min_{0 < \delta \le \bar{\delta}}\frac{\tilde{D}_{\delta}}{\tilde{r}_{\delta}} \le \frac{\tilde{D}_{\bar{\delta}}}{\tilde{r}_{\bar{\delta}}} \le \frac{\alpha}{\bar{\delta}}\cdot \frac{\tilde{D}_\alpha}{\tilde{r}_\alpha}$.
	And for $\alpha < \bar{\delta}$ we have $\min_{0 < \delta \le \bar{\delta}}\frac{\tilde{D}_{\delta}}{\tilde{r}_{\delta}} \le \frac{\tilde{D}_\alpha}{\tilde{r}_\alpha}$. Therefore,
	$$
		\begin{array}{c}
			\min_{0 < \delta \le \bar{\delta}}\frac{\tilde{D}_{\delta}}{\tilde{r}_{\delta}} \le \max\left\{1,\frac{\alpha}{\bar{\delta}}\right\}\cdot \frac{\tilde{D}_\alpha}{\tilde{r}_\alpha} \le 4\cdot \max\left\{1,\frac{\alpha}{\bar{\delta}}\right\}\cdot \left(\vartheta_f + \sqrt{2\vartheta_f}\right) \
		\end{array}
	$$
	where the last inequality uses \eqref{eq thm overall complexity rescaled 3}. Finally, notice that $\vartheta_f = n$ for the logarithmic barrier function $f(x) = -\sum_{j=1}^n x_j$ for $K_p=\mathbb{R}^n_+$. Applying \eqref{eq thm overall complexity rescaled 3} and \eqref{eq thm overall complexity rescaled 5} to \eqref{eq thm overall complexity rescaled 6} yields \eqref{eq overall complexity rescaled lp}.
\end{proof}

\section{Further details of \cpipm} \label{clifficlerun}

In this section we discuss some of the ways in which \cpipm \ either is identical to or is different from the practical IPM in \cite[Section 14.2]{nocedal2006numerical}. 

First, \cpipm \ runs on the rescaled problem \eqref{pro:rescaled problem}, where $D_1$ and $D_2$ are diagonal rescaling matrices derived from 10 iterations of Ruiz rescaling, as opposed to the original problem. Upon termination of \cpipm, solutions are converted back to the original problem.

Second, \cpipm \  incorporates all of the heuristics in \cite[Section 14.2]{nocedal2006numerical} that do not require solving linear systems. These heuristics include the selection of the centering parameter, the choice of primal and dual step lengths, and the starting point selection. The practical IPM in \cite[Section 14.2]{nocedal2006numerical} also uses a step-length parameter $\eta_k$ (using the notation of \cite{nocedal2006numerical} which is very different from our notation in Definition \ref{italy} which refers to the barrier parameter).  In \cpipm \ we set $\eta_k :=0.9$ . The starting point for \cpipm \ involves computing projections of the primal and dual zero vectors onto $\tilde{V}_p$ and $\tilde{V}_d$, and to do so  \cpipm \  uses $1,000$ iterations of CGM to approximately compute these projections.

Lastly, each iteration of \cpipm \ solves two linear systems, one for the predictor step and the other for the corrector step. These two linear systems, which are equations (14.30) and (14.35) in \cite{nocedal2006numerical}, share the same coefficient matrix and can be described as follows:
\begin{equation}\label{eq:coefficient matrix}
	\begin{pmatrix}
		0 & A^T & I \\
		A & 0   & 0 \\
		S & 0   & X
	\end{pmatrix}\begin{pmatrix}
		\Delta x \\
		\Delta y \\
		\Delta s
	\end{pmatrix}=\begin{pmatrix}
		-r_c \\
		-r_b \\
		-r_{x s}
	\end{pmatrix} \ , 
\end{equation}
where, $r_c$, $r_b$, and $r_{xs}$ are different for the two systems. We denote the coefficient matrix and the right-hand side vector of equation \eqref{eq:coefficient matrix} as $Q$ and $q$, respectively. Solving equation \eqref{eq:coefficient matrix} can be simplified to solving the following normal equation for $\Delta y$:
\begin{equation}\label{eq:normal matrix}
	A D^2 A^\top \Delta y = -r_b - AXS^{-1}r_c + AS^{-1}r_{xs}
\end{equation}
where $D$ is the diagonal matrix $S^{-1/2}X^{1/2}$. Once $\Delta y$ is computed, then $\Delta s$ and $\Delta x$ can be computed using $\Delta s  =-r_c-A^T \Delta y$ and $\Delta x =-S^{-1} r_{x s}-X S^{-1} \Delta s$. The matrix $AD^2A^\top$ in \eqref{eq:normal matrix} is always positive semidefinite so the normal equation \eqref{eq:normal matrix} is solved by the Jacobi preconditioned CGM \cite{barrett1994templates}, which is equivalent to the regular CGM for the linear system with $M^{-1}$ multiplied on both sides of \eqref{eq:normal matrix}, where $M$ is the diagonal matrix composed of the diagonal entries of $AD^2A^\top$. The CGM stops when it reaches either (i) $m$ iterations (the number of rows of $A$), or (ii) when the recovered solution $\left(\begin{smallmatrix}
			\Delta x \\
			\Delta y \\
			\Delta s
		\end{smallmatrix}\right)$ satisfies $\left\|Q \left(\begin{smallmatrix}
			\Delta x \\
			\Delta y \\
			\Delta s
		\end{smallmatrix}\right) - q\right\|\le \frac{0.1}{\sqrt{k}}\cdot \|q\|$ at iteration $k$ of the interior-point method.


\section*{Declarations} This work was supported by AFOSR Grant No. FA9550-22-1-0356. The authors declare that they have no relevant financial or non-financial interests to disclose.

\bibliographystyle{abbrv}
\bibliography{reference}

\end{document}